\documentclass[leqno,a4paper,12pt]{article}
\usepackage[english]{babel}
\usepackage{lmodern}
\usepackage[utf8x]{inputenc} 
\usepackage[T1]{fontenc} 
\usepackage{hyperref} 
\usepackage{fancyhdr} 
\usepackage{pgf,tikz} 
\usetikzlibrary{shapes,arrows,chains}
\usetikzlibrary[calc]

\usepackage[textwidth=0.8\paperwidth, textheight=0.8\paperheight]{geometry}

\usepackage{amsmath, amsfonts, amssymb, amsthm}

\usepackage{graphicx,color,epsfig}
\usepackage{euscript}
\usepackage{mathrsfs}
\usepackage{float} 

\usepackage{stmaryrd} 
\newcommand{\brint}[1]{{\llbracket #1 \rrbracket}} 
\newcommand{\ebrint}[1]{{\rrbracket #1 \llbracket}} 

\renewcommand{\labelitemi}{$\bullet$}

\theoremstyle{definition}
\newtheorem{definition}{Definition}

\theoremstyle{remark}
\newtheorem{remark}{Remark}

\theoremstyle{plain}
\newtheorem{theorem}{Theorem}
\newtheorem{lemma}{Lemma}
\newtheorem{proposition}{Proposition}

\makeatletter
\renewcommand\theequation{\thesection.\arabic{equation}}
\@addtoreset{equation}{section}
\makeatother

\newcommand\R{{\mathbb R}} 
\newcommand\Z{{\mathbb Z}} 
\newcommand\N{{\mathbb N}} 
\newcommand\Ns{{\mathbb N}^*} 

\renewcommand\P{{\bf P}} 
\newcommand\Ph{{\widehat {\bf P}}} 
\newcommand\Pb{{\mathbb{P}}} 
\newcommand\E{{\bf E}} 
\newcommand\Eb{{\mathbb{E}}} 
\newcommand\Eh{{\widehat{\bf E}}} 
\newcommand\Pbh{{\widehat{\mathbb{P}}}} 
\newcommand\Ebh{{\widehat{\mathbb{E}}}} 

\newcommand\iid{{i.i.d.\ }} 
\newcommand\as{{a.s.\ }} 
\newcommand\ie{{i.e.\ }} 

\renewcommand\root{{\rho}} 
\newcommand\Tb{{\mathbf T}} 
\newcommand\Fb{{\mathbf F}} 
\newcommand\Tr{{\mathcal{T}}} 
\newcommand\V{{\mathbb{V}}} 
\newcommand\W{{\mathbb{W}}} 
\newcommand\T{{\mathbb T}} 
\newcommand\F{{\mathbb F}} 
\renewcommand\L{{\cal{L}}} 
\newcommand\B{{\cal{B}}}
\newcommand\tauh{{\widehat{\tau}}} 

\newcommand\ph{{\widehat{p}}} 
\newcommand\ty{{\beta}} 
\newcommand\tyt{{\widetilde{\beta}}} 
\newcommand\tyl{{e}} 

\newcommand{\ee}[1]{{\mathrm{e}^{#1}}} 
\newcommand{\cl}[1]{{\lceil #1 \rceil}} 
\newcommand{\fl}[1]{{\lfloor #1 \rfloor}} 
\newcommand{\1}[1]{{{\bf 1}_{\{ {#1} \}}}} 
\newcommand{\sto}[2][\longrightarrow]{{\substack{ \\ {#1} \\ {#2}}}} 
\newcommand{\ssim}[1]{{\underset{#1}{\sim}}} 
\newcommand{\parent}[1]{{\buildrel \leftarrow \over {#1}}} 
\newcommand{\up}[1]{\textsuperscript{#1}} 
\newcommand\eps{\varepsilon} 
\renewcommand\d{\, \mathrm{d}} 
\renewcommand{\cal}[1]{{\mathcal{#1}}} 
\newcommand{\scr}[1]{{\mathscr{#1}}} 
\newcommand{\st}{{\; :\;}} 
\newcommand{\commentblock}[1]{} 

\newcommand\bzeta{{\boldsymbol \zeta}} 
\newcommand\hzeta{{\widehat{\zeta}}} 

\title{\textbf{Scaling limit of the subdiffusive random walk on a Galton--Watson tree in random environment}}
\author{Loïc de Raphélis \thanks{UMPA UMR 5669 CNRS, ENS de Lyon, 46 allée d'Italie 69007 Lyon, FRANCE. Partially supported by the ANR project Liouville (ANR-15-CE40-0013)}}
\date{April 18, 2019}

\begin{document}

\maketitle

\begin{abstract}
  We consider a random walk on a Galton--Watson tree in random environment, in the subdiffusive case. We prove the convergence of the renormalised height function of the walk towards the continuous-time height process of a spectrally positive strictly stable Lévy process, jointly with the convergence of the renormalised trace of the walk towards the real tree coded by the latter continuous-time height process. 
\end{abstract}

\bigskip


\bigskip


\section{Introduction}\label{sec:intro}

Let us consider $N$ a point process taking values in $\bigsqcup_{n\in\N\cup\{\N\}}{\R}^n$ (with the convention that~${\R}^0$ is the empty sequence, and $\R^{\N}$ is the set of real sequences, allowing $\# N=\infty$). Let ${\V:=\big(\T,(V(u))_{u\in\T}\big)}$ be a branching random walk with reproduction law $N$, that is a random marked tree built by induction as follows: 
\begin{itemize}
\item {\bf Initialisation} \\
Generation $0$ of $\T$ is only made up of the root, denoted by $\root$. We set $V(\root)=0$. 
\item {\bf Induction} \\
Let $n≥0$, and suppose that the tree has been built up to generation $n$. If generation $n$ is empty, then generation $n+1$ is empty. Otherwise, each vertex $u$ of generation $n$ gives progeny to a set of marked children $(c(u),(\Delta V(v))_{v\in c(u)})$ according to $N$, independently of other vertices, thus forming generation $n+1$. We set $V(v):=V(u)+\Delta V(v)$ for every~$v$ in $c(u)$.
\end{itemize}
For a vertex $u$ in the tree, we will denote by $|u|$ its generation, by $\parent{u}$ its parent, by $c(u)$ its set of children in~$\T$ (which may be an infinite set), and by $u_0,u_1,\dots,u_{|u|-1}$ its ancestors at generation $0,\dots,|u|-1$ (so $u_0=\root$ and $u_{|u|-1}=\parent{u}$). We also let $\T_u:=\{v\in\T\st u≤v\}$ be the "subtree" of $\T$ rooted in $u$. For every $u,v\in\T$, we will write $u≤v$ (resp.\ $u<v$) when $u$ is an ancestor of $v$ (resp.\ $u$ is a strict ancestor of $v$). 

We let $\Pb$ be the measure of this branching random walk, $\Eb$ be the associated expectation, and $\Pb^*$ be the measure of the branching random walk conditioned to survive.\\

We are interested in the nearest-neighbour random walk $(X^\V_n)_{n≥0}$ on $\T$ starting in $\root$, and with transition probabilities depending on $(V(u))_{u\in\T}$, defined as follows for every $u\in\T$ and~$n≥0$: 
\begin{align}\label{eq:probatrans}
\P^\V\big(X^\V_{n+1}=\parent{u}\,\mid\, X^\V_n=u\big) =& \frac{\ee{-V(u)}}{\ee{-V(u)} + \sum_{w\in c(u)} \ee{-V(w)} }\nonumber\\
\P^\V\big(X^\V_{n+1}=v\,\mid\, X^\V_n=u\big) =& \frac{\ee{-V(v)}}{\ee{-V(u)} + \sum_{w\in c(u)} \ee{-V(w)} }\quad \text{for every } v\in c(u).  
\end{align} 
We denote by $(p_{u,v})_{u,v\in\T}$ these transition probabilities (with $p_{u,v}=0$ if $u$ and $v$ are not neighbours in $\T$). Notice that they only depend on $(\Delta V(v))_{v\in c(u)}$ (this can be seen by dividing both the top and bottom of the fractions above by $\ee{-V(u)}$). In order to properly define the transition probabilities from the root, we artificially add a parent $\parent{\root}$ to the root and we suppose that the random walk is reflected in $\parent{\root}$. 

We denote by $\P^\V$ the law of $(X^\V_n)_{n≥0}$ conditionally on $\V$: this law is called the \textit{quenched} law of the random walk. We associate the expectation $\E^\V$ to this law. We denote by $\P$ the \textit{annealed} law of the random walk (and by $\E$ the associated expectation), that is the quenched law averaged over $\Pb$; we will also denote by $\P^*$ the annealed probability averaged over $\Pb^*$. When there is no ambiguity on $\V$, we will simply denote the walk by $(X_n)_{n≥0}$. \\

Let us introduce the Laplace transform of $\V$, defined for all $t≥0$ by: 
\begin{equation*}
\psi(t):=\Eb\Big[ \sum_{|u|=1} \ee{-tV(u)} \Big], 
\end{equation*}
Notice that $\psi(0)$ is the mean of the offspring distribution of $\T$; in order the tree $\T$ to have a positive probability to be infinite, we will suppose from now on that $m:=\psi(0)>1$ (allowing~$m=\infty$, as we allowed $\Pb(\sum_{|u|=1}1=\infty)>0$). It was shown in this case by R.~Lyons and R.~Pemantle~\cite{lyons-pemantle} that on the event of non-extinction, the walk is transient or positive recurrent, depending on whether $\min_{t\in[0;1]}\psi(t)$ is respectively $>1$ or $<1$. When $\min_{t\in[0;1]}\psi(t)=1$, the random walk is recurrent. If $\psi$ is well-defined in a small neighbourhood of $1$ and is differentiable in $1$, then it was shown by G.~Faraud~\cite{faraud} that the walk is null recurrent when $\psi'(1)<0$. We will consider this case in this article; to sum up, we make the following hypotheses: 
\begin{equation*}
{\bf (H_c)}\begin{cases}\parbox{\textwidth}{
\begin{itemize}
\item $m:=\psi(0)=\Eb\big[\sum_{|u|=1}1\big]\in (1,\infty]$, 
\item $\min_{t\in[0;1]}\psi(t)=1$, 
\item $\psi'(1)=-\Eb\Big[ \sum_{|u|=1} V(u)\ee{-V(u)} \Big]<0$. 
\end{itemize}}\end{cases}
\end{equation*}
When the marks $V$ (also called {\it environment}) of the point process $N$ are all equal to a constant~$\ln(\lambda)$, then the walk is a {\it $\lambda$-biased random walk}. The second condition, which actually reads~$\lambda=m$, implies the third. In a seminal work, Y.~Peres and O.~Zeitouni~\cite{peres-zeitouni} proved for the $\lambda$-biased random walk a central limit theorem on the height function of the walk, $(|X_n|)_{n≥0}$, under some conditions on the moments of the offspring distribution of $\T$. Namely, they showed that when renormalised by a factor $n^{1/2}$, the height function converges towards a reflected Brownian motion. This theorem has been extended later by A.~Dembo and N.~Sun in~\cite{dembo-sun} to the case where the underlying tree is a multitype Galton--Watson tree, and under weaker hypotheses. \\

When the environment is random, in order to further understand the behaviour of the random walk we need to introduce the following quantity: 
\begin{equation*}
\kappa:=\inf\{t>1 \st \psi(t)≥1 \}  
\end{equation*}
(with the convention $\inf \emptyset:=\infty$). Indeed, in~\cite{hu-shi}, Y.~Hu and Z.~Shi showed that the height of the maximum of the random walk at time $n$ is of order $n^{\gamma}$, where $\gamma:=\min(\frac{1}{2},1-\frac{1}{\kappa})$. This indicates a phase transition at $\kappa=2$. \\

When $\kappa>2$, the walk is therefore of order $n^{1/2}$. In~\cite{faraud}, G.~Faraud generalised Y.~Peres and O.~Zeitouni's result~\cite{peres-zeitouni} to the random-environment case, at least for $\kappa>5$ (resp.\ $\kappa>8$) in the annealed case (resp.\ in the quenched case).
Then, in~\cite{aidekon-raphelis}, we considered with E.~Aïdékon the trace of the walk at time $n$, that is the sub-tree of $\T$ made up of the vertices visited by the walk:
\begin{equation*}
\ \mathcal{R}_n:=(\Tb^n,d_{\Tb^n}), 
\end{equation*}
this notation standing for the graph $\Tb^n:=\{u\in\T\st \exists k≤n,\, X_k=u\}$ regarded as a metric space, equipped with the natural graph distance denoted by $d_{\Tb^n}$. We proved for $\kappa>2$ the convergence after renormalisation by a factor $n^{1/2}$ of $\mathcal{R}_n$ towards the real tree coded by the very reflected Brownian motion towards which converges the height function. Let us say a few words on the definition of the real tree coded by a function.  

Let $g:[a;b]\to\R_+$ a càdlàg function on $\R+$. Consider $d_g$ on $[a;b]^2$ defined by $d_g(t,s)=g(t)+g(s)-2\inf_{r\in[t,s]}g(r)$. Then set $t\sim s$ if $d_g(t,s)=0$; the real tree coded by $g$ is the metric space $\Tr_g:=[a;b]/\sim$ equipped with the distance $d_g$. If $g$ has compact support, then $\Tr_g$ is compact. The set of compact real trees can be endowed with the pointed Gromov-Hausdorff distance, which makes it a Polish space. See~\cite{le-gall-rt} for a detailed introduction to real trees. \\

We intend in this article to deal with the case $\kappa\in(1;2]$. More precisely, we will show that when renormalised by a factor $n^{1-1/\kappa}$, the height function of the walk $(|X_n|)_{n≥0}$ converges towards the continuous-time height process of a spectrally positive stable Lévy process of index $\kappa$ (a notion which will be defined in the next paragraph). We will also show the joint convergence of the trace $(\mathcal{R}_n)_{n≥0}$ towards the Lévy forest coded by this very continuous-time height process. 

We mention that several results were already obtained about the behaviour of the walk in the case $\kappa\in(1;2]$. Among others, in~\cite{andreoletti-debs}, P.~Andreoletti and P.~Debs showed that the largest entirely visited generation is of order $\ln(n)$, and that the local time of the root at time $n$ is of order $n^{1/\kappa}$. This was greatly refined in~\cite{hu}, where Y.~Hu obtained the limit law of the local time of the root after renormalisation. \\ 

The proof of the main theorem in~\cite{aidekon-raphelis} relied on a result of T.~Duquesne and J.-F.~Le~Gall~\cite{duquesne-le-gall}: consider a critical Galton--Watson forest the offspring distribution of which has finite variance, then the associated height process (that is the sequence of heights of the vertices of the aforesaid forest taken in the lexicographical order induced by Neveu's notation~\cite{neveu}) converges in law after renormalisation towards the reflected Brownian motion. This implies in some sense the convergence of the forest viewed as a metric space towards the real tree coded by the later reflected Brownian motion. 

Actually~\cite{duquesne-le-gall} also covers the case of offspring distributions with regularly varying tails: this is this version of the theorem that we will use in our article. However the limit is different (if~$\kappa<2$), so let us introduce it. 

Assume that the offspring distribution is regularly varying with index $-\theta$ where $\theta\in (1;2]$. Let $(Y_t)_{t≥0}$ be a strictly stable spectrally positive Lévy process (that is with no negative jumps) of index $\theta$ (if $\theta=2$, then $Y$ is a Brownian motion). Consider $I_s^t:=\inf_{r\in[s;t]}Y_r$, and set for all~$t≥0$
\begin{equation*}
H_t:=\lim_{\eps\to 0}\frac{1}{\eps} \int_0^t\1{Y_s<I_s^t+\eps}\d s. 
\end{equation*}
This random variable $H_t$ is the local time at level $0$ of $Y-I^t$ at time $t$. As explained in Subsection~4.3 of~\cite{le-gall-le-jan}, the limit exists almost surely, and it is possible to consider a measurable version of the whole process $(H_t)_{t≥0}$. This process is called the continuous-time height process of $Y$, and was first introduced in~\cite{le-gall-le-jan} (if $\theta=2$, then $H$ is a reflected Brownian motion). Notice that given the definition, as $Y$ satisfies the scaling property $(Y_{ct})_{t≥0}\;\substack{\text{law}\\=\\}\;(c^{\frac{1}{\kappa}}Y_{t})_{t≥0}$ for any $c≥0$, so does $H$: $(H_{ct})_{t≥0}\;\substack{\text{law}\\=\\}\;(c^{1-\frac{1}{\kappa}}H_{t})_{t≥0}$. Moreover, as $Y$ is spectrally positive, it admits a Laplace transform which characterizes its law. 

The main theorem of~\cite{duquesne-le-gall} says that the height process associated with a Galton--Watson forest with regularly varying offspring distribution converges towards the continuous-time height process $H$. \\

Let us add a few technical hypotheses, which will be necessary for us to apply Kesten's renewal theorem (Theorem B of~\cite{kesten}) in certain proofs: 
\begin{equation}\tag{${\bf H_\kappa}$}
\psi(\kappa)=1,\qquad \Eb\Big[\sum_{|u|=1} (-V(u))^{+}\ee{-\kappa V(u)}\Big]<\infty\quad {\rm and}\quad\Eb\Big[\big(\sum_{|u|=1} \ee{-V(u)}\big)^\kappa\Big]<\infty, 
\end{equation}
where for $x\in\R$, $x^{+}:=\max(x,0)$. The statement of our main theorem is the following. 

\begin{theorem}\label{th:main}
Suppose $\bf (H_c)$ and $\bf (H_\kappa)$ for a certain $\kappa\in(1;2]$. Suppose also that the distribution of the point process $N$ is non-lattice. Then there exists a constant $C^\star\in(0;\infty)$ such that the following convergence holds in law under $\P^\V$ for $\Pb^*$-almost every $\V$: 
\begin{align*}
&&&\frac{1}{n^{1-\frac{1}{\kappa}}}\Big((|X_{\fl{nt}}|)_{t≥0},\mathcal{R}_n\Big)\sto[\Longrightarrow]{n\to\infty} C^\star\Big((H_t)_{t≥0},\Tr_{(H_t)_{0≤t≤1}}\Big)\qquad&& \text{if $\kappa\in(1;2)$},\\
&or\\
&&&\frac{1}{{\left(n\ln^{-1}(n)\right)}^{\frac{1}{2}}}\Big((|X_{\fl{nt}}|)_{t≥0},\mathcal{R}_n\Big)\sto[\Longrightarrow]{n\to\infty}C^\star\Big((|B_t|)_{t≥0},\Tr_{(|B_t|)_{0≤t≤1}}\Big)\qquad&& \text{if $\kappa=2$}, 
\end{align*}
where $(H_t)_{t≥0}$ is the continuous-time height process of a strictly stable spectrally positive Lévy process of index $\kappa$, $(B_t)_{t≥0}$ is a standard Brownian motion, and where $\cal{T}_{(H_t)_{0≤t≤1}}$ (resp.~$\cal{T}_{(|B_t|)_{0≤t≤1}}$) is the real tree coded by $(H_t)_{0≤t≤1}$ ((resp.~by $(|B_t|)_{0≤t≤1}$)). The convergence holds in law for the Skorokhod topology on càdlàg functions and the pointed Gromov-Hausdorff topology on compact metric spaces. 
\end{theorem}
\noindent
\begin{remark}
By requiring the fact that the point process $N$ is non-lattice we actually mean that there shall not exist any $a>0$ such that almost surely, for every $x\in N$, $x\in a\Z$.
\end{remark} 
\begin{remark}
We will see in the proofs of Proposition~\ref{prop:forest} and Proposition~\ref{prop:xkappa} that the constant $C^\star$ is actually equal to
\begin{equation*}
C^\star=\begin{cases}|\Gamma(1-\kappa)|^{-1/\kappa}\big(\frac{\kappa-1}{\kappa}\widehat{C}_\infty(\lim_{A\to\infty}K_A) \big)^{-1/\kappa}\mu (m_X)^{-1+\frac{1}{\kappa}} &\text{if}\ \kappa\in(1;2),\\
\big(\frac{1}{2}\widehat{C}_\infty(\lim_{A\to\infty}K_A) \big)^{-1/2}\mu (m_X)^{-1/2} &\text{if}\ \kappa=2, 
\end{cases}  
\end{equation*}
where $\widehat{C}_\infty$ is defined in Lemma~\ref{eq:defbetaA}, $K_A$ is defined in Lemma~\ref{lemma:KA}, and $\mu$ and $m_X$ are defined in Proposition~\ref{prop:regvar}. We were not able to compute explicitly the limit $\lim_{A→∞} K_A$. 
\end{remark}

Our paper is organised as follows. The next section (Section~\ref{sec:strategy}) gives the global strategy of the proof, which is similar to that used in~\cite{aidekon-raphelis}. It contains the proof of Proposition~\ref{prop:forest} (which is an annealed version of Theorem~\ref{th:main} for random walks on forests), provided that certain hypotheses are satisfied. 

Section~\ref{sec:studylocaltimes} introduces a change of measure on the trace, for which we give two equivalent constructions. The first construction consists in size-biasing the trace (which we will see is a multitype Galton--Watson tree); the second consists in size-biasing  the environment first, and then considering the trace of a series of random walks on the size-biased environment. 

Then, the whole Section~\ref{sec:technique} is dedicated to proving that the hypotheses required for the proof of Proposition~\ref{prop:forest} are satisfied indeed. It contains the most important part of this paper, Subsection~\ref{subsec:cvqueue1}, in which we establish the regular variation of the tail of a key random variable (denoted by $L^1$), namely the number of vertices which are the first of their ancestry line to have been visited only once by a given number of excursions of the walk. It is the main contribution of this paper, as in~\cite{aidekon-raphelis} we only had to establish the finiteness of the second moment of the aforesaid key random variable $L^1$. 

To do this, we first have to understand the behaviour of this random variable $L^1$ when the number of excursions is large, which is done in Proposition~\ref{prop:Winfini}. Next, in Subsection~\ref{subsec:cvqueue}, we carry out a sharp study of the random walk on a biased tree (inspired by the strategy developed by H.~Kesten, M.V.~Kozlov and F.~Spitzer in~\cite{kks}), on which we are led to use Kesten's renewal theorem~\cite{kesten}. 

Finally, we deduce Theorem~\ref{th:main} from Proposition~\ref{prop:forest} in Section~\ref{sec:final}. A sum-up of the notation used in the paper can be found in the appendix.

\section{Overview of the proof}\label{sec:strategy}

Inspired by the strategy of Y.~Peres and O.~Zeitouni in~\cite{peres-zeitouni} who, in order to study $\lambda$-biased random walks on a Galton--Watson tree, studied random walks on Galton--Watson trees with infinite ray, we will first prove an annealed version of our main theorem for random walks on forests. Let us denote by $\W=\big(\F,V(u)_{u\in\F}\big)$ a marked forest made up of a collection of i.i.d.\ branching random walks $(\V_i)_{i≥1}=(\T_i,V(u)_{u\in\T_i})_{i≥1}$ (as defined in Section~\ref{sec:intro}). The nearest-neighbour random walk $(X^\W_n)_{n≥0}$ on $\W$ starts on $\root_1$ the root of $\V_1$ and has transition probabilities defined as follows:
\begin{align*}
&\P^\W\big(X^\W_{n+1}=v\,\mid\, X^\W_n=u\big) = \frac{\ee{-V(v)}}{\ee{-V(u)} + \sum_{w\in c(u)} \ee{-V(w)} }\quad &&\mbox{for all }v\in c(u), \\
&\P^\W\big(X^\W_{n+1}=\parent{u}\,\mid\, X^\W_n=u\big) = \frac{\ee{-V(u)}}{\ee{-V(u)} + \sum_{w\in c(u)} \ee{-V(w)} }\quad &&\mbox{if $u$ is not a root,}\\
&\P^\W\big(X^\W_{n+1}=\root_{i+1}\,\mid\, X^\W_n=u\big) = \frac{\ee{-V(u)}}{\ee{-V(u)} + \sum_{w\in c(u)} \ee{-V(w)} }\quad &&\mbox{if for a certain $i$, $u=\root_i$,}
\end{align*} 
where for $i≥1$ we denote by $\root_i$ the root of $\T_i$. The behaviour of the random walk on $\W$ is then similar to that on $\V$ except when it is on a root. Being recurrent, it will then visit each tree composing $\W$ for a finite time, and sometimes when on a root then jump to the next one. Thus, if we let 
\begin{equation*}
\Fb:=\{u\in\W\st \exists n≥0,X^\W_n=u\}
\end{equation*}
be the sub-forest of $\W$ made up of the visited vertices, it is therefore almost surely well-defined with finite component trees. As there should not be any ambiguity on the context thereafter, we will simply denote the walk by $(X_n)_{n≥0}$, and we still denote by $\mathcal{R}_n$ the set of vertices of $\F$ visited before time $n$, seen as a metric space. 

\begin{proposition}\label{prop:forest}
Suppose $\bf (H_c)$ and $\bf (H_\kappa)$ for a certain $\kappa\in(1;2]$, and suppose that the distribution of the point process $N$ is non-lattice. Then the following convergence holds in law under $\P$: 
\begin{align*}
&&&\frac{1}{n^{1-\frac{1}{\kappa}}}\Big((|X_{\fl{nt}}|)_{t≥0},\mathcal{R}_n\Big)\sto[\Longrightarrow]{n\to\infty}C^\star\Big((H_t)_{t≥0},\Tr_{(H_t)_{0≤t≤1}}\Big)\qquad&& \text{if $\kappa\in(1;2)$},\\
&or\\
&&&\frac{1}{{\left(n\ln^{-1}(n)\right)}^{\frac{1}{2}}}\Big((|X_{\fl{nt}}|)_{t≥0},\mathcal{R}_n\Big)\sto[\Longrightarrow]{n\to\infty}C^\star\Big((|B_t|)_{t≥0},\Tr_{(|B_t|)_{0≤t≤1}}\Big)\qquad&& \text{if $\kappa=2$}, 
\end{align*}
where $(H_t)_{t≥0}$ is the continuous-time height process of a strictly stable spectrally positive Lévy process of index $\kappa$, and $(B_t)_{t≥0}$ is a standard Brownian motion. The convergence holds in law for the Skorokhod topology on càdlàg functions and the pointed Gromov-Hausdorff topology on metric spaces. 
\end{proposition}

In Section~\ref{sec:final}, we will deduce Theorem~\ref{th:main} from Proposition~\ref{prop:forest}. The strategy will be to consider excursions of the random walk above a level of height $\ln^2(n)$, the collection of which will behave as a random walk on a forest. As said before, this strategy was inspired by~\cite{peres-zeitouni}. Then the quenched version will result from the annealed one: as the tree is supercritical, conditionally on survival up to high height, the subtrees grafted at that height in $\T$ and visited by the walk tend to look like a collection of i.i.d. trees chosen according to $\Pb$ (the proof is similar to that at the end of Section~5 of~\cite{aidekon-raphelis}). \\

Let us detail the organisation of this section. First, we will see in Subsection~\ref{subsec:reduction} that the trace can be seen as a multitype Galton--Watson tree/forest, and we will introduce the notion of {\it leafed Galton--Watson forest with edge lengths}. We will then associate to~$((|X_n|)_{n≥0},\Fb)$ two leafed Galton--Watson forests with edge lengths, the first being such that its associated height process is equal to $(|X_n|)_{n≥0}$, and the second being such that its associated height process is equal to that of $\Fb$. In Subsection~\ref{subsec:cvleafed}, we will state a result on the associated height processes of such forests, under certain hypotheses. Provided that these hypotheses are satisfied by the leafed Galton--Watson forest with edge lengths associated with $(X_n)_{n≥0}$, we conclude the proof of Proposition~\ref{prop:forest} in Subsection~\ref{subsec:conclusionth2}. 

\subsection{Reduction of trees}\label{subsec:reduction}

Recall that $\Fb$ is the sub-tree of $\F$ made up of the visited vertices. For every $u\in\Fb$ not a root, let us denote by $\beta(u)$ the {\it edge local time} of $u$:
\begin{equation}\label{eq:defbeta}
\beta(u):=\#\{n≥0\st X_n=\parent{u}\ \text{and}\ X_{n+1}=u\}, 
\end{equation}
that is the number of visits made to $u$ by $X_n$ from $\parent{u}$. If $u$ is a root, we set $\beta(u)=1$. The forest $\Fb$ can be seen as the set of vertices with non-null edge local time. 

Notice that the vertex local times can be retrieved from edge local times: $\forall u\in\Fb$ not a root, $\#\{n≥0\st X_{n}=u\}=\beta(u)+\sum_{\parent{v}=u}\beta(v)$ (as the walk is recurrent, we also have $\beta(v)=\#\{n≥0\st X_n=v\ \text{and}\ X_{n+1}=\parent{v}\}$). However, edge local times are more convenient to study, as for example under $\P^\V$, the process $((\beta(u))_{|u|=n})_{n≥0}$ is Markovian, unlike the process of vertex local times. We will give a description of $\beta$ as a Markov process in Subsection~\ref{subsec:lawbeta}.  

The idea of looking at these edge local times to understand a random walk was introduced in the paper of H.~Kesten, M.V.~Kozlov and F.~Spitzer~\cite{kks}, and also used by A.-L.~Basdevant and A.~Singh in~\cite{basdevant-singh1},\cite{basdevant-singh2},\cite{basdevant-singh3}. In our case, they allow to understand the law of $\Fb$, which is given in Lemma~3.1 of~\cite{aidekon-raphelis} which we recall here.
\begin{lemma}{\bf \cite{aidekon-raphelis}}\label{lemma:gw} Under the annealed law $\P$, the marked forest $(\Fb,\beta)$ is a multitype Galton--Watson forest with roots of initial type $1$.
\end{lemma}
\noindent The reproduction law of process $(\beta(u))_{u\in\Fb}$ seen as a multitype Galton--Watson tree will be given in details in Section~\ref{sec:studylocaltimes}. \\

Let us now introduce a new kind of branching process, that we introduced in~\cite{raphelis}: leafed Galton--Watson forests with edge lengths. These forests are multitype Galton--Watson forests with edge lengths with two types, $0$ and $1$, such that only vertices of type $1$ may give progeny (so vertices of type $0$ are sterile). 

Formally, a leafed Galton--Watson forest with edge lengths consists in a triplet $(F,\tyl,\ell)$ where for every $u\in F$, $\tyl(u)\in\{0;1\}$ stands for the {\it type} of $u$ and $\ell(u)$ for the length of the edge joining $u$ with its parent. Let $\zeta$ be a probability measure on $\bigsqcup_{n\in\N\cup\{\N\}}(\{0;1\}\times \R_+)^n$ (with the convention that $(\{0;1\}\times \R_+)^0$ is the empty sequence). We build each component of the forest $(T,\tyl,\ell)$ by induction on generations as follows: 
\begin{itemize}
\item {\bf Initialisation} \\
Generation $0$ of $T$ is only made up of the root, denoted by $\root$, such that $\tyl(\root)=1$ and $\ell(\root)=0$. 
\item {\bf Induction} \\
Let $n≥0$, and suppose that the tree has been built up to generation $n$. If generation $n$ is empty, then generation $n+1$ is empty. Otherwise, each vertex $u$ of generation $n$ such that $\tyl(u)=1$ gives progeny according to $\zeta$, independently of other vertices, thus forming generation $n+1$. Vertices $u$ of generation $n$ such that $\tyl(u)=0$ give no progeny. 
\end{itemize}
The forest $(F,\tyl,\ell)$ is then built as a collection of \iid trees built as explained above. We define its associated {\it weighted height process} $H_F^\ell$, which is such that for every $n≥1$, if $u(n)$ is the $n$\up{th} vertex of $F$ in the lexicographical order, then 
\begin{equation*}
H_F^\ell(n)=\sum_{\root≤v≤u(n)}\ell(v).
\end{equation*}
We intend in this subsection to build from~$(\Fb,\beta)$ two leafed Galton--Watson forests with edge lengths: the first one will be such that its height process is equal to that of $\Fb$, and the second one will be such that its height process is equal to $(|X_n|)_{n≥0}$.

We emphasize that the height process of a forest is a different notion from that of height function of a random walk, even though the two terms are close. The height process of a forest is the sequence of heights of its vertices taken in the lexicographical order (somehow it is the height function of a deterministic walk on the tree, with jumps).  
 
Among others, we will apply to~$(\Fb,\beta)$ the transformation introduced in Section~3.2 of~\cite{raphelis} (which is inspired by another transformation introduced in~\cite{miermont}). To this end, let us define the notion of \textit{optional line of a given type}. 
\begin{definition}\label{def:optline}
Let $u\in\Fb$. For $v>u$, we denote by $\ebrint{u,v}:=\{w\in\Fb\st u<v<w\}$ the set of vertices between $u$ and $v$. 
\begin{itemize}
\item We let ${\B}_u^1$ be the set of vertices descending from $u$ in $\Fb$ having no ancestor of type~$1$ since $u$. Formally, 
\begin{equation*}
{\B}_u^1:=\big\{ v\in\Fb \st u<v\text{ and } \beta(w)\neq 1 \quad \forall w\in\ebrint{u,v} \big\}. 
\end{equation*}
\item We denote by ${\cal{L}}_u^1$ the set of vertices of type $1$ descending from $u$ in $\Fb$ and having no ancestor of type $1$ since $u$. Formally, 
\begin{equation*}
{\cal{L}}_u^1:=\big\{ v\in\Fb \st u<v,\: \beta(v)=1,\ \text{and}\ \beta(w)\neq 1 \quad \forall w\in\ebrint{u,v} \big\}. 
\end{equation*} 
\end{itemize}
We will denote by ${L}_u^1$ (resp.\ ${B}_u^1$) the cardinal of ${\cal L}_u^1$ (resp.\ ${\B}_u^1$). 
\end{definition}
Figure~\ref{f:transformation} gives a representation of these sets, among others. This construction is also valid for trees. As for every $u\in\Tb$ the law of the random variables ${\cal L}_u^1$ (resp.\ ${\B}_u^1$) only depends on the type of $u$ under the annealed law, we will generically denote $\cal{L}^1_\root$ (resp.\ $\cal{B}^1_\root$) by ${\cal L}^1$ (resp.\ ${\B}^1$), and by $L^1$ (resp.\ $B^1$) its cardinal. For $u\in\Tb$, we will write $u<\cal{L}^1$ if $u\in\cal{B}^1$ but $u\notin\cal{L}^1$. \\

\subsubsection{The forests $\Fb^R$ and $\Fb^{R^1}$}\label{subsec:FR}
Let us build a leafed Galton--Watson forest with edge lengths the associated weighted height process of which matches exactly that of $\Fb$. Let us explain how each component $\Tb_k^R$ of $\Fb^R$ is built from each component $\Tb_k$ of $\Fb$. Note that this construction implies a bijection between vertices of $\Fb$ and vertices of $\Fb^R$, as illustrated in Figure~\ref{f:transformation}. We proceed by induction as follows:
\begin{itemize}
\item {\bf Initialisation} \\
Generation $0$ of $\Tb_k^R$ is made up of the root $\root_k$ of $\Tb_k$ the $k$\up{th} component of $\Fb$, and we set $\ell(\root_k)=0$ and $\tyl(\root_k)=1$. 

\item {\bf Induction} \\
Let $n≥1$, and suppose that generation $n$ of $\Tb_k^R$ has been built. If generation $n$ of $\Tb_k^R$ is empty then generation $n+1$ of $\Tb_k^R$ is empty. 

Otherwise, by construction, to each $u\in\Tb_k^R$ of the $n$\up{th} generation of $\Tb_k^R$ such that $\tyl(u)=1$ was associated a vertex $u'$ in $\Tb_k$. Take in the lexicographical order the vertices $v'\in\Tb_k$ such that $v'\in{\B}_{u'}^1$, and add a vertex $v$ as a child of $u$ in $\Tb_k^R$, thus forming the progeny of $u$. We set $\tyl(v)=1$ if $\beta(u')=1$ (that is if $v'\in\cal{L}_u^1$) and $\tyl(v)=0$ otherwise. Then, for each of these vertices $v\in\Tb_k^R$, we set $\ell(v)=|v'|-|u'|$. 
\end{itemize} 
We let $\Fb^{R^1}$ be the sub-forest of $\Fb^R$ made up of the vertices of type $1$. By construction, we get the following lemma. 
\begin{lemma}\label{lemma:contourFr}
The forest $(\Fb^R,\tyl,\ell)$ is a leafed Galton--Watson forest with edge lengths. Denoting by $H_{R}$ the associated weighted height process, and $H_\Fb$ the height process of $\Fb$, then for all $n≥0$, 
\begin{equation*}
H^\ell_R(n)=H_\Fb(n). 
\end{equation*}
\end{lemma}

\begin{figure}[H]
\begin{tikzpicture}[line cap=round,line join=round,>=triangle 45,x=0.9cm,y=0.9cm]
\clip(4.2,-2.5) rectangle (22.8,7.5);
\draw[|->] (4.4,0)-- (4.4,6.7) node[above]{\hspace{1.1cm}\small $|H_{\Tb_k}(n)|$\phantom{))}};
\foreach \y in {1.0,2.0,3.0,4.0,5.0,6.0}
\draw[shift={(4.4,\y)},color=black] (2pt,0pt) -- (-2pt,0pt);
\draw[|->] (22.4,0)-- (22.4,6.7) node[above]{\small $H^\ell_R(n)$\phantom{))}};
\foreach \y in {1.0,2.0,3.0,4.0,5.0,6.0}
\draw[shift={(22.4,\y)},color=black] (2pt,0pt) -- (-2pt,0pt);
\draw (8,7.5) node[anchor=north west] {\large $\Tb_k$};
\draw (17,7.5) node[anchor=north west] {\large $\Tb^R_k$};
\draw [color=blue] (5,-1) ++(-2.0pt,0 pt) -- ++(2.0pt,2.0pt)--++(2.0pt,-2.0pt)--++(-2.0pt,-2.0pt)--++(-2.0pt,2.0pt);
\draw [fill=red] (5,-2) ++(-3.0pt,0 pt) -- ++(3.0pt,3.0pt)--++(3.0pt,-3.0pt)--++(-3.0pt,-3.0pt)--++(-3.0pt,3.0pt);
\draw (5,-1.5) node[anchor=north west] {Vertices $u$ s.t.\ $\beta(u)\neq 1$};
\draw (5,-0.5) node[anchor=north west] {Vertices $u$ s.t.\ $\beta(u)=1$};
\draw [dash pattern=on 3pt off 3pt,color=gray] (11,-0.5) -- (11,-1) -- (10.5,-1) -- (10.5,-0.5) -- (11,-0.5);
\draw (11,-0.5) node[anchor=north west] {The set $\B^{1}$};
\draw [dotted, color=gray] (11,-1.5)-- (10.5,-1.5)-- (10.5,-2)-- (11,-2)-- (11,-1.5);
\draw (11,-1.5) node[anchor=north west] {First generation of $\Tb^R_k$};
\draw [color=blue] (17,-1) circle (2.0pt);
\draw [fill=red] (17,-2) circle (2.5pt);
\draw (17,-1.5) node[anchor=north west] {Vertices of type $0$};
\draw (17,-0.5) node[anchor=north west] {Vertices of type $1$};
\draw [->] (12.7,3.5)-- (13.7,3.5);

\draw[dash pattern=on 2pt off 2pt,color=gray] (5.5,0.5) -- (5.5,2.5) -- (6.5,2.5) -- (6.5,2.0) -- (7.5,2.0) -- (7.5,3.5) -- (8.5,3.5) -- (8.5,1.5) -- (9.5,1.5) -- (9.5,1.0) -- (10.5,1.0) -- (10.5,1.5) -- (11.5,1.5) -- (11.5,0.5) -- (8.5,0.5) -- (7.5,1.5) -- cycle;
\draw[dotted,color=gray] (15.0,0.5) -- (15.0,2.5) -- (16.0,2.5) -- (16.0,2.0) -- (17.0,2.0) -- (17.0,3.5) -- (18.0,3.5) -- (18.0,1.5) -- (19.0,1.5) -- (19.0,1.0) -- (20.0,1.0) -- (20.0,1.5) -- (21.0,1.5) -- (21.0,0.5) -- (18.0,0.5) -- (17.0,1.5) -- cycle;
\draw (9.0,0.0)-- (6.0,1.0);
\draw (9.0,0.0)-- (9.0,1.0);
\draw (9.0,0.0)-- (11.0,1.0);
\draw (6.0,1.0)-- (6.0,2.0);
\draw (6.0,1.0)-- (8.0,2.0);
\draw (11.0,1.0)-- (10.0,2.0);
\draw (11.0,1.0)-- (11.0,2.0);
\draw (6.0,2.0)-- (5.0,3.0);
\draw (6.0,2.0)-- (5.5,3.0);
\draw (6.0,2.0)-- (7.0,3.0);
\draw (8.0,2.0)-- (8.0,3.0);
\draw (10.0,2.0)-- (10.0,3.0);
\draw (10.0,2.0)-- (10.5,3.0);
\draw (11.0,2.0)-- (12.0,3.0);
\draw (10.0,3.0)-- (9.0,4.0);
\draw (10.0,3.0)-- (10.0,4.0);
\draw (12.0,3.0)-- (11.0,4.0);
\draw (12.0,3.0)-- (12.0,4.0);
\draw (5.5,3.0)-- (5.5,4.0);
\draw (5.5,3.0)-- (6.5,4.0);
\draw (10.0,4.0)-- (9.5,5.0);
\draw (10.0,4.0)-- (10.5,5.0);
\draw (5.5,4.0)-- (5.0,5.0);
\draw (5.0,5.0)-- (5.0,6.0);
\draw (5.0,5.0)-- (6.0,6.0);
\draw (9.5,5.0)-- (9.5,6.0);
\draw (15.5,2.0)-- (14.5,3.0);
\draw (15.5,2.0)-- (15.0,3.0);
\draw (15.5,2.0)-- (16.5,3.0);
\draw (19.5,3.0)-- (18.5,4.0);
\draw (19.5,3.0)-- (19.5,4.0);
\draw (21.5,3.0)-- (20.5,4.0);
\draw (21.5,3.0)-- (21.5,4.0);
\draw (19.5,4.0)-- (19.0,5.0);
\draw (19.5,4.0)-- (20.0,5.0);
\draw (15.0,4.0)-- (14.5,5.0);
\draw (18.5,0.0)-- (15.5,1.0);
\draw (18.5,0.0)-- (18.5,1.0);
\draw (18.5,0.0)-- (20.5,1.0);
\draw (18.5,0.0)-- (15.5,2.0);
\draw (18.5,0.0)-- (17.5,2.0);
\draw (18.5,0.0)-- (17.5,3.0);
\draw (20.5,1.0)-- (19.5,2.0);
\draw (20.5,1.0)-- (19.5,3.0);
\draw (20.5,1.0)-- (20.0,3.0);
\draw (20.5,1.0)-- (20.5,2.0);
\draw (20.5,1.0)-- (21.5,3.0);
\draw (19.5,4.0)-- (19.0,6.0);
\draw (15.0,4.0)-- (14.5,6.0);
\draw (15.0,4.0)-- (15.5,6.0);
\draw (15.5,2.0)-- (15.0,4.0);
\draw (15.5,2.0)-- (16.0,4.0);
\begin{scriptsize}
\draw [fill=red] (9.0,0.0) ++(-3.0pt,0 pt) -- ++(3.0pt,3.0pt)--++(3.0pt,-3.0pt)--++(-3.0pt,-3.0pt)--++(-3.0pt,3.0pt);
\draw [color=blue] (6.0,1.0) ++(-2.0pt,0 pt) -- ++(2.0pt,2.0pt)--++(2.0pt,-2.0pt)--++(-2.0pt,-2.0pt)--++(-2.0pt,2.0pt);
\draw [color=blue] (9.0,1.0) ++(-2.0pt,0 pt) -- ++(2.0pt,2.0pt)--++(2.0pt,-2.0pt)--++(-2.0pt,-2.0pt)--++(-2.0pt,2.0pt);
\draw [fill=red] (11.0,1.0) ++(-3.0pt,0 pt) -- ++(3.0pt,3.0pt)--++(3.0pt,-3.0pt)--++(-3.0pt,-3.0pt)--++(-3.0pt,3.0pt);
\draw [fill=red] (6.0,2.0) ++(-3.0pt,0 pt) -- ++(3.0pt,3.0pt)--++(3.0pt,-3.0pt)--++(-3.0pt,-3.0pt)--++(-3.0pt,3.0pt);
\draw [color=blue] (8.0,2.0) ++(-2.0pt,0 pt) -- ++(2.0pt,2.0pt)--++(2.0pt,-2.0pt)--++(-2.0pt,-2.0pt)--++(-2.0pt,2.0pt);
\draw [color=blue] (10.0,2.0) ++(-2.0pt,0 pt) -- ++(2.0pt,2.0pt)--++(2.0pt,-2.0pt)--++(-2.0pt,-2.0pt)--++(-2.0pt,2.0pt);
\draw [color=blue] (11.0,2.0) ++(-2.0pt,0 pt) -- ++(2.0pt,2.0pt)--++(2.0pt,-2.0pt)--++(-2.0pt,-2.0pt)--++(-2.0pt,2.0pt);
\draw [color=blue] (5.0,3.0) ++(-2.0pt,0 pt) -- ++(2.0pt,2.0pt)--++(2.0pt,-2.0pt)--++(-2.0pt,-2.0pt)--++(-2.0pt,2.0pt);
\draw [color=blue] (5.5,3.0) ++(-2.0pt,0 pt) -- ++(2.0pt,2.0pt)--++(2.0pt,-2.0pt)--++(-2.0pt,-2.0pt)--++(-2.0pt,2.0pt);
\draw [color=blue] (7.0,3.0) ++(-2.0pt,0 pt) -- ++(2.0pt,2.0pt)--++(2.0pt,-2.0pt)--++(-2.0pt,-2.0pt)--++(-2.0pt,2.0pt);
\draw [fill=red] (8.0,3.0) ++(-3.0pt,0 pt) -- ++(3.0pt,3.0pt)--++(3.0pt,-3.0pt)--++(-3.0pt,-3.0pt)--++(-3.0pt,3.0pt);
\draw [fill=red] (10.0,3.0) ++(-3.0pt,0 pt) -- ++(3.0pt,3.0pt)--++(3.0pt,-3.0pt)--++(-3.0pt,-3.0pt)--++(-3.0pt,3.0pt);
\draw [color=blue] (10.5,3.0) ++(-2.0pt,0 pt) -- ++(2.0pt,2.0pt)--++(2.0pt,-2.0pt)--++(-2.0pt,-2.0pt)--++(-2.0pt,2.0pt);
\draw [fill=red] (12.0,3.0) ++(-3.0pt,0 pt) -- ++(3.0pt,3.0pt)--++(3.0pt,-3.0pt)--++(-3.0pt,-3.0pt)--++(-3.0pt,3.0pt);
\draw [color=blue] (9.0,4.0) ++(-2.0pt,0 pt) -- ++(2.0pt,2.0pt)--++(2.0pt,-2.0pt)--++(-2.0pt,-2.0pt)--++(-2.0pt,2.0pt);
\draw [fill=red] (10.0,4.0) ++(-3.0pt,0 pt) -- ++(3.0pt,3.0pt)--++(3.0pt,-3.0pt)--++(-3.0pt,-3.0pt)--++(-3.0pt,3.0pt);
\draw [color=blue] (11.0,4.0) ++(-2.0pt,0 pt) -- ++(2.0pt,2.0pt)--++(2.0pt,-2.0pt)--++(-2.0pt,-2.0pt)--++(-2.0pt,2.0pt);
\draw [color=blue] (12.0,4.0) ++(-2.0pt,0 pt) -- ++(2.0pt,2.0pt)--++(2.0pt,-2.0pt)--++(-2.0pt,-2.0pt)--++(-2.0pt,2.0pt);
\draw [fill=red] (5.5,4.0) ++(-3.0pt,0 pt) -- ++(3.0pt,3.0pt)--++(3.0pt,-3.0pt)--++(-3.0pt,-3.0pt)--++(-3.0pt,3.0pt);
\draw [color=blue] (6.5,4.0) ++(-2.0pt,0 pt) -- ++(2.0pt,2.0pt)--++(2.0pt,-2.0pt)--++(-2.0pt,-2.0pt)--++(-2.0pt,2.0pt);
\draw [color=blue] (9.5,5.0) ++(-2.0pt,0 pt) -- ++(2.0pt,2.0pt)--++(2.0pt,-2.0pt)--++(-2.0pt,-2.0pt)--++(-2.0pt,2.0pt);
\draw [color=blue] (10.5,5.0) ++(-2.0pt,0 pt) -- ++(2.0pt,2.0pt)--++(2.0pt,-2.0pt)--++(-2.0pt,-2.0pt)--++(-2.0pt,2.0pt);
\draw [color=blue] (5.0,5.0) ++(-2.0pt,0 pt) -- ++(2.0pt,2.0pt)--++(2.0pt,-2.0pt)--++(-2.0pt,-2.0pt)--++(-2.0pt,2.0pt);
\draw [color=blue] (5.0,6.0) ++(-2.0pt,0 pt) -- ++(2.0pt,2.0pt)--++(2.0pt,-2.0pt)--++(-2.0pt,-2.0pt)--++(-2.0pt,2.0pt);
\draw [color=blue] (6.0,6.0) ++(-2.0pt,0 pt) -- ++(2.0pt,2.0pt)--++(2.0pt,-2.0pt)--++(-2.0pt,-2.0pt)--++(-2.0pt,2.0pt);
\draw [fill=red] (9.5,6.0) ++(-3.0pt,0 pt) -- ++(3.0pt,3.0pt)--++(3.0pt,-3.0pt)--++(-3.0pt,-3.0pt)--++(-3.0pt,3.0pt);
\draw [color=blue] (15.5,1.0) circle (2.0pt);
\draw [color=blue] (18.5,1.0) circle (2.0pt);
\draw [fill=red] (20.5,1.0) circle (2.5pt);
\draw [fill=red] (15.5,2.0) circle (2.5pt);
\draw [color=blue] (17.5,2.0) circle (2.0pt);
\draw [color=blue] (20.5,2.0) circle (2.0pt);
\draw [color=blue] (14.5,3.0) circle (2.0pt);
\draw [color=blue] (15.0,3.0) circle (2.0pt);
\draw [color=blue] (16.5,3.0) circle (2.0pt);
\draw [fill=red] (17.5,3.0) circle (2.5pt);
\draw [fill=red] (19.5,3.0) circle (2.5pt);
\draw [color=blue] (20.0,3.0) circle (2.0pt);
\draw [fill=red] (21.5,3.0) circle (2.5pt);
\draw [color=blue] (18.5,4.0) circle (2.0pt);
\draw [fill=red] (19.5,4.0) circle (2.5pt);
\draw [color=blue] (20.5,4.0) circle (2.0pt);
\draw [color=blue] (21.5,4.0) circle (2.0pt);
\draw [fill=red] (15.0,4.0) circle (2.5pt);
\draw [color=blue] (16.0,4.0) circle (2.0pt);
\draw [color=blue] (19.0,5.0) circle (2.0pt);
\draw [color=blue] (20.0,5.0) circle (2.0pt);
\draw [color=blue] (14.5,5.0) circle (2.0pt);
\draw [color=blue] (14.5,6.0) circle (2.0pt);
\draw [color=blue] (15.5,6.0) circle (2.0pt);
\draw [fill=red] (19.0,6.0) circle (2.5pt);
\draw [fill=red] (18.5,0.0) circle (2.5pt);
\draw [color=blue] (19.5,2.0) circle (2.0pt);
\end{scriptsize}
\end{tikzpicture}
\caption{Construction of $\Tb^R_k$ from $\Tb_k$. }
\label{f:transformation}
\end{figure}

\subsubsection{The forests $\Fb^X$ and $\Fb^{X^1}$}\label{subsec:FX}

Consider the forest $\Fb^{R}$. The set of vertices of this forest can be put in bijection with the set of vertices visited by $(X_n)_{n≥0}$. However one cannot completely retrieve $(|X_n|)_{n≥0}$ from $\Fb^R$, as some vertices have been visited several times, and in a specific order. Let us build $\Fb^X$ a leafed Galton--forest with edge lengths, the associated weighted height process of which is equal to~$(|X_n|)_{n≥0}$. 

Take $\Fb^{R}$. In this forest, vertices of type $0$ are vertices that have been visited from their parent more than once. For each vertex $v$ of type $0$, add as siblings as many vertices as the number of time it was visited by $(X_n)_{n≥0}$ (minus one, as $v$ stands already for one step of $(X_n)_{n≥0}$) (that is even when the visit was made from a child), and attach them to the parent of $v$ (of type $1$, as in $\Fb^R$) of~$v$. Put~$0$ as their type, and give them the same length as that of $v$. Do the same for vertices~$u$ of type~$1$: for each visit of $(X_n)_{n≥0}$ to $u$ from a children, attach a new vertex to $u$, of type $0$, and length $0$. 

Now, each vertex in this forest correspond to a step of $(X_n)_{n≥0}$. In each set of sibling, re-order the vertices according to the time of visit of the step to which they correspond. The forest $\Fb^X$ is built. 

The sub-forest of $\Fb^X$ made up only of vertices of type $1$, denoted by $\Fb^{X^1}$, is the same as~$\Fb^{R^1}$, up to a chronological ordering of the vertices (vertices are ordered according to their first time of visit, rather than to the spatial position of the vertex in $\F$ to which they correspond). 

By construction, we get the following lemma: 
\begin{lemma}\label{lemma:contourFx}
The forest $\Fb^X$ is a leafed Galton--Watson forest with edge lengths. Let $H^\ell_X$ be its associated weighted height process. Then for each $n≥0$ we have $|X_n|=H^\ell_X(n)$. 
\end{lemma}
\noindent Thus, any result obtained on the scaling limit of $H^\ell_X$ will be valid for $(|X_n|)_{n≥0}$ too. The tree reduction is summarised in the following diagram. \\

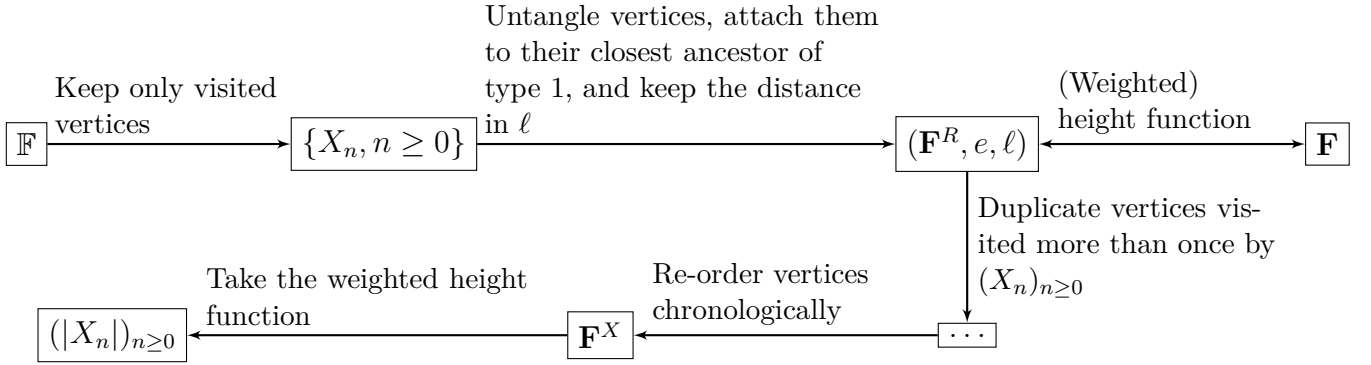
\tikzstyle{block} = [rectangle, draw, text centered]
\tikzstyle{line} = [draw, -latex', thick]
\tikzstyle{line2} = [draw, latex'-latex',thick]
\begin{figure}[H]
\begin{tikzpicture}[node distance=1cm, auto]
\node (init) {};
\node [block] (A) {$\F$};
\node [block, right=3.2cm of A] (B) {$\{X_n,n≥0\}$};
\node [block, right=5.5cm of B] (C) {$(\Fb^R,\tyl,\ell)$};
\node [block, below=2cm of C] (D) {$\ldots$};
\node [block, left=4cm of D] (E) {$\Fb^{X}$};
\node [block, left=5cm of E] (F) {$(|X_n|)_{n≥0}$};
\node [block, right=3.5cm of C] (W) {$\Fb$};
\begin{small}
\path [line] (A) -- node [text width=3cm,midway,above ] {Keep only visited vertices} (B);
\path [line] (B) -- node [text width=5.3cm,midway,above ] {Untangle vertices, attach them to their closest ancestor of type 1, and keep the distance in $\ell$} (C);
\path [line] (C) -- node [text width=4cm,midway,right ] {Duplicate vertices visited more than once by $(X_n)_{n≥0}$} (D);
\path [line] (D) -- node [text width=3.5cm,midway,above ] {Re-order vertices chronologically} (E);
\path [line] (E) -- node [text width=4.5cm,midway,above ] {Take the weighted height function} (F);
\path [line2] (C) -- node [text width=3cm,midway,above ] {(Weighted) height function} (W);
\end{small}
\end{tikzpicture}
\end{figure}

\subsection{Convergence of the height processes associated with leafed Galton--Watson forests with edge lengths}\label{subsec:cvleafed}

Lemmas~\ref{lemma:contourFr} and~\ref{lemma:contourFx} indicate that the study of $\Fb$ and $(|X_n^\W|)_{n≥0}$ can boil down to that of certain leafed Galton--Watson forests with edge lengths. We intend in this section to give a result on the scaling limit of the height process associated with such forests. In~\cite{raphelis}, we proved under a condition of finite variance (Theorem~1) that their height process converges towards the reflected Brownian motion. Adjusting the proof of this theorem, we will be able to obtain an equivalent result in the infinite variance case. 

So let us consider the setting of~\cite{raphelis}: let $(F,\tyl,\ell)$ be a leafed Galton--Watson forest of reproduction law $\zeta$, recall that for every vertex $u\in F$ we let $\nu(u)$ (resp.\ $\nu^1(u)$) be the total number of children (resp.\ number of children of type $1$) of $u$ in $F$, and recall that we denote by $F^1$ the forest $F$ limited to its vertices of type $1$. 
We make the following hypotheses on $\zeta$: 
\begin{equation*}
{\bf (H_{\ell})}\begin{cases}\parbox{\textwidth}{
\begin{itemize}
\item[$(i)$] $\E\big[ \sum_{|u|=1,\tyl(u)=1} 1 \big]=\E\big[ \nu^1 \big]=1$,
\item[$(ii)$] There exists $\eps>0$ s.t.\ $\E\big[ \big(\sum_{|u|=1} 1\big)^{1+\eps} \big]=\E\big[ (\nu)^{1+\eps} \big]<\infty$,
\item[$(iii)$] There exists $C_0>0$ s.t.\ $\P\big( \sum_{|u|=1,\tyl(u)=1}1>x \big)=\P(\nu^1>x)\ssim{x\to\infty} C_0 x^{-\kappa}$, 
\item[$(iv)$] There exists $r>1$ s.t.\ $\E\Big[ \sum_{|u|=1} r^{\ell(u)} \Big]<\infty$. 
\end{itemize}}\end{cases}
\end{equation*}
Following the notation of~\cite{raphelis}, we let 
\begin{equation*}
m:=\E\big[\sum_{|u|=1}1\big]=\E[\nu]\quad \text{and}\quad \mu:=\E\big[\sum_{|u|=1,\tyl(u)=1} \ell(u)\big].
\end{equation*}
We let $H_F^1$ be the height process of $F^1$ and recall that we denote by $H_F^{\ell}$ the \textit{weighted} height process of $F$:
\begin{equation*}
\forall n\in\N,\;H_F^1(n):=|u^1(n)|\;\text{ and }\;H_F^{\ell}(n)=\sum_{\root≤v≤u(n)}\ell(v), 
\end{equation*}
where $u(n)$ (resp.\ $u^1(n)$) stands for the $n$\up{th} vertex of $F$ (resp.\ $F^1$) taken in the lexicographical order. Our main result on leafed Galton--Watson forests with edge lengths,  which is an extension of Theorem~1 of~\cite{raphelis}, is the following. 
\begin{theorem}{\label{th:bitype}}
Let $(F,\tyl,\ell)$ be a leafed Galton--Watson forest with edge lengths, with offspring distribution $\zeta$ satisfying hypothesis $\bf (H_{\ell})$. The following convergence in law holds for the Skorokhod topology on the space $\mathbb{D}(\R_+,\R)$ of càdlàg functions: 
\begin{align*}
&\frac{1}{n^{1-1/\kappa}}\left(\Big(H_F^{\ell}(\fl{ns})\Big)_{s≥0},(H_F^{1}(\fl{ns}))_{s≥0}\right) \substack{ \Longrightarrow \\ n\to\infty} \frac{1}{(C_0|\Gamma(1-\kappa)|)^{\frac{1}{\kappa}}}\left(\Big(\mu H_{m^{-1}s}\Big)_{s≥0},(H_s)_{s≥0}\right)&&\text{if $1<\kappa<2$}, \\
&or\\
&\frac{1}{(n\ln^{-1}(n))^{\frac{1}{2}}}\left(\Big(H_F^{\ell}(\fl{ns})\Big)_{s≥0},(H_F^{1}(\fl{ns}))_{s≥0}\right) \substack{ \Longrightarrow \\ n\to\infty} \frac{1}{(C_0)^{\frac{1}{2}}}\left(\Big(\mu |B_{m^{-1}s|}\Big)_{s≥0},(|B_s|)_{s≥0}\right)&&\text{if $\kappa=2$}, 
\end{align*}
where $H$ is the continuous-time height process of a spectrally positive Lévy process $Y$ of Laplace transform $\E[\exp(-\lambda Y_t)]=\exp(t\lambda^\kappa)$ for $\lambda,t≥0$, $(B_t)_{t≥0}$ is a standard Brownian motion, and where $C_0$ is the constant for which ${\bf (H_{\ell})}(iii)$ is satisfied. 
\end{theorem}
\begin{proof}
See the appendix. 
\end{proof}

We acknowledge that hypotheses ${\bf (H_{\ell})} (ii)$ and ${\bf (H_{\ell})} (iv)$ can probably be weakened, but this would be unnecessary for the sequel of our paper and this would lengthen the proof.

\subsection{Proof of Proposition~\ref{prop:forest}}\label{subsec:conclusionth2}
Let us introduce the following proposition, which will allow us to apply Theorem~\ref{th:bitype} to the forests $\Fb^R$ and $\Fb^X$. 

\begin{proposition}\label{prop:regvar}
Suppose $\bf (H_c)$ and $\bf (H_\kappa)$ for a certain $\kappa\in(1;2]$, and suppose that the distribution of the point process $N$ is non-lattice. Then the leafed Galton--Watson forests with edge lengths $\Fb^R$ and $\Fb^X$ satisfy hypotheses $\bf (H_{\ell})$. In particular, they satisfy ${\bf H_\ell}(iii)$ for the same constant. Moreover, the associated constants $\mu,m$ are respectively
\begin{align*}
&\mu_R=\mu:=(a_1 b_1)^{-1}\quad\text{and}\quad m_R=a_1^{-1}\quad\text{for $\Fb^R$},\\
&\mu_X=\mu=(a_1 b_1)^{-1}\quad\text{and}\quad m_X=2(a_1 b_1)^{-1}\quad\text{for $\Fb^X$}, 
\end{align*}
where $a_1$ and $b_1$ are defined in Subsection~\ref{subsec:multitype}
\end{proposition}
\begin{proof}
The proof of this proposition is the subject of Section~\ref{sec:technique}. Hypothesis ${\bf (H_\ell)}(i)$ is shown to be  satisfied by the laws of $\Fb^R$ and $\Fb^X$ in Lemma~\ref{lemma:Hl1}, hypothesis~${\bf (H_\ell)}(iv)$ in Lemma~\ref{lemma:lyapunov}, hypothesis ${\bf (H_\ell)}(ii)$ in Lemma~\ref{lemma:momB1}, and hypothesis ${\bf (H_\ell)}(iii)$ in Proposition~\ref{prop:xkappa}. The constants are computed in Lemma~\ref{lemma:constants}. 
\end{proof}
 Let us give a lemma on the range $R_n$ the cardinal of $\mathcal{R}_n$ before giving the proof of Proposition~\ref{prop:forest}.

\begin{lemma}\label{lemma:Rn}
Let $R_n:=\#\{X_k, k≤n\}$ be the number of different vertices visited by the walk at time $n$. Then, as $n\to\infty$, 
\begin{equation*}
R_n/n \to m_R/m_X\quad \text{a.s.}
\end{equation*}
\end{lemma}
\begin{proof}
Let $R^1_n:=\#\{X_k, k≤n\st \beta(X_k)=1\}$ be the number of vertices visited by the walk at time $n$ that will eventually be of type $1$. For any $k≥1$, let us denote by $\chi^X(k)$ (resp.\ $\chi^R(k)$) the index for the lexicographical order of the $k$\up{th} vertex of type $1$ in $\Fb^X$ (resp.\ in ${\Fb^R}$). Then
\begin{equation*}
\chi^X(R^1_n)≤n≤\chi^X(R^1_{n+1})\quad\text{and}\quad \chi^R(R^1_n)≤R_n≤\chi^R(R^1_{n+1}).
\end{equation*}
According to the equation below equation~(2.10) in the proof of Proposition~5 of~\cite{raphelis} (where $\chi$ is denoted by $\psi$), almost surely $\chi^X(R^1_n)/R^1_n\to m_X$ and $\chi^R(R^1_n)/R^1_n\to m_R$ as $n→∞$. As $|R_{n+1}-R_n|≤1$, we obtain $R_n/n\to m_R/m_X$ as $n→∞$ almost surely. 
\end{proof}
We are now ready to give the proof of Proposition~\ref{prop:forest}. This proof is summarized in a diagram at the end. 
\begin{proof}[Proof of Proposition~\ref{prop:forest}]
Recall that we denote by $H^\ell_R$ (resp.\ $H^\ell_X$) the weighted height process of the leafed Galton--Watson forest with edge lengths $\Fb^R$ (resp.\ $\Fb^X$). Notice also that $H_R^1$ (resp.\ $H_X^1$) the height process restricted to vertices of type $1$ associated with $\Fb^R$ (resp.\ $\Fb^X$) is actually the non-weighted height process associated with $\Fb^{R^1}$ (resp.\ $\Fb^{X^1}$). Let $K^\star$ be the constant for which the offspring distributions of $\Fb^R$ and $\Fb^X$ satisfy ${\bf (H_{\ell})} (iii)$, and let us denote for every $n≥1$, 
\begin{equation*}
c_n:=(K^\star|\Gamma(1-\kappa)|)^{-1/\kappa}n^{1-1/\kappa}\quad \text{if $\kappa\in(1;2)$}\qquad\textit{or}\qquad c_n:=((2K^\star)^{-1} n\ln^{-1}(n))^{1/2}\quad \text{if $\kappa=2$}. 
\end{equation*}
Now according to Proposition~\ref{prop:regvar}, we can apply Theorem~\ref{th:bitype} to $\Fb^R$: for the Skorokhod topology on the space of càdlàg functions, we have the following convergence in law: 
\begin{equation}\label{eq:cvHr'}
c_n^{-1}\left( ({H^\ell_R}(\fl{nt}))_{t≥0},{H^1_{R}}(\fl{nt})_{t≥0} \right)\Rightarrow \left( \mu(H_{m_R^{-1} t})_{t≥0},(H_t)_{t≥0} \right), 
\end{equation} 
where $H$ is the continuous-time height process of a spectrally positive Lévy process $Y$ of Laplace transform $\E[\exp(-\lambda X_t)]=\exp(t\lambda^\kappa)$ for $t,\lambda>0$ (if $\kappa=2$, then $H=\sqrt{2}|B|$ where $B$ is a standard Brownian motion). 

Now notice that for every $n≥1$, $\mathcal{R}_n$ is coded by $({H_R^\ell}(k))_{1≤k≤R_n}$. This together with~(\ref{eq:cvHr'}) and Lemma~\ref{lemma:Rn} ($R_n/n\to m_R/m_X$) yield the following convergence: 
\begin{equation}\label{eq:cvT}
c_n^{-1}\left( \mathcal{R}_n,\Tr_{({H^1_{R}}(\fl{nt}))_{0≤t≤1}} \right)\Rightarrow \left( \mu\Tr_{(H_{m_X^{-1} t})_{0≤t≤1}},\Tr_{(H_t)_{0≤t≤1}} \right), 
\end{equation} 
where the convergence for the Gromov-Hausdorff topology is due to Lemma~2.4 of~\cite{le-gall-rt} (the convergence of the height function implying that of the contour function in our case, as we can use the same arguments as those of Theorem~2.4.1 of~\cite{duquesne-le-gall}). Now, still thanks to Proposition~\ref{prop:regvar}, let us apply Theorem~\ref{th:bitype} to $\Fb^X$:
\begin{equation}\label{eq:cvHX}
c_n^{-1}\left( (H^\ell_X(\fl{nt}))_{t≥0},\Tr_{(H^1_{X}(\fl{nt}))_{0≤t≤1}} \right)\Rightarrow \left( \mu(H_{m_X^{-1} t})_{t≥0},\Tr_{(H_t)_{0≤t≤1}} \right). 
\end{equation} 
Notice that by construction, $\Fb^{X^1}$ is equal to ${\Fb^{R^1}}$ up to a re-ordering of the vertices. Therefore, $\Tr_{({H^1_{R}}(\fl{nt}))_{t≥0}}=\Tr_{(H^1_{X}(\fl{nt}))_{t≥0}}$ for all $n\in\N^*$, $t≥0$. Moreover, according to Lemma~\ref{lemma:contourFx}, we have for every $n≥0$ that $H^\ell_{X}(n)=|X_n|$. So equation~(\ref{eq:cvHX}) yields: 
\begin{equation}\label{eq:cvXn}
c_n^{-1}\left( (|X_{\fl{nt}}|)_{t≥0},\Tr_{({H^1_{R}}(\fl{nt}))_{0≤t≤1}} \right)\Rightarrow \left( \mu(H_{m_X^{-1} t})_{t≥0},\Tr_{(H_t)_{0≤t≤1}} \right). 
\end{equation} 
Together with~\eqref{eq:cvT} (and the scaling property: $(H_{ct})_t$ has same law as $(c^{1-\frac{1}{\kappa}}H_t)_t$), this yields the convergence in law of the couple 
\begin{equation*}
c_n^{-1}\left((|X_{\fl{nt}}|)_{t≥0},\mathcal{R}_n\right)\Rightarrow\mu(m_X)^{-1+\frac{1}{\kappa}}\left((H_{t})_{t≥0},\Tr_{(H_t)_{0≤t≤1}}\right).
\end{equation*}
This can be seen by the following reasoning. According to Skorokhod's representation theorem, there exists a probability space in which the convergence~(\ref{eq:cvXn}) holds \as and therefore in probability. It is then possible in this new probability space to build a sequence of random variables which has same law as $(\mathcal{R}_n)_{n≥1}$, and which depends on $(|X_n|)_{n≥0}$ and $\mathcal{\Fb}^{R^1}$ as initially. Now thanks to the convergence of $\mathcal{T}_{H_R^1}$ in this space, the convergence of~(\ref{eq:cvT}) implies the convergence in probability of $\mathcal{R}_n$ in the new probability space towards the real tree coded by the continuous time height process. This implies the convergence in probability (still in this space) of the couple $c_n^{-1}\Big((|X_n|)_{t≥0},\mathcal{R}_n\Big)$ towards $\mu\Big( (H_{m_X^{-1}t})_{t≥0},\Tr_{(H_{m_X^{-1}t})_{0≤t≤1}} \Big)$, convergence which then holds in law for the initial couple (which was in the initial probability space).
\end{proof}

For the sake of understanding, we summarize the proof in the following diagram. \\

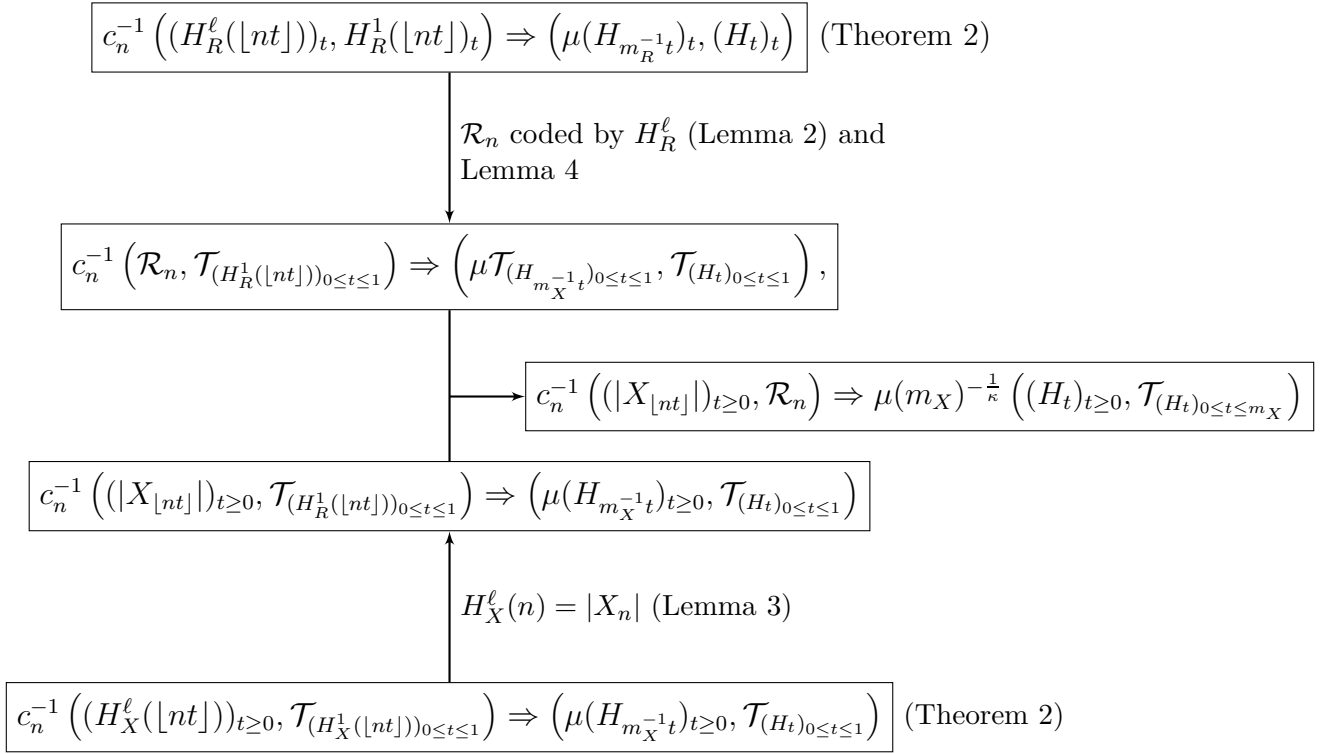
\tikzstyle{block} = [rectangle, text centered,draw]
\tikzstyle{blocksans} = [rectangle, text centered]
\tikzstyle{line} = [draw, thick, -latex']
\tikzstyle{line2} = [draw, thick, latex'-latex']
\tikzstyle{linedouble} = [draw,double, thick, -latex']
\tikzstyle{linesans} = [draw, thick, -]
\begin{figure}[H]
\begin{tikzpicture}[node distance=1cm, auto]
\node (init) {};
\node [block] (A) {$\displaystyle c_n^{-1}\left( ({H^\ell_R}(\fl{nt}))_{t},{H^1_{R}}(\fl{nt})_{t} \right) \Rightarrow \left( \mu(H_{m_R^{-1} t})_{t},(H_t)_{t} \right) $};
\node [blocksans, right=0cm of A] (T) {(Theorem~\ref{th:bitype})};
\node [block, below=2cm of A] (B) {$\displaystyle c_n^{-1}\left( \mathcal{R}_n,\Tr_{({H^1_{R}}(\fl{nt}))_{0≤t≤1}} \right)\Rightarrow \left( \mu\Tr_{(H_{m_X^{-1} t})_{0≤t≤1}},\Tr_{(H_t)_{0≤t≤1}} \right), $};
\node [below=1cm of B] (X) {};
\node [right=0cm of X] (W) {};
\node [left=0.15cm of W] (Y) {};
\node [block, right=1cm of Y] (C) {$c_n^{-1}\left((|X_{\fl{nt}}|)_{t≥0},\mathcal{R}_n\right)\Rightarrow\mu(m_X)^{-\frac{1}{\kappa}}\left((H_{t})_{t≥0},\Tr_{(H_t)_{0≤t≤m_X}}\right)$};
\node [block, below=2cm of B] (D) {$c_n^{-1}\left( (|X_{\fl{nt}}|)_{t≥0},\Tr_{({H^1_{R}}(\fl{nt}))_{0≤t≤1}} \right)\Rightarrow \left( \mu(H_{m_X^{-1} t})_{t≥0},\Tr_{(H_t)_{0≤t≤1}} \right)$};
\node [block, below=2cm of D] (E) {$c_n^{-1}\left( (H^\ell_X(\fl{nt}))_{t≥0},\Tr_{(H^1_{X}(\fl{nt}))_{0≤t≤1}} \right)\Rightarrow \left( \mu(H_{m_X^{-1} t})_{t≥0},\Tr_{(H_t)_{0≤t≤1}} \right)$};
\begin{small}
\node [blocksans, right=0cm of E] {(Theorem~\ref{th:bitype})};
\path [line] (A) -- node [text width=6cm,midway,right ] {$\mathcal{R}_n$ coded by $H^\ell_R$ (Lemma~\ref{lemma:contourFr}) and Lemma~\ref{lemma:Rn}} (B);
\path [line] (E) -- node [text width=6cm,midway,right ] {$H^\ell_X(n)=|X_{n}|$ (Lemma~\ref{lemma:contourFx})} (D);
\path [linesans] (B) -- (D);
\path [line] (Y) -- (C);
\end{small}
\end{tikzpicture}
\end{figure}

\section{Study of the trace of the random walk and change of measure}\label{sec:studylocaltimes}

\subsection{Quenched law of $\beta$}\label{subsec:lawbeta}

Under $\P^\V$, the process $(\beta(u))_{u\in\Fb}$ introduced in Subsection~\ref{subsec:reduction} is not a multitype Galton--Watson tree (unless $\V$ is a regular tree). However the description of its law under $\P^\V$ will allow us to understand its law under $\P$, and it will also be very useful in Subsection~\ref{subsec:cvqueue1}. 

Recall that we denote by $(p_{u,v})_{u,v,\in\F}$ the transition probabilities of the random walk $(X_n)_{n≥0}$. 

\begin{lemma}\label{lemma:lawbeta}
Let $k≥1$ and $u\in\Fb$. Then, under $\P^\V$, conditionally on $\beta(u)=k$,the random variable $(\beta(v))_{v\in c(u)}$ has negative multinomial distribution with parameters $k$ and $(p_{u,v})_{v\in c(u)}$ (see Chapter 36 of~\cite{johnson} for details on this law). That is for every $k≥1$ and every sequence with finite support $(\ell_v)_{v\in c(u)}$, 
\begin{equation}\label{eq:lawbeta}
\P^\V\big((\beta(v))_{v\in c(u)}=(\ell_v)_{v\in c(u)}|\beta(u)=k\big)=p_{u,\parent{u}}^k\binom{k-1+\sum_{v\in c(u)}\ell_v}{(\ell_v)_{v\in c(u)},k-1}\times \prod_{v\in c(u)}p_{u,v}^{\ell_v}, 
\end{equation}
where for $(a_n)_n$ a sequence with finite support, $\binom{\sum_n a_n}{(a_n)_{n}}:=\frac{(\sum_n a_n)!}{\prod_n a_n!}$ stands for the multinomial coefficient of the sequence $(a_n)_n$. Its generating function is, for every $(s_v)_{c(u)}\in[0;1]^{c(u)}$, 
\begin{equation*}
\E^\V[\prod_{v\in c(u)} s_v^{\beta(v)}|\beta(u)=k]=\Big(\frac{p_{u,\parent{u}}}{1-\sum_{v\in c(u)} p_{u,v}s_v}\Big)^k. 
\end{equation*}
\end{lemma}
\begin{remark} A negative multinomial random variable of parameter $k$ and $(p_i)_{i≥1}$ has the law of the sum of $k$ independent negative multinomial random variables of parameters $1$ and $(p_i)_{i≥1}$. 
\end{remark}

\begin{proof}
If $\beta(u)=1$ and $c(u)=\{v\}$, that is $u$ has only one child, then $\beta(v)$ is a geometric distribution with parameter $p_{u,\parent{u}}$ (each time we have $X_n=u$, either it goes to $v$ ("failure" event, with probability $p_{u,v}$) and $\beta(v)$ is increased by one, or it goes to $\parent{u}$ ("success" event, with probability $p_{u,\parent{u}}$) and never comes back, as there was only one ascent from $\parent{u}$ to $u$). In this case, for every $\ell≥0$, 
\begin{equation*}
\P^\V\big(\beta(v)=\ell|\beta(u)=1\big)=p_{u,\parent{u}}p_{u,v}^\ell. 
\end{equation*}

If $u$ has several children, then each time we have $X_n=u$, either it goes to a vertex $v\in c(u)$, with probability $p_{u,v}$, the local time $\beta(v)$ is increased by one, and since the random walk is recurrent it eventually comes back to $u$. Or it goes to $\parent{u}$ and never comes back. In this case, we have for any $(\ell_v)_{v\in c(u)}$,
\begin{equation*}
\P^\V\big((\beta(v))_{v\in c(u)}=(\ell_v)_{v\in c(u)}|\beta(u)=1\big)=p_{u,\parent{u}}\binom{\sum_{v\in c(u)}\ell_v}{(\ell_v)_{v\in c(u)}} \prod_{v\in c(u)}p_{u,v}^{\ell_v}. 
\end{equation*}

Finally, if $\beta(u)=k$, $(\beta(v))_{v\in c(u)}$ is the sum of the local times of $k$ \iid excursions. Another way to describe it : $(\beta(v))_{v\in c(u)}$ is the number of outcomes of the events $\{X_n=u,X_{n+1}=v\}$ for $v\in c(u)$, before the event $\{X_n=u,X_{n+1}=\parent{u}\}$ has occurred $k$ times. That is $(\beta(v))_{v\in c(u)}$ has negative multinomial distribution (see Chapter~36 of~\cite{johnson}) with parameters $k$ and $(p_{u,v})_{v\in c(u)}$, and has distribution as prescribed in the statement of the lemma. The generating function of such a law is given in~(36.1) of~\cite{johnson}. 
\end{proof}

\begin{remark} In this proof and in the remaining of the paper, when mentioning the geometric distribution, we consider the definition for which the support is $\{0,1,2,\dots\}$. 
\end{remark}

We give a few consequences of this lemma, some of which may seem trivial but will actually be quite helpful to understand the process. 

\begin{lemma}\label{lemma:lawbetaconseq}
Let $u\in\F$. 
\begin{itemize}
\item Let $c'(u)\subset c(u)$. Under $\P^\V$, conditionally on $\beta(u)$,the random variable $(\beta(v))_{v\in c'(u)}$ has negative multinomial distribution with parameters $\beta(u)$ and $(\ee{-\Delta V(v)}/(1+\sum_{w\in c'(u)}\ee{-\Delta V(w)}))_{v\in c'(u)}$.
\item Let $c'(u)\subset c(u)$. Under $\P^\V$, conditionally on $\beta(u)$, the random variable $\sum_{v\in c'(u)} \beta(v)$ has the law of the sum of $\beta(u)$ geometric random variables of parameter $1/(1+\sum_{w\in c'(u)}\ee{-\Delta V(w)})$ (that is of expectation $\beta(u)\sum_{w\in c'(u)}\ee{-\Delta V(w)})$). 

In particular, for any $v\in c(u)$, under $\P^\V$, $\beta(v)$ has the law of the sum of $\beta(u)$ independent geometric random variables of parameter $1/(1+\ee{-\Delta V(v)})$
\item For any $w\in\F$ such that $u≤w$, 
\begin{equation}\label{eq:espcondbeta}
\E^\V\big[\beta(w)|\beta(u)\big]=\beta(u)\ee{-(V(w)-V(u))}. 
\end{equation}
\end{itemize}
\end{lemma}

\begin{proof}
The first point can be established by simply summing over all $(\ell_v)_{v\in c(u)\setminus c'(u)}$ in \eqref{eq:lawbeta}, but this can also be seen as follows. Since $(X_n)_{n≥0}$ is a Markov chain under $\P^\V$, we can consider $(\widetilde{X}_n)_{n≥0}$ its restriction to the tree $\T$ without the subtrees $(\T_v)_{v\in c(u)\setminus c'(u)}$. The probability to go from $u$ to $v\in c'(u)$ is now $\ee{-\Delta V(v)}/(1+\sum_{w\in c'(u)}\ee{-\Delta V(w)})$, and we can do the same proof as that of Lemma~\ref{lemma:lawbeta} to get the result. 

The second point can be shown using the same restricted Markov chain as in the first point, but considering the event $\cup_{v\in c'(u)}\{$the random walk goes to $v$\empty$\}$, against the event $\{$the random walk goes to $\parent{u}$\empty$\}$, which happens with probability $1/(1+\sum_{v\in c'(u)}\ee{-\Delta V(v)})$. 

Finally, the third point can be shown by induction, as for every $v\in\T$, one has (according to the first point of the lemma) $\E[\beta(v)|\beta(\parent{v})]=\beta(\parent{v})\frac{p_{u,\parent{u}}}{p_{u,v}}=\beta(\parent{v})\ee{-\Delta V(v)}$, 
\begin{align*}
\E^\V\big[\beta(w)|\beta(u)\big]=\E^\V\big[\E^\V\big[\beta(w)|\beta(\parent{w}),\beta(u)\big]|\beta(u)\big]&=\E^\V\big[\beta(\parent{w})|\beta(u)\big]\ee{-(V(w)-V(\parent{w}))}\\
&=\dots=\beta(u)\ee{-(V(w)-V(u))}. 
\end{align*} 
\end{proof}

\subsection{Law of $(\Fb,\beta)$ as a multitype Galton--Watson forest}\label{subsec:multitype}

Recall from Lemma~\ref{lemma:gw} that under $\P$, the marked forest $(\Fb,\beta)$ is a multitype Galton--Watson forest with roots of type $1$.

We denote by $\bzeta=(\zeta_i)_{i≥1}$ the generic offspring distribution of this multitype Galton--Watson tree (that is for every $i≥1$, $\zeta_i$ stands for the offspring distribution of a vertex of type~$i$). For every $i≥1$, we denote by $\P_i$ the law of a generic multitype Galton--Watson tree $(\Tb,\beta)$ with such an offspring distribution, and with initial type $i$.

Considering the setting of multitype Galton--Watson forests used in~\cite{raphelis}, (\ref{eq:lawbeta}) yields that the mean matrix $(m_{i,j})_{i,j≥1}$ of $\Tb$ is given by 
\begin{equation}\label{eq:defmij}
m_{i,j}:=\E_i\Big[ \sum_{|u|=1}\1{\beta(u)=j} \Big]=\binom{i+j-1}{j}\Eb\Big[ \sum_{|u|=1} \frac{\ee{-jV(u)}}{(1+\ee{-V(u)})^{i+j}} \Big]  
\end{equation} 
(to establish the previous equation, one just has to use the expression of the quenched probability $\P^\V(\beta(u)=j)$ for a given $u$ and the linearity of the expectation)
Let us introduce a random variable $\widehat{S}_1$ with law defined by $\E[ f(\widehat{S}_1) ]=\Eb[ \sum_{|u|=1} f(V(u))\ee{-V(u)}]$ for any bounded continuous real function $f$. We define $(\widehat{S}_k)_{k≥0}$ as the random walk starting from $0$ with step distribution $\widehat{S}_1$. Notice that under $\bf (H_c)$, $\E[\widehat{S}_1]=-\psi'(1)>0$ so $ \sum_{k≥1} \ee{-\widehat{S}_k}\in(0;\infty)$ almost surely. Letting for all $i≥1$
\begin{align}\label{eq:defaibi}
&a_i:=C_a^{-1} \E\left[ \frac{\Big( \sum_{k≥1} \ee{-\widehat{S}_k} \Big)^{i-1}}{\Big(1+\sum_{k≥1} \ee{-\widehat{S}_k} \Big)^{i+1}} \right],\nonumber\\
&b_i:=C_a i,
\end{align}
with $C_a:=\E\left[\big(1+\sum_{k≥1} \ee{-\widehat{S}_k}\big)^{-1} \right]$, then $(a_i)_{i≥1}$ and $(b_i)_{i≥1}$ are respectively left and right eigenvectors of the matrix $(m_{i,j})_{i,j≥1}$ associated with the eigenvalue $1$ (this is proved in Lemma~6.2 of~\cite{aidekon-raphelis}). We normalised $a$ and $b$ so that $\sum_i a_i=1$ and $\sum_i a_i b_i=1$ for convenience in the sequel, and for the notation to be coherent with~\cite{raphelis}. 

\subsection{Change of measure on $(\Tb,\beta)$}\label{subsec:changemeasure}

We let for all $n≥1$, 
\begin{equation}\label{eq:defZn}
Z_n:=\sum_{|u|=n}\beta(u)
\end{equation}
be the {\it multitype additive martingale} of $(\Tb,\beta)$. If for $n≥1$ we let $\mathcal{F}_n$ be the sigma-algebra generated by $(u,\beta(u))_{u\in\Tb,|u|≤n}$, then for every $i≥1$, $(Z_n/i)_{n≥0}$ is an $\cal{F}_n$-martingale under $\P_i$ of expectation $1$ (this is due to the fact that $(b_i)_{i≥1}$ is a right eigenvector of $(m_{i,j})_{i,j≥1}$ associated with $1$). 

Let us introduce a new law $\Ph_i$ on marked trees with distinguished path $(\Tb,\beta,(w_n)_{n≥0})$ (potentially on an enriched probability space). Let $\widehat{\bzeta}=(\widehat{\zeta}_i)_{i≥1}$ be a probability law on $\bigsqcup_{n\in\N}{\N^*}^n$ of Radon-Nikodym derivative $Z_1$ with respect to $\bzeta$; that is for any $i≥1$, if $X\sim \zeta_i$, then $\widehat{X}\sim\widehat{\zeta}_i$ if and only if for any bounded real-valued function $f$ on $\bigsqcup_{n\in\N}{\N^*}^n$, $\E[f(\widehat{X})]=\E[\sum_{1≤k≤|X|}X_kf(X)]$ (where we denoted by $|X|$ the length of $X$, and by $X_k$ its components). For every $i≥1$, we construct $(\Tb,\beta,(w_n)_{n≥0})$ under $\Ph_i$ by induction as follows: 

\begin{itemize}
\item {\bf Initialisation} \\
Generation $0$ of $\Tb$ is only made up of the root $\root$ of given type $\beta(\root)=i$. We set $w_0=\root$. 
\item {\bf Induction} Let $n≥0$. Suppose that the tree up to generation $n$ and $(w_k)_{k≤n}$ have been built. The vertex $w_n$ has progeny according to $\hzeta_{\beta(w_n)}$. Other vertices $u$ of generation $n$ have progeny according to $\zeta_{\beta(u)}$. Then, choose a vertex at random among children $u$ of $w_{n}$, each with probability $\beta(u)/\Big(\sum_{\parent{v}=w_{n}} \beta(v)\Big)$ and set $w_{n+1}$ as this vertex. 
\end{itemize}

\noindent We denote by $\Eh_i$ the expectation associated with $\Ph_i$. The following proposition is easily deduced from~\cite{kurtz-lyons-pemantle-peres-multitype}: 

\begin{proposition} {\bf~\cite{kurtz-lyons-pemantle-peres-multitype}}\label{prop:markovtype}
Let $i≥1$. We have the following links between $\P_i$ and $\Ph_i$: 
\begin{itemize}
\item[(i)] Recall that for every $n≥0$, $\cal{F}_n=\sigma((u,\beta(u))_{u\in\Tb, |u|≤n})$. The restricted measure $\Ph_{i}|_{\cal{F}_n}$ is absolutely continuous with respect to $\P_{i}|_{\cal{F}_n}$ and is such that
\begin{equation*}
 \frac{\d \Ph_{i}|_{\cal{F}_n}}{\d\P_{i}|_{\cal{F}_n}}=\frac{1}{i}Z_n. 
\end{equation*}
\item[(ii)] Conditionally on $\cal{F}_n$, for all $u\in\Tb$ such that $|u|=n$, 
\begin{equation*}
\Ph_{i}\Big( w_n=u \, |\, \mathcal{F}_n \Big)=\frac{\beta(u)}{Z_n}
\end{equation*}
\item[(iii)] Under $\Ph$, the process $(\beta(w_k))_{k≥0}$ is a Markov chain on $\N$ with initial state $i$, and with transition probabilities $(\ph_{i,j})_{i,j≥1}$ where for all $i,j≥1$, 
\begin{equation}\label{eq:pij}
\ph_{i,j}:=\frac{b_i}{b_j}m_{i,j}=\binom{i+j-1}{i}\Eb\Big[ \sum_{|u|=1} \frac{\ee{-jV(u)}}{(1+\ee{-V(u)})^{i+j}} \Big]. 
\end{equation}
\end{itemize}
\end{proposition}

\noindent Notice that the Markov chain $(\beta(w_k))_{k≥0}$ introduced in $(iii)$ admits an invariant measure $(\pi_i)_{i≥1}$ where for all $i≥1$, 
\begin{equation}\label{eq:defpi}
\pi_i=a_i\times b_i= i\E\left[ \frac{\Big( \sum_{k≥1} \ee{-\widehat{S}_k} \Big)^{i-1}}{\Big(1+\sum_{k≥1} \ee{-\widehat{S}_k} \Big)^{i+1}} \right], 
\end{equation}
and that this measure is of total mass $1$ (using Tonelli's theorem and the formula $\sum_{i≥1} i (\frac{x}{1+x})^i=x(1+x)$ for $x>0$), thus making $(\beta(w_k))_{k≥0}$ a positive recurrent Markov chain. Proposition~\ref{prop:markovtype} yields the \textit{multitype many-to-one lemma}: 
\begin{lemma}\label{lemma:many-to-one-multitype}
For all $n \in \Ns$, let $g:{\N}^n \to \R_+$ be a positive measurable function and $X_n$ a positive $\cal{F}_n$-measurable random variable, then
\begin{equation*}
\E_{i}\Big[\sum_{|u|=n}\beta(u)g(\beta(u_1),\beta(u_2),\dots,\beta(u_n)) X_n \Big]=i\Eh_i\Big[g(\beta(w_1),\beta(w_2),\dots,\beta(w_n)) X_n \Big]. 
\end{equation*}
\end{lemma}

\subsection{Understanding the law of $(\Tb,\beta)$ under the new measure $\Ph$}\label{subsec:lawtb}

The idea is that $(\Tb,\beta,(w_k)_{k≥0})$ under $\Ph$ can actually be seen as the natural trace of a series of random walks on a certain tree with spine. Let us begin by introducing this tree.

Recall that $N$ is the point process giving the offspring distribution of $\V$. Let us consider $\widehat{N}$ a point process defined as follows: for any bounded measurable function $f:\bigsqcup_{n\in\N\cup\{\N\}}{\R}^n\to\R$ we have $\Eb[f(\widehat{N})]=\Eb[\sum_{u\in N}\ee{-V(u)}f(N)]$. Hypothesis $\bf (H_c)$ ensures that this defines a probability law indeed. Let us introduce a law $\Pbh$ on $(\T,V)$ to which is added a spine $(\tilde{w}_k)_{k≥0}$ (potentially on an enriched probability space). We build $(\T,V,(\tilde{w}_k)_{k≥0})$ under $\Pbh$ by induction as follows: 
\begin{itemize}
\item {\bf Initialisation} \\
Generation $0$ of $\T$ is only made up of the root $\root$ and we set $V(\root)=0$. We set $\tilde{w}_0=\root$. 
\item {\bf Induction} Let $n≥0$. Suppose that the tree up to generation $n$ and $(\tilde{w}_k)_{k≤n}$ have been built. The vertex $\tilde{w}_n$ has progeny according to $\widehat{N}$ translated by $V(\tilde{w}_n)$. Other vertices $u$ of generation $n$ have progeny according to $N$ translated by $V(u)$. Then, choose a vertex at random among children $u$ of $\tilde{w}_{n}$, each with probability $\ee{-V(u)}/\Big(\sum_{\parent{v}=\tilde{w}_n}\ee{-V(v)}\Big)$ and set $\tilde{w}_{n+1}$ as this vertex. 
\end{itemize}
For every $k≥0$, we let
\begin{equation}\label{eq:defGk}
\cal{G}_k:=\sigma\Big((u,V(u))_{u\in\T,|u|≤k}\Big)\qquad\text{and}\qquad \cal{G}:=\bigvee_{k≥1}\cal{G}_k
\end{equation}
be the sigma-algebras generated by the environment, and
\begin{equation*}
W_k:=\sum_{|u|=k}\ee{-V(u)}
\end{equation*}
be the \textit{additive martingale} of the environment. Let $W_\infty:=\lim_{k→∞} W_k$ be its a.s.\ limit. According to~\cite{lyons}, under $\bf (H_c)$, $W_\infty$ is positive $\Pb^*$-almost surely.

We get the following proposition from~\cite{lyons}: 
\begin{proposition}\label{prop:lyons}{\bf \cite{lyons}} Let $n≥0$ and $u\in\T$ such that $|u|=n$. We have \begin{equation*}
\Pbh(\tilde{w}_n=u\mid \cal{G}_n)=\frac{\ee{-V(u)}}{W_n}. 
\end{equation*}
Moreover, the process $(V(\tilde{w}_k))_{k≥0}$ under $\Pbh$ has the same distribution as that of the random walk $(\widehat{S}_k)_{k≥0}$ introduced in~Subsection~\ref{subsec:multitype}. 
\end{proposition}
From this proposition we get the many-to-one lemma for the environment, which can be proved by induction: 
\begin{lemma}\label{lemma:many-to-one-environnement}
Let for every $n≥0$, $g:\R^n\to\R_+$ be a positive function. Then for any $k≥0$, and for any positive $\cal{G}_k$-measurable random variable $X_k$, 
\begin{equation*}
\Eb[\sum_{|u|=k}\ee{-V(u)}g(V(u_1),\dots,V(u_k))X_k]=\Ebh[g(V(\tilde{w}_1),\dots,V(\tilde{w}_k))X_k]=\Ebh[g(\widehat{S}_1,\dots,\widehat{S}_k)X_k], 
\end{equation*}
where $(\widehat{S}_k)_{k≥0}$ is the random walk with step distribution $\widehat{S}_1$ introduced in Subsection~\ref{subsec:multitype}. 
\end{lemma}

Let us now consider for each vertex of the spine $\tilde{w}_i$ two \iid truncated nearest-neighbour random walks with same law $(X^{1,\tilde{w}_i}_n)_{n≥0}$ and $({X}^{2,\tilde{w}_i}_n)_{n≥0}$, each defined as follows. It starts on $\tilde{w}_i$. If it is on a vertex $u\in\T$, then for each vertex $v$ child of $u$, it will jump to $v$ with probability $\frac{\ee{-V(v)}}{\ee{-V(u)}+\sum_{\parent{z}=u}\ee{-V(z)}}$ and towards $\parent{u}$ with probability $\frac{\ee{-V(u)}}{\ee{-V(u)}+\sum_{\parent{z}=u}\ee{-V(z)}}$. If it reaches $\tilde{w}_{i-1}$, then it is killed instantly. 

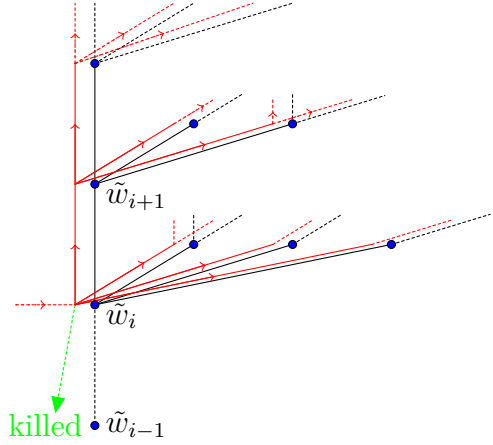
\begin{figure}[H]
\begin{tikzpicture}[line cap=round,line join=round,>=triangle 45,x=1.3cm,y=0.8cm]
\clip(2.0,-1.5) rectangle (13.0,7.0);

\draw (8.8,1.0)-- (8.8,3.0);
\draw [dash pattern=on 1pt off 1pt] (8.8,1.0)-- (8.8,-1.0);
\draw (8.8,3.0)-- (8.8,5.0);
\draw (8.8,1.0)-- (11.8,2.0);
\draw (8.8,1.0)-- (10.8,2.0);
\draw (8.8,1.0)-- (9.8,2.0);
\draw (8.8,3.0)-- (9.8,4.0);
\draw (8.8,3.0)-- (10.8,4.0);
\draw [dash pattern=on 1pt off 1pt] (9.8,2.0)-- (9.8,2.5);
\draw [dash pattern=on 1pt off 1pt] (9.8,2.0)-- (10.3,2.5);
\draw [dash pattern=on 1pt off 1pt] (10.8,2.0)-- (11.3,2.5);
\draw [dash pattern=on 1pt off 1pt] (11.8,2.0)-- (12.8,2.5);
\draw [dash pattern=on 1pt off 1pt] (8.8,5.0)-- (8.8,6.0);
\draw [dash pattern=on 1pt off 1pt] (8.8,5.0)-- (9.8,6.0);
\draw [dash pattern=on 1pt off 1pt] (8.8,5.0)-- (10.8,6.0);
\draw [dash pattern=on 1pt off 1pt] (9.8,4.0)-- (10.3,4.5);
\draw [dash pattern=on 1pt off 1pt] (10.8,4.0)-- (10.8,4.5);
\draw [dash pattern=on 1pt off 1pt] (10.8,4.0)-- (11.733857473145854,4.4678642878968935);
\draw [line width=0.4pt,color=red] (8.6,1.0)-- (8.6,3.0);
\draw [line width=0.4pt,color=red] (8.6,3.0)-- (9.6,4.0);
\draw [line width=0.4pt,color=red] (8.6,3.0)-- (10.6,4.0);
\draw [line width=0.4pt,color=red] (8.6,3.0)-- (8.6,5.0);
\draw [line width=0.4pt,dash pattern=on 1pt off 1pt,color=red] (8.6,5.0)-- (8.6,6.0);
\draw [line width=0.4pt,dash pattern=on 1pt off 1pt,color=red] (8.6,5.0)-- (9.6,6.0);
\draw [line width=0.4pt,dash pattern=on 1pt off 1pt,color=red] (8.6,5.0)-- (10.4,6.0);
\draw [line width=0.4pt,dash pattern=on 1pt off 1pt,color=red] (8.0,1.0)-- (8.6,1.0);
\draw [line width=0.4pt,dash pattern=on 1pt off 1pt,color=red] (9.6,4.0)-- (10.0,4.4);
\draw [line width=0.4pt,dash pattern=on 1pt off 1pt,color=red] (10.6,4.0)-- (10.6,4.4);
\draw [line width=0.4pt,dash pattern=on 1pt off 1pt,color=red] (10.6,4.0)-- (11.4,4.4);
\draw (8.8,3.2) node[anchor=north west] {$\tilde{w}_{i+1}$};
\draw (8.8,1.2) node[anchor=north west] {$\tilde{w}_i$};
\draw (8.8,-1)  node[anchor=west] {$\tilde{w}_{i-1}$};
\draw [line width=0.4pt,dash pattern=on 1pt off 1pt,color=green,->] (8.6,1)-- (8.4,-0.8);
\draw [color=green] (8.3,-1) node {killed};
\draw [line width=0.4pt,color=red] (8.6,1.0)-- (9.6,2.0);
\draw [line width=0.4pt,dash pattern=on 1pt off 1pt,color=red] (9.6,2.0)-- (9.6,2.4);
\draw [line width=0.4pt,dash pattern=on 1pt off 1pt,color=red] (9.6,2.0)-- (10.0,2.4);
\draw [line width=0.4pt,color=red] (8.6,1.0)-- (10.6,2.0);
\draw [line width=0.4pt,color=red] (8.6,1.0)-- (11.6,2.0);
\draw [line width=0.4pt,dash pattern=on 1pt off 1pt,color=red] (11.6,2.0)-- (12.4,2.4);
\draw [line width=0.4pt,dash pattern=on 1pt off 1pt,color=red] (10.6,2.0)-- (11.0,2.4);
\draw [line width=0.4pt,color=red,-to] (8.6,1.0)-- (8.6,2.0);
\draw [line width=0.4pt,color=red,-to] (8.6,1.0)-- (9.3,1.7);
\draw [line width=0.4pt,color=red,-to] (8.6,1.0)-- (9.92,1.66);
\draw [line width=0.4pt,color=red,-to] (8.6,1.0)-- (10.01,1.47);
\draw [line width=0.4pt,color=red,-to] (8.6,3)-- (8.6,4.0);
\draw [line width=0.4pt,color=red,-to] (8.6,3)-- (9.3,3.7);
\draw [line width=0.4pt,color=red,-to] (8.6,3)-- (9.92,3.66);
\draw [line width=0.4pt,dotted,color=red,-to] (9.6,4.0)-- (9.9,4.3);
\draw [line width=0.4pt,dotted,color=red,-to] (10.6,4)-- (10.6,4.2);
\draw [line width=0.4pt,dotted,color=red,-to] (10.6,4)-- (11.0,4.2);
\draw [line width=0.4pt,dotted,color=red,-to] (8.6,5.0)-- (8.6,5.5);
\draw [line width=0.4pt,dotted,color=red,-to] (8.6,5.0)-- (9.1,5.5);
\draw [line width=0.4pt,dotted,color=red,-to] (8.6,5.0)-- (9.5,5.5);
\draw [line width=0.4pt,dotted,color=red,-to] (8.0,1.0)-- (8.3,1.0);
\begin{scriptsize}
\draw [fill=blue] (8.8,1.0) circle (1.5pt);
\draw [fill=blue] (8.8,3.0) circle (1.5pt);
\draw [fill=blue] (8.8,-1.0) circle (1.5pt);
\draw [fill=blue] (8.8,5.0) circle (1.5pt);
\draw [fill=blue] (11.8,2.0) circle (1.5pt);
\draw [fill=blue] (10.8,2.0) circle (1.5pt);
\draw [fill=blue] (9.8,2.0) circle (1.5pt);
\draw [fill=blue] (9.8,4.0) circle (1.5pt);
\draw [fill=blue] (10.8,4.0) circle (1.5pt);
\end{scriptsize}
\end{tikzpicture}
\caption{Possible paths for the walks $(X^{1,\tilde{w}_i}_n)_{n≥0}$ and $({X}^{2,\tilde{w}_i}_n)_{n≥0}$. }
\label{f:X1X2}
\end{figure}
Now for each $u\in\T$ we let
\begin{equation*}
\tyt_i^1(u):= \#\{n≥0\st X^{1,\tilde{w}_i}_n=\parent{u},X^{1,\tilde{w}_i}_{n+1}=u\}\qquad\text{and}\qquad \tyt_i^2(u):= \#\{n≥0\st {X}^{2,\tilde{w}_i}_n=\parent{u},{X}^{2,\tilde{w}_i}_{n+1}=u\}
\end{equation*}
be the edge local times on $u$ of the walks launched on $\tilde{w}_i$, and we let
\begin{equation*}
\tyt^1(u):=\sum_{i=1}^{\infty} \tyt_i^1(u)\qquad\text{and}\qquad \tyt^2(u):=\sum_{i=1}^{\infty} \tyt_i^2(u)
\end{equation*}
be the sum of the edge local times of all the walks. Finally, we let 
\begin{equation}\label{eq:defbetatilde}
\tyt(u):=\tyt^1(u)+\tyt^2(u)+\1{u\in (\tilde{w}_k)_{k≥0}}. 
\end{equation}
In other words, for every vertex $u\in\T$ not on the spine, $\tyt(u)$ corresponds to the sum of the edge local times of the walks launched on each vertex of the spine below; moreover, we add $1$ for each vertex of the spine. 
We let $\widetilde{\Tb}:=\{u\in\T,\tyt(u)>0\}$

\begin{proposition}\label{prop:lawbetatilde}
Under $\Ph_1$, $(\Tb,\beta,(w_k)_{k≥0})$ has same law as $(\widetilde{\Tb},\tyt,(\tilde{w}_k)_{k≥0})$ (with $(\T,(\tilde{w}_k)_{k≥0})$ built under $\Pbh$). 
\end{proposition} 
\begin{proof}
Notice that $(\widetilde{\Tb},\tyt,(\tilde{w}_k)_{k≥0})$ is a multitype tree with spine, and can be built by induction. Indeed, the same reasoning as that in the proof of~Lemma~\ref{lemma:lawbeta} yields the two following facts. 

\begin{itemize}
\item Let $u\in\widetilde{\Tb}$ be a vertex not on the spine. Conditionally on $\tyt(u)$, the progeny of $u$ has the law of a negative multinomial random variable of parameter $\big(\tyt(u),(1/(1+\ee{-V(v)}))_{v\in N^u}\big)$, (where $N^u\sim N$ is a random variable independent of everything else) independent of that of other vertices.
\item Let $i≥0$. Conditionally on $\tyt(\tilde{w}_{i})$, the progeny of $\tilde{w}_i$ has the law of a negative multinomial random variable of parameters $\tyt(\tilde{w}_i)+1$. This is because $\tyt(\tilde{w}_i)-1$ stands for the edge local time of the walks launched below $\tilde{w}_i$, to which we have to add the contribution of the two walks $(X^{1,\tilde{w}_i}_k)_{k≥0}$ and $({X}^{2,\tilde{w}_i}_k)_{k≥0}$) and $\big(1/(1+\ee{-V(v)})\big)_{v\in\widehat{N}^{\tilde{w}_i}}$ (where $\widehat{N}^{\tilde{w}_i}\sim \widehat{N}$ is independent of everything else). After choosing $\tilde{w}_{i+1}$ proportionally to $\widehat{N}^{\tilde{w}_i}$, we add $1$ to its type. 
\end{itemize} 
Since the offspring distribution outside the spine is the same as that of $(\Tb,\beta)$ under $\Ph$ (edge local times of a random walk on a non-biased tree), we only need to focus on the offspring distribution of vertices of the spine. We let $\mathcal{U}:=\bigsqcup_{n≥0}\N^{(n)}$ be the set of all potential vertices (following Neveu's notation~\cite{neveu}). For every $(s_u)_{u\in\mathcal{U}}\in [0;1]^{\mathcal{U}}$ and $t\in [0;1]$, we have
\begin{align*}
\Eh_i\Big[ t^{\beta(w_1)} \prod_{|u|=1,u\neq w_1}s_u^{\beta(u)}\Big]&=\frac{1}{i}\E_i\Big[\sum_{|v|=1} \beta(v) t^{\beta(v)}\prod_{|u|=1,u\neq v}s_u^{\beta(u)} \Big]\\
&=\frac{1}{i}\Eb\Big[\sum_{|v|=1}t\frac{\d}{\d t}\Big(\E_i^\V\Big[ t^{\beta(v)}\prod_{|u|=1,u\neq v}s_u^{\beta(u)}\Big]\Big) \Big].  
\end{align*}
The first equality is due to the multitype many-to-one lemma (Lemma~\ref{lemma:many-to-one-multitype}). With a slight abuse of notation, we wrote $\P_i^\V$ for the quenched law of $\beta$ after $i$ excursions from the root. 

Now according to Lemma~\ref{lemma:lawbeta}, $(\beta(u))_{|u|=1}$ is a negative multinomial random variable of parameter $\big(i,(p_{\root,u})_{|u|=1}\big)$ under $\P_i^\V$. Therefore its generating function under $\P_i^\V$ is $(\frac{p_{\root,\parent{\root}}}{1-\sum_{|u|=1}p_{\root,u} s_u})^i$. Using this expression with $t$ instead of $s_v$, we get
\begin{align*}
\Eh_i\Big[ t^{\beta(w_1)} \prod_{|u|=1,u\neq w_1}s_u^{\beta(u)}\Big]&=\frac{1}{i}\Eb\Big[\sum_{|v|=1}t\frac{\d}{\d t}\Big(\frac{p_{\root,\parent{\root}}}{1-p_{\root,v}t-\sum_{\substack{|u|=1 \\ u\neq v}}p_{\root,u} s_u}\Big)^i\Big]\\
&=\Eb\Big[\sum_{|v|=1}\ee{-V(v)}t\Big(\frac{p_{\root,\parent{\root}}}{1-p_{\root,v}t-\sum_{\substack{|u|=1 \\ u\neq v}}p_{\root,u} s_u}\Big)^{i+1}\Big]\\
&=\Ebh\Big[t\Big(\frac{p_{\root,\parent{\root}}}{1-p_{\root,w_1}t-\sum_{\substack{|u|=1 \\ u\neq w_1}}p_{\root,u} s_u}\Big)^{i+1}\Big]. 
\end{align*}
We used the fact that $p_{\root,v}=\ee{-V(v)}p_{\root,\parent{\root}}$ in the last but one equality, and the many-to-one lemma (Lemma~\ref{lemma:many-to-one-environnement}) in the last equality. Now this last generating function is precisely the annealed version (with respect to $\Pbh$) of a negative multinomial random variable of parameter $(i+1,(p_{\root,u})_{|u|=1})$, to which $1$ was added on the vertex $\tilde{w}_1$. But since the type of $\tilde{w}_0$ was $i$, it means that it was reached $i-1$ times by the walks below (since $1$ was added), so two random walks were added, exactly as described above. 
\end{proof}

The idea of Proposition~5 sums up in the following diagram. \\

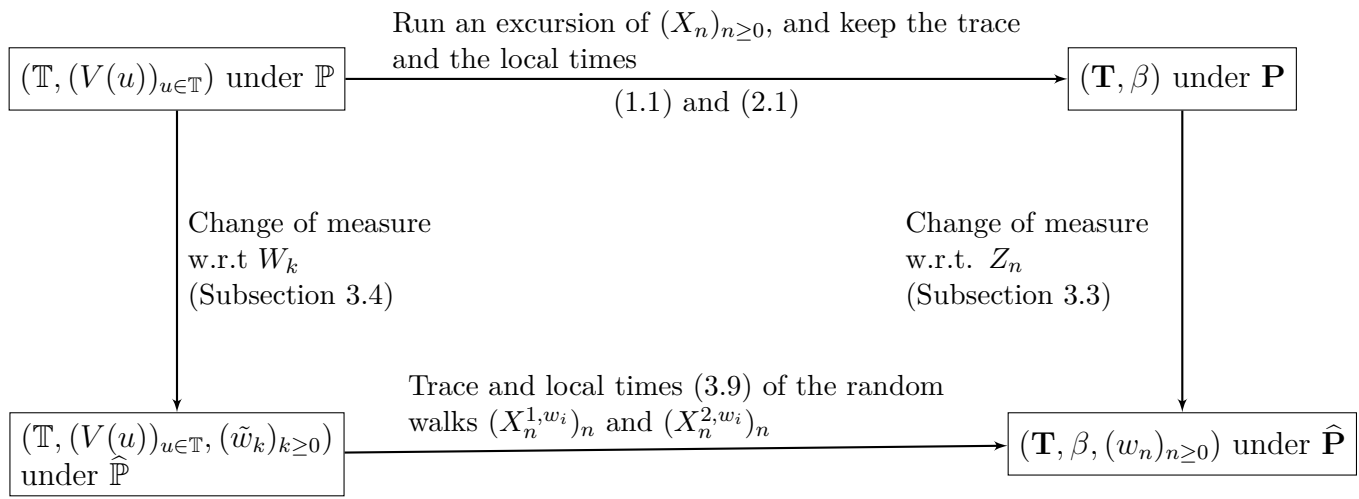
\tikzstyle{block} = [rectangle, draw, text centered, minimum height=2em]
\tikzstyle{line} = [draw, -latex', thick]
\tikzstyle{line2} = [draw, latex'-latex',thick]
\begin{figure}[H]
\begin{tikzpicture}[node distance=1cm, auto]
\node (init) {};
\node [block] (A) {$(\T,(V(u))_{u\in\T})$ under $\Pb$};
\node [block, right=9.5cm of A] (B) {$(\Tb,\beta)$ under $\P$};
\node [block,align=left, below=4cm of A] (C) {$(\T,(V(u))_{u\in\T},(\tilde{w}_k)_{k≥0})$ \\ under $\Pbh$};
\node [block, below=4cm of B] (D) {$(\Tb,\beta,(w_n)_{n\geq 0})$ under $\Ph$};
\begin{small}
\path [line] (A) -- node [midway,below ] {\eqref{eq:probatrans} and~\eqref{eq:defbeta}} node [midway,above,align=left] {Run an excursion of $(X_n)_{n≥0}$, and keep the trace \\ and the local times} (B);
\path [line] (B) -- node [midway,left,text width=3.5cm,align=left ] {Change of measure w.r.t. $Z_n$\\ (Subsection~\ref{subsec:changemeasure})} (D);
\path [line] (A) -- node [midway,right,text width=3.5cm,align=left ] {Change of measure w.r.t $W_k$\\ (Subsection~\ref{subsec:lawtb})} (C);
\path [line] (C) -- node [midway,above,align=left ] {Trace and local times~\eqref{eq:defbetatilde} of the random \\ walks $(X_n^{1,w_i})_n$ and $(X_n^{2,w_i})_n$} (D);
\end{small}
\end{tikzpicture}
\end{figure}

From now on we will therefore consider $(\Tb,\beta,(w_k)_{k≥0})$ under $\Ph_1$ as the trace of all the random walks on $(\V,(\tilde{w}_k)_{k≥0})$ under $\Pbh$, and $\beta$ will also stand for their edge local times $\widetilde{\beta}$.

\section{Proof of Proposition~\ref{prop:regvar}}\label{sec:technique}
Before tackling the proof of Proposition~\ref{prop:regvar}, that is ensuring that $\Fb^R$ and $\Fb^X$ satisfy $\bf (H_{\ell})$, it is very important to notice that by construction:
\begin{itemize}
\item For both $\Fb^R$ and $\Fb^X$, the law of the offspring distribution of vertices of type~$1$, denoted by $\nu^1$ in $(\bf H_\ell)$, has the law of $L^1$ (the cardinal of the optional line $\cal{L}^1$, introduced in Definition~\ref{def:optline}) under $\P_1$.
\item[$\bullet $] The law of the total offspring distribution of $\Fb^R$ (vertices of type $0$ and $1$ denoted by $\nu$ in $\bf (H_{\ell})$), has the law of $B^1$ (the cardinal of the optional line $\cal{B}^1$, introduced in Definition~\ref{def:optline}) under $\P_1$.
\item[$\bullet $] The law of the total offspring distribution of $\Fb^X$ (vertices of type $0$ and $1$) has the law of $\sum_{u\in\B^1}2\beta(u)$ (the total time spent by the walk in $\cal{B}^1$ in one excursion) under $\P_1$.
\end{itemize}
Therefore, proving that hypothesis ${\bf (H_{\ell})}(i)$ is satisfied boils down to proving that $\E_1[L^1]=1$, which will be done in Lemma~\ref{lemma:Hl1}. Then, to ensure that hypothesis ${\bf (H_{\ell})}(ii)$ is satisfied it is enough to prove that there exists an $\eps>0$ such that $\E_1[(\sum_{u\in\cal{B}^1}\beta(u))^{1+\eps}]<\infty$, (we would also have to prove that $\E_1[(B^1)^{1+\eps}]<\infty$ but as $\sum_{u\in\cal{B}^1}\beta(u)≥B^1$ this will be automatic), which will be done in Lemma~\ref{lemma:momB1}. 

Still by construction, notice that to each vertex $u'$ of the first generation of $\Fb^R$ corresponds a vertex $u$ of $\cal{B}^1$, and that $\ell(u')=|u|$. Moreover, the set of vertices of the first generation of $\Fb^X$ matches that of the first generation of $\Fb^R$ after having replicated each vertex $u$ a number $2\beta(u)-1$ of times. Therefore, denoting by $\nu^R$ (resp.\ $\nu^X$) the law of the first generation of $\Fb^R$ (resp.\ $\Fb^X$), we have for any $r>1$, 
\begin{equation}\label{eq:momexpnuRnuX}
\E\big[\sum_{u\in\nu^R}r^{\ell(u)}\big]=\E_1\big[\sum_{u\in\cal{B}^1}r^{|u|}\big]≤\E_1\big[\sum_{u\in\cal{B}^1}2\beta(u)r^{|u|}\big]=\E\big[\sum_{u\in\nu^X}r^{\ell(u)}\big]. 
\end{equation}
Thus, it will be enough to show that there exists an $r>1$ such that $\E_1\big[\sum_{u\in\cal{B}^1}\beta(u)r^{|u|}\big]<\infty$ to prove that hypothesis ${\bf (H_{\ell})}(iv)$ is satisfied, and this is what we will do with Lemma~\ref{lemma:lyapunov}. 

We will also compute the constants given in Proposition~\ref{prop:regvar} in Lemma~\ref{lemma:constants} \\

Finally, hypothesis ${\bf (H_{\ell})}(iii)$ will be satisfied if we are able to show that there exists a positive constant $K^\star$ such that $\P_1\big(L^1>x\big)\sim K^\star x^{-\kappa}$. This will be the subject of the whole Subsection~\ref{subsec:cvqueue1}, and it will be finally proved in Proposition~\ref{prop:xkappa}. This is actually the main contribution of this paper when compared to~\cite{aidekon-raphelis}. Indeed, in the latter, we only had to show the finiteness of the second moment of $\nu^1$, and this could be obtained by quite straightforward backbone decomposition techniques; whereas here proving the regular variation of the tails of $\nu^1$ will appear to be quite technical, and will require the fine understanding of $\cal{L}^1$. 

\subsection{Hypotheses ${\bf (H_{\ell})}$ $(i)$, $(ii)$ and $(iv)$}\label{subsec:Hl}

Now recall that we assume that hypotheses $\bf (H_c)$ and $\bf (H_\kappa)$ are satisfied for a certain $\kappa\in(1;2]$. Let 
\begin{equation*}
\tauh_1:=\min\{k≥1\st \beta(w_k)=1\}
\end{equation*}
be the first non-null hitting time of $1$ by the Markov chain $(\beta(w_k))_{k≥0}$ . 
\begin{lemma}\label{lemma:Hl1}
For any $i≥1$, we have, 
\begin{equation}\label{eq:majLj}
\E_i[L^1]=i. 
\end{equation}
In particular, $\E_1[L^1]=1$, and the reproduction law of both forests $\Fb^R$ and $\Fb^X$ satisfies condition ${\bf (H_{\ell})}(i)$.  
\end{lemma}
\begin{proof}
For any $i≥1$, we have
\begin{align*}
\E_i\Big[L^1\Big]&=\E_i\Big[\sum_{u\in\cal{L}^1}1\Big]=\E_i\Big[\sum_{u\in\cal{L}^1}\beta(u)\Big]\nonumber\\
&=\sum_{k≥1}\E_i\Big[\sum_{|u|=1}\beta(u)\1{\beta(u_1),\dots,\beta(u_{k-1})\neq 1\text{ and }\beta(u)=1}\Big]\nonumber\\
&=\sum_{k≥1}i\Eh_i\Big[\1{\beta(w_1),\dots,\beta(w_{k-1})\neq 1\text{ and }\beta(w_k)=1}\Big]=i\sum_{k≥1}\Eh_1\Big[\1{\tauh_1=k}\Big]=i, 
\end{align*}
the Markov chain $(\beta(w_k))_{k≥0}$ being recurrent, and where we used the multitype many-to-one lemma (Lemma~\ref{lemma:many-to-one-multitype}) between the last two lines. The case $i=1$ reads $\E_1[L^1]=1$, which proves that ${\bf (H_{\ell})} (i)$ is satisfied by $\Fb^R$ and $\Fb^X$.
\end{proof}

Let us prove now a Lemma that allows us to control the Markov chain $(\beta(w_k))_{k≥0}$ up to a hitting time, and which will yield that $\Fb^R$ and $\Fb^X$ also satisfy hypothesis ${\bf (H_{\ell})} (iv)$. 
\begin{lemma}\label{lemma:lyapunov}
For every $\alpha\in(0;\kappa-1)$, there exists $C_{\alpha}>0$ such that for any $r\in(1;\psi(1+\alpha)^{-1})$ and $i≥1$, 
\begin{equation}\label{eq:majtaualpha}
\Eh_i\Big[ \sum_{k=1}^{\tauh_1} (\beta(w_k))^\alpha r^k \Big]≤C_{\alpha} i^{\alpha}. 
\end{equation}
As consequences, first there exists a constant $C_1>0$ such that for any $p>0$
\begin{equation}\label{eq:taulog}
\Eh_i[{\tauh_1}^p]\leq C_1\ln^p(1+i),  
\end{equation}
and second, the laws of $\Fb^R$ and $\Fb^X$ satisfy hypothesis ${\bf (H_\ell)} (iv)$. 
\end{lemma}

\begin{proof}
Let $\alpha\in(0;\kappa-1)$, and let us set for every $i≥1$, $F(i):=\frac{\Gamma(i+1+\alpha)}{\Gamma(i+1)}$. As proved in the appendix~(Lemma~\ref{lemma:applyapunov}), there exists $d\in (0;1)$ such that for any $i>i_0$ large enough, 
\begin{equation*}
\sum_{j≥1} \ph_{i,j} F(j)≤ d F(i), 
\end{equation*}
that is $F$ is a \textit{Lyapounov function} satisfying condition $(V4)$ introduced p.371 of~\cite{meyn-tweedie}. Thus, Theorem~15.3.3 of~\cite{meyn-tweedie} ensures that there exists a constant $C_F>0$ (depending on $d$, and then on $F$) such that for any $r\in(1;d^{-1})$ and for any $i≥1$, 
\begin{equation}\label{eq:majSigmatauV}
\Eh_i\Big[ \sum_{k=1}^{\tauh_1} F(\beta(w_k))r^k \Big]≤C_F F(i). 
\end{equation}
The function $F$ being greater than $1$, and equivalent to $i^{\alpha}$ as $i\to\infty$, (\ref{eq:majSigmatauV}) yields~(\ref{eq:majtaualpha}). 

Now the first consequence is just Jensen's inequality applied to~\eqref{eq:majtaualpha} with $x\mapsto\ln^p(1+x)$. 

Let us deal with the second consequence. We want to show that there exists $r>1$ such that $\E_1[\sum_{u\in\cal{B}_1}\beta(u)r^{|u|}]<\infty$.  Let $r\in(1;d^{-1})$, using the many-to-one lemma (Lemma~\ref{lemma:many-to-one-multitype}), we get 
\begin{align*}
\E[\sum_{u\in\nu^R}r^{\ell(u)}]≤\E[\sum_{u\in\nu^X}r^{\ell(u)}]=\E_1\Big[\sum_{u\in\cal{B}^1}2\beta(u)r^{|u|}\Big]&=\sum_{k≥1}\E_1\Big[\sum_{|u|=k}\beta(u)\1{\beta(u_1),\dots,\beta(u_{k-1})\neq 1}r^k\Big]\\
&=\sum_{k≥1}\Eh_1\big[\1{\beta(w_1),\dots,\beta(w_{k-1})\neq 1}r^k\big]=\Eh_1\Big[\sum_{k=1}^{\tauh_1}r^k\Big], 
\end{align*}
and~\eqref{eq:majtaualpha} ensures the finiteness of this last quantity. Hence by~\eqref{eq:momexpnuRnuX}, the laws of $\Fb^R$ and $\Fb^X$ satisfy hypothesis~${\bf (H_\ell)} (iv)$. 
\end{proof}

Let us now focus on hypothesis ${\bf (H_{\ell})} (ii)$. We are then interested in moments of order $1+\eps$ of $\sum_{u\in\cal{B}^1}\beta(u)$ under $\P_1$. For convenience, let us denote this quantity by $\widetilde{B}^1$. We have the following lemma: 
\begin{lemma}\label{lemma:momB1}
For any $\alpha\in(0;\kappa-1)$, $\eps>0$, there exists a constant $C'_{\alpha+\eps}\in(0;\infty)$ such that for any $i≥1$, 
\begin{equation*}
\E_i[(\widetilde{B}^1)^{1+\alpha}]≤C'_{\alpha+\eps}i^{1+\alpha+\eps}. 
\end{equation*}
As a consequence, the laws of $\Fb^R$ and $\Fb^X$ satisfy hypothesis ${\bf (H_{\ell})} (ii)$.
\end{lemma}
\begin{proof}

 Let us start by computing some estimates on the first moment of $\widetilde{B}^1$ under $\P_i$ for $i≥1$: 
\begin{align*}
 \E_i\Big[ {\widetilde{B}^1} \Big]&= \E_i\Big[ \sum_{u\in\Tb\backslash\{\root\}} \beta(u)\1{\beta(u_1),\beta(u_2),\dots,\beta(\parent{u})\neq 1} \Big] \\
&= \sum_{k≥1}\E_i\Big[ \sum_{|u|=k} \beta(u) \1{\beta(u_1),\dots,\beta(u_{k-1})\neq 1} \Big]\\
&= i\sum_{k≥1}\Eh_i\Big[ \1{\beta(w_1),\dots,\beta(w_{k-1})\neq 1} \Big]=i\Eh_i[\tauh_1], 
\end{align*}
where we used the many-to-one lemma (Lemma~\ref{lemma:many-to-one-multitype}). Now, equation~(\ref{eq:taulog}) yields 
\begin{equation}\label{eq:majEiBj}
\E_i\Big[ {\widetilde{B}^1} \Big]≤C_1 i\ln(1+i). 
\end{equation}
Let us now compute the $1+\alpha$\up{th} moment of $\widetilde{B}^1$ under $\P_i$. Discussing on the generation to which vertices of $\B^1$ belong, we get 
\begin{equation*}
\E_i[(\widetilde{B}^1)^{1+\alpha}]=\sum_{k≥1}\E_i\Big[\Big(\sum_{|u|=k} \beta(u)\1{u\in{\B^1}}\Big)\times(\widetilde{B}^1)^{\alpha}\Big]. 
\end{equation*}
For $k≥1$, let us focus on the general term of the sum. When conditioning on $\cal{F}_k$, it can be written as
\begin{equation*}
\E_i\Big[\Big(\sum_{|u|=k}\beta(u)\1{u\in{\B^1}}\Big)(\widetilde{B}^1)^{\alpha}\Big]=\E_i\Big[\Big(\sum_{|u|=k}\beta(u)\1{u\in{\B^1}}\Big)\E_i\big[(\widetilde{B}^1)^{\alpha} \mid \cal{F}_k \big]\Big]. 
\end{equation*}
Let us apply the many-to-one lemma (Lemma~\ref{lemma:many-to-one-multitype}) at generation $k$ to this expectation, with the setting $X_k=\E_i\Big[(\widetilde{B}^1)^\alpha \mid \cal{F}_k \Big]$ (which is $\cal{F}_k$-measurable); we get
\begin{equation}\label{eq:majbalphaEh}
\E_i\Big[\Big(\sum_{|u|=k}\beta(u)\1{u\in{\B^1}}\Big)(\widetilde{B}^1)^{\alpha}\Big]=i\Eh_i\Big[ \1{k≤\tauh_1} \E_i\big[(\widetilde{B}^1)^\alpha \mid \cal{F}_k \big] \Big]. 
\end{equation}
Decomposing $\B^1$ according to generations we have for any $k≤\tauh_1$,
\begin{equation*}
\widetilde{B}^1=\sum_{|u|<k}\beta(u)\1{u\in \cal{B}^1}+\sum_{|u|=k,u\neq w_k}\1{u\in\cal{B}^1}(\beta(u)+\widetilde{B}^1_u)+\beta(w_k)+\widetilde{B}^1_{(w_k)}. 
\end{equation*}
where for $u\in\Tb$ we denoted by $\widetilde{B}^1_u$ the quantity $\sum_{v\in \cal{B}^1_u}\beta(v)$, that we choose to be equal to $0$ if $\beta(u)=1$. The fact that $\alpha<1$ and equation~(\ref{eq:majEiBj}) yield for any $k≤\tauh_1$, 
\begin{align}\label{eq:B1toL1}
\E_i&\big[(\widetilde{B}^1)^\alpha\mid\cal{F}_k\big]≤\E_i\Big[\Big(\sum_{|u|<k}\beta(u)\1{u\in \cal{B}^1}+\sum_{|u|=k,u\neq w_k}\1{u\in\cal{B}^1}(\beta(u)+\widetilde{B}^1_u)+\beta(\beta(w_k))\Big)^\alpha\Big]+\E\big[\widetilde{B}^1_{\beta(w_k)}\mid \cal{F}_k\big]^\alpha \nonumber\\
&≤\E_i\Big[\Big(\sum_{|u|<k}\beta(u)\1{u\in \cal{B}^1}+\sum_{|u|=k,u\neq w_k}\1{u\in\cal{B}^1}(\beta(u)+\widetilde{B}^1_u)+\beta(w_k)\Big)^\alpha\Big]+\big(C_1\beta(w_k)\ln(1+\beta(w_k))\big)^\alpha.
\end{align}
Now notice that according to the construction of $\Tb$ in Subsection~\ref{subsec:multitype}, for $u\neq w_k$ the $\widetilde{B}^1_u$ have same law under $\Ph_i$ and $\P_i$. The previous equation can therefore be written as
\begin{align*}
\E_i\big[(\widetilde{B}^1)^\alpha\mid\cal{F}_k\big]&≤\Eh_i\Big[\Big(\sum_{|u|<k}\beta(u)\1{u\in \cal{B}^1}+\sum_{|u|=k,u\neq w_k}\1{u\in\cal{B}^1}(\beta(u)+\widetilde{B}^1_u)+\beta(\beta(w_k))\Big)^\alpha\Big]\\
&\phantom{≤}+\big(C_1\beta(w_k)\ln(1+\beta(w_k))\big)^\alpha \\
&≤\Eh_i\big[(\widetilde{B}^1)^{\alpha}\mid \cal{F}_k\big]+\big(C_1\beta(w_k)\ln(1+\beta(w_k))\big)^\alpha
\end{align*}

Plugging this in~(\ref{eq:majbalphaEh}) and summing over $k≥1$, we finally get a more convenient upper bound for the $(1+\alpha)$\up{th} moment: 
\begin{align}\label{eq:EiNjalpha}
\E_i\big[(\widetilde{B}^1)^{1+\alpha}\big]&≤i\Eh_i\Big[\sum_{k=1}^{\tauh_1}\Eh_i[(\widetilde{B}^1)^\alpha]\Big]+i\Eh_i\Big[\sum_{k=1}^{\tauh_1}\big(C_1\beta(w_k)\ln(1+\beta(w_k))\big)^\alpha\Big]\nonumber\\
&=
i\Eh_i\big[\tauh_1(\widetilde{B}^1)^{\alpha}\big]+i\Eh_i\Big[\sum_{k=1}^{\tauh_1}\big(C_1\beta(w_k)\ln(1+\beta(w_k))\big)^\alpha\Big]\nonumber\\
&≤i\Eh_i\big[\tauh_1(\widetilde{B}^1)^{\alpha}\big]+C_1^{\alpha}C_{\alpha+\eps} i^{1+\alpha+\eps}, 
\end{align}
where the last inequality was obtained thanks to~(\ref{eq:majtaualpha}) for any $\eps>0$. We therefore only have to bound the first quantity to prove the lemma. Now we get, thanks to Hölder's inequality, for any $\eps'>0$ (we choose $\eps'$ such that $\alpha':=\alpha(1+\eps')<\kappa-1$)
\begin{equation}\label{eq:B1holder}
i\Eh_i[\tauh_1(\widetilde{B}^1)^{\alpha}]≤i\Eh_i[\tauh_1^{1+1/\eps'}]^{\frac{\eps'}{1+\eps'}}\Eh_i[(\widetilde{B}^1)^{\alpha(1+\eps')}]^{\frac{1}{1+\eps'}}≤C_1 i\ln(1+i)\Eh_i[(\widetilde{B}^1)^{\alpha(1+\eps')}]^{\frac{1}{1+\eps'}}, 
\end{equation}
where we used equation~(\ref{eq:taulog}) in the last inequality. Now, computing this last quantity will require a more subtle decomposition . Under the biased law $\Ph$, $\B^1$ is made up of the spine below $w_{\tauh_1}$, together with the sets $\B_u^1$ for any sibling $u$ of the spine below $w_{\tauh_1}$ such that $\beta(u)\neq i$. That is, denoting by $\Omega(w_k)$ the set of siblings of $w_k$, and still for $u\in\Tb$ by $\widetilde{B}^1_u$ the quantity $\sum_{v\in \cal{B}^1_u}\beta(v)$ (that we choose to be equal to $0$ if $\beta(u)=1$),
\begin{align*}
\Eh_i[(\widetilde{B}^1)^{\alpha'}]&=\Eh_i\Big[\Big( \sum_{k=1}^{\tauh_1} \beta(w_k) + \sum_{k=1}^{\tauh_1}\sum_{u\in\Omega(w_k)}(\beta(u)+\widetilde{B}_u^1) \Big)^{\alpha'} \Big]\nonumber\\
&≤\Eh_i\Big[\sum_{k=1}^{\tauh_1} (\beta(w_k))^{\alpha'}\Big]+ \Eh_i\Big[\sum_{k=1}^{\tauh_1}\Big(\sum_{u\in\Omega(w_k)}(\beta(u)+\widetilde{B}_u^1)\Big)^{\alpha'}\Big]\nonumber\\
&≤C_{\alpha'}i^{\alpha'}+ \Eh_i\Big[\sum_{k=1}^{\tauh_1}\Big(\sum_{u\in\Omega(w_k)}\Eh\big[\beta(u)+\widetilde{B}_u^1\big|\sigma\big(\beta(v),v\in\bigcup_{1≤k≤\tauh_1}\Omega(w_k)\big)\big]\Big)^{\alpha'}\Big]
\end{align*}
after this decomposition along the spine (we used~(\ref{eq:majtaualpha}) and Jensen's inequality in the last inequality). Now~(\ref{eq:majEiBj}) yields that 
\begin{equation}\label{eq:EiNjalpha1}
\Eh_i[(\widetilde{B}^1)^{\alpha'}]≤C_{\alpha'}i^{\alpha'}+ \Eh_i\Big[\sum_{k=1}^{\tauh_1}\Big(\sum_{u\in\Omega(w_k)}\underbrace{\beta(u)+C_1\beta(u)\ln(1+\beta(u))}_{≤C_1'\beta(u)\ln(1+\beta(u))}\Big)^{\alpha'}\Big], 
\end{equation}
where $C_1'$ is a suitable constant. Conditioning with respect to $\sigma((w_k)_{k\in\brint{0;\tauh_1-1}})$, we get 
\begin{align}\label{eq:EiNjalpha2}
\Eh_i\Big[\sum_{k=1}^{\tauh_1}\Big(\sum_{u\in\Omega(w_k)}{\beta(u)\ln(1+\beta(u))}\Big)^{\alpha'} \Big]&≤\Eh_i\Big[ \sum_{k=0}^{\tauh_1-1} \Eh_{\beta(w_k)}\Big[ \Big(\sum_{|u|=1} \beta(u)\ln(1+\beta(u))\Big)^{\alpha'} \Big] \Big]\nonumber\\
&=\Eh_i\Big[ \sum_{k=0}^{\tauh_1-1} \frac{1}{\beta(w_k)}\E_{\beta(w_k)}\Big[\Big(\sum_{|u|=1}\beta(u)\Big) \Big(\sum_{|u|=1} \beta(u)\ln(1+\beta(u))\Big)^{\alpha'} \Big] \Big]\nonumber\\
&≤\Eh_i\Big[ \sum_{k=0}^{\tauh_1-1}\frac{1}{\beta(w_k)} \E_{\beta(w_k)}\Big[\Big(\sum_{|u|=1} \beta(u)\Big)^{1+\alpha'(1+\eps')} \Big] \Big],
\end{align}
where we used the branching property on each $w_k$ for $0≤k≤\tauh_1-1$ and where $\eps'>0$ can be chosen as small as we want. We let $\alpha''=\alpha(1+\eps')$.Recall from~\eqref{eq:defGk} that $(\cal{G}_k)_{k≥0}$ is the filtration generated by the environment. Recall from the second point of Lemma~\ref{lemma:lawbetaconseq} that under $\P_{\beta(w_k)}$, $\sum_{|u|=1}\beta(u)$ is the sum of $\beta(w_k)$ independent geometric random variables of parameter $\frac{1}{1+\sum_{|u|=1}\ee{-V(u)}})$; Lemma~\ref{lemma:NegBin} (given in the appendix) yields
\begin{align*}
\frac{1}{\beta(w_k)}\E_{\beta(w_k)}&\Big[\Big(\sum_{|u|=1} \beta(u)\Big)^{1+\alpha'(1+\eps')} \Big] =\frac{1}{\beta(w_k)}\Eb\Big[\E_{\beta(w_k)}\Big[\Big(\sum_{|u|=1} \beta(u)\Big)^{1+\alpha''} \Big| \cal{G}_1\Big] \Big]\\
&≤\underbrace{16\Big(\Eb\Big[\sum_{|u|=1}\ee{-V(u)}\Big]+\Eb\Big[\Big(\sum_{|u|=1}\ee{-V(u)}\Big)^{1+\alpha''}\Big]\Big)}_{=:K^1_{\alpha''}}+(\beta(w_k))^{1+\alpha''}\underbrace{2\Eb\Big[\Big(\sum_{|u|=1}\ee{-V(u)}\Big)^{\alpha''}\Big]}_{=:K^2_{\alpha''}}. 
\end{align*}
Hypothesis $\bf (H_\kappa)$ ensures the finiteness of $K^1_{\alpha''}$ and $K^2_{\alpha''}$. Plugging this back in equation~(\ref{eq:EiNjalpha2}) yields 
\begin{equation*}
\Eh_i\Big[\sum_{k=1}^{\tauh_1}\Big(\sum_{u\in\Omega(w_k)}{\beta(u)\ln(1+\beta(u))}\Big)^{\alpha'} \Big]≤\Eh_i\Big[\sum_{k=1}^{\tauh_1}K^1_{\alpha''}+K^2_{\alpha''}(\beta(w_k))^{\alpha''}\Big]
\end{equation*} 
Now, together with inequality~(\ref{eq:majtaualpha}) this yields
\begin{equation*}
\Eh_i\Big[\sum_{k=1}^{\tauh_1}\Big(\sum_{u\in\Omega(w_k)}{\beta(u)\ln(1+\beta(u))}\Big)^{\alpha'} \Big]≤(K^1_{\alpha''}+K^2_{\alpha''})C_{\alpha''}i^{\alpha''}, 
\end{equation*}
where we recall that $\alpha''$ can be chosen as close to $\alpha'$ as we want, and $\alpha'$ as close to $\alpha$ as we want. Plugging this into~(\ref{eq:EiNjalpha1}) and then in~(\ref{eq:EiNjalpha}) (via~(\ref{eq:B1holder})), we get that for any $\alpha'''>\alpha$ as close to $\alpha$ as we want, 
\begin{equation*}
\E_i\big[(\widetilde{B}^1)^{1+\alpha}\big]≤C'_{\alpha'''}i^{1+\alpha'''}
\end{equation*}
where $C'_{\alpha'''}$ is a suitable constant. 
\end{proof}

To conclude this subsection and before proving that ${\bf (H_{\ell})}(iii)$ is satisfied, let us compute the constants $\mu_R$, $m_R$, $\mu_X$ and $m_X$ introduced in Proposition~\ref{prop:regvar}. 
\begin{lemma}\label{lemma:constants}
The constants $m$ and $\mu$ associated with forests $\Fb^R$ and $\Fb^X$ are
\begin{align*}
&\mu_R=\mu:=(a_1 b_1)^{-1}\quad\text{and}\quad m_R=a_1^{-1}\quad\text{for $\Fb^R$},\\
&\mu_X=\mu=(a_1 b_1)^{-1}\quad\text{and}\quad m_X=2(a_1 b_1)^{-1}\quad\text{for $\Fb^X$}, 
\end{align*}
\end{lemma}
\begin{proof}
Recall from Proposition~\ref{prop:markovtype} that $(\beta(w_k))_{k≥0}$ is a Markov chain on $\N^*$ with invariant measure $(\pi_i)_{i≥1}$ given in~\eqref{eq:defpi}. According to the remark at the beginning of Section~\ref{sec:technique}, we have
\begin{equation*}
\mu_X=\mu_R=\E_1\big[ \sum_{u\in\cal{L}^1} |u| \big]=\Eh_1\big[ |w_{\tauh_1}| \big]=\Eh_1[\tauh_1]=1/\pi_1=(a_1 b_1)^{-1}, 
\end{equation*}
where we used the many-to-one lemma (Lemma~\ref{lemma:many-to-one-multitype}) in the third equality. We also have still using the remark at the beginning of Section~\ref{sec:technique} and the many-to-one lemma, 
\begin{equation*}
m_R=\E_1\big[ \sum_{u\in\cal{B}^1} 1 \big]=\Eh_1\big[\sum_{k=1}^{\tauh_1}\frac{1}{\beta(w_k)}]=\sum_{i≥1}\frac{1}{i}\frac{\pi_i}{\pi_1}=(a_1)^{-1}
\end{equation*}
and
\begin{equation*}
m_X=\E_1\big[ \sum_{u\in\cal{B}^1} 2\beta(u) \big]=2\Eh_1\big[\sum_{k=1}^{\tauh_1}1]=2\sum_{i≥1}\frac{\pi_i}{\pi_1}=2(a_1 b_1)^{-1}. 
\end{equation*}
\end{proof}

\subsection{Hypothesis ${\bf (H_{\ell})} (iii)$}\label{subsec:cvqueue1}
In this section, we intend to show that hypothesis ${\bf (H_{\ell})}(iii)$ is satisfied by $\Fb^R$ and $\Fb^X$, \ie according to the beginning of Section~\ref{sec:technique}, that there exists a constant $K^\star>0$ such that 
\begin{equation}\label{eq:queueL1}
\P_1(L^1>x)\ssim{x\to\infty} K^\star x^{-\kappa}. 
\end{equation}
Actually, establishing the regular variation of the tail of $L^1$ under $\P_1$ is equivalent to establishing that of $L^1$ under $\Ph_1$. 
\begin{lemma}\label{lemma:P1L1eqPh1L1}
There exists a constant $K^\star\in(0;\infty)$ such that as $x\to\infty$, $
\P_1(L^1>x)\sim K^\star x^{-\kappa}$ if and only if as $x\to\infty$, 
\begin{equation}\label{eq:equivPhL1>n}
\Ph_1(L^1>x)\sim \frac{\kappa}{\kappa-1}K^\star x^{-(\kappa-1)}. 
\end{equation}
\end{lemma}
\begin{proof}
According to Theorem~8.1.4 of~\cite{bingham-goldie-teugels}, equation~\eqref{eq:queueL1} is equivalent to 
\begin{equation*}
\E_1[L^1\1{L^1>x}]\ssim{x\to\infty} \frac{\kappa }{\kappa-1}K^\star x^{-(\kappa-1)}. 
\end{equation*}
The many-to-one lemma (Lemma~\ref{lemma:many-to-one-multitype}) yields
\begin{align*}
\E_1\Big[L^1\1{L^1>x}\Big]&=\sum_{k≥1}\E_1\Big[\sum_{|u|=k}\1{u\in\L^1}\1{L^1>x}\Big]\\
&=\sum_{k≥1}\Eh_1\Big[\1{k=\tauh_1}\P_1\big(L^1>x\mid \cal{F}_k\big)\Big]\\
&=\sum_{k≥1}\Eh_1\Big[\1{k=\tauh_1}\Ph_1\big(L^1>x\mid \cal{F}_k\big)\Big]=\Ph_1\Big(L^1>x\Big), 
\end{align*}
where between the last two lines we used the fact that on the event $\{k=\tauh_1\}$ we have $\P_1\big(L^1>x\mid \cal{F}_k\big)=\Ph_1\big(L^1>x\mid \cal{F}_k\big)$. The last equality comes from the fact that $(\beta(w_k))_{k≥0}$ is recurrent. This concludes the proof.
\end{proof} 
This lemma motivates us to understand the behaviour of $L^1$ under $\Ph$. To this end, let us describe the behaviour of $L^1$ after a large number of excursions from the root. 

\subsubsection{Behaviour of $L^1$ with large initial local time}
Let us begin with a result allowing us to control the small moments of $L^1$. 
\begin{lemma}
Let $\alpha\in(0,\kappa-1)$, and $i≥1$. There exists a constant $C_\alpha\in(0,\infty)$ such that
\begin{equation}\label{eq:majLjalpha}
\E_i[(L^1)^{1+\alpha}]≤C_\alpha i^{1+\alpha}. 
\end{equation}
\end{lemma}
\begin{proof}
The proof of Lemma~\ref{lemma:momB1} on the small moments of $\widetilde{B}^1$ can easily be adjusted to $L^1$, just by following its lines, to get this sharper estimate. Indeed, there is not any need to shift from $\alpha$ to $\alpha'$, since we get $\E_i\Big[\Big(\sum_{|u|=k}\1{u\in{\L^1}}\Big){(L^1)}^{\alpha}\Big]=i\Eh_i\Big[ \1{\tauh_1=k}(L^1)^\alpha\Big]$ in~(\ref{eq:majbalphaEh}) and so $\E_i[(L^1)^{1+\alpha}]≤i\Eh_i[(L^1)^{\alpha}]$ in~(\ref{eq:EiNjalpha}). For the same reason, equation~(\ref{eq:B1holder}) will not be necessary and the $\ln(1+\beta(w_k))$ will not appear in~(\ref{eq:B1toL1}), so there will not be any need to shift from $\alpha''$ to $\alpha'''$. Moreover, there is not any need either to shift from $\alpha'$ to $\alpha(1+\eps)=\alpha''$ in equation~(\ref{eq:EiNjalpha2}) since the term $\ln(1+\beta(u))$ will not appear in~(\ref{eq:EiNjalpha1}) (as we can use equation~(\ref{eq:majLj}) instead of~(\ref{eq:majEiBj})). 
\end{proof}

Let us now state a proposition that describes the behaviour of $L^1$ under $\P_n$ when $n$ is large. 
\begin{proposition}\label{prop:Winfini} Let $\alpha\in(0;\kappa-1)$, we have the following convergence in mean of order $1+\alpha$, 
\begin{equation*}
\E_n\Big[\Big|\frac{L^1}{n}-W_\infty|^{1+\alpha}\Big]\sto{n\to\infty} 0, 
\end{equation*}
where $W_\infty$ is the almost sure limit of the positive martingale $(W_k)_{k≥1}:=(\sum_{|u|=k}\ee{-V(u)})_{k≥1}$. 
\end{proposition}

\begin{proof}
At the heart of this proof is~\eqref{eq:majLj} (which states that for every $i≥1$, $E_i[L^1]=i$), and the law of large numbers : for every $u\in\T$, since excursions from $\root$ are independent, 
\begin{equation*}
\frac{\beta(u)}{n}\substack{\P^\V_n\\\longrightarrow\\n\to\infty}\ee{-V(u)}, 
\end{equation*}
the expectation of $\beta(u)$ being computed in~\eqref{eq:espcondbeta}. 

Let us first establish the convergence in law of $L^1/n$. Actually we intend to prove a little more than this: for any $\ell≥1$, $M≥0$, we will prove the convergence in law of $L^1/n-{W_\ell}\wedge M$, where we recall that $W_\ell:=\sum_{|u|=\ell}\ee{-V(u)}$. To this end, we will establish the convergence of its Laplace transform (which exists as it is bounded from below). 

Let us start with the lower bound. Recall that for every $k≥1$ we denote by $\cal{F}_k$ the sigma-algebra generated by $(\beta(u),u\in\Tb,|u|≤k)$, and  $\cal{G}_k:=\sigma((u,V(u))_{u\in\T,|u|≤k})$ the sigma-algebra generated by the environment below generation $k$. For $u\in\Tb$, recall that we denote by $L^1_u$ the cardinal of $\cal{L}^1_u$, and let $W_\ell^M:={W_\ell}\wedge M$. Now for any $t>0$, $k≥\ell$, 
\begin{align*}
\E_n\big[\exp(-t(\frac{L^1}{n}-W_\ell^M))\big]&=\E_n\Big[\E_n\big[\exp(-\frac{t}{n}\Big(\sum_{|u|=k}\1{u<\mathcal{L}^1}L^1_u+\sum_{|u|<k}\1{u\in\mathcal{L}^1}\Big))\mid\cal{F}_k,\cal{G}_k\big]\exp(t W_\ell^M)\Big]\\
&≥\E_n\Big[\exp(-\frac{t}{n}\sum_{|u|=k}\1{u<\mathcal{L}^1}\E_{\beta(u)}[L^1]-\sum_{|u|<k}\frac{t}{n}\1{u\in\mathcal{L}^1})\exp(t W_\ell^M)\Big] \\
&=\E_n\Big[\exp(-\frac{t}{n}\sum_{|u|=k}\1{u<\mathcal{L}^1}\beta(u)-\sum_{|u|<k}\frac{t}{n}\1{u\in\mathcal{L}^1})\exp(t W_\ell^M)\Big], 
\end{align*}
where the last equality comes from~(\ref{eq:majLj}). Now recall that under $\P^\V_n$, $\sum_{|u|=k}\beta(u)$ is the sum of $n$ independent $\sum_{|u|=k}\beta(u)$ under $\P^\V_1$, and that according to Lemma~\ref{lemma:lawbetaconseq}, $\E^\V_1[\sum_{|u|=k}\beta(u)]=\sum_{|u|=k}\ee{-V(u)}$. Hence, conditionally on the environment, $\frac{1}{n}\sum_{|u|=k}\beta(u)\substack{ \P^\V_n\\ \longrightarrow \\ {n\to\infty}}\sum_{|u|=k}\ee{-V(u)}$ in probability, by the law of large number. Moreover, as for any $u\in\T$ of generation $k$, $\1{u<\cal{L}^1}$ tends to $1$ \as with $n$, this implies the convergence in probability of $\frac{1}{n}\sum_{|u|=k}\1{u<\mathcal{L}^1}\beta(u)$ towards $\sum_{|u|=k}\ee{-V(u)}$. These convergences are immediate if $\{u\in\T\st|u|=k\}$ is finite; otherwise, one way to see that is to introduce for every $p≥1$
\begin{equation*}
\cal{E}_k^p:=\{u\in\T\st |u|=k, \#\{v\in\T\st \ee{-V(v)}≥\ee{-V(u)}\}<p \}, 
\end{equation*}
that is the set of the $p$ vertices of generation $k$ with lowest potential. Now 
\begin{align}\label{eq:majEkp}
\Big|\frac{1}{n}\sum_{|u|=k}\beta(u)\1{u<\cal{L}^1}-\sum_{|u|=k}\ee{-V(u)}\Big|≤&\Big|\frac{1}{n}\sum_{u\in\cal{E}_k^p}\beta(u)\1{u<\cal{L}^1}-\sum_{u\in\cal{E}_k^p}\ee{-V(u)}\Big|\nonumber\\
&+\frac{1}{n}\Big|\sum_{|u|=k,u\notin \cal{E}_k^p}\beta(u)\Big|+\Big|\sum_{|u|=k,u\notin \cal{E}_k^p}\ee{-V(u)}\Big|, 
\end{align}
and since $\cal{E}_k^p$ is finite, the first member of this last sum tends to $0$ with $n$, and the second tends \as to $|\sum_{|u|=k,u\notin \cal{E}_k^p}\ee{-V(u)}|$ by the law of large numbers, which can be made as small as desired with $p$ large, hence ensuring the convergence in probability of $\frac{1}{n}\sum_{|u|=k}\beta(u)\1{u<\cal{L}^1}$. Moreover, the many-to-one lemma yields
\begin{equation*}
\frac{1}{n}\E_n\Big[\sum_{|u|<k}\1{u\in\cal{L}^1}\Big]=\Ph_n\big(\tauh_1<k\big)≤\Big(\Ph\big(X^{1,w_0}\text{ does not reach the level $k$}\big)\Big)^n\sto{n\to\infty} 0.
\end{equation*}
The inequality can be seen by the representation of the Markov chain $(\beta(w_k))_{k≥0}$ given in the previous subsection, and using the fact that $\tauh_1$ is the first level such that none of the walks launched below reach it. The latter equation implies the convergence in probability of $n^{-1}\sum_{|u|=k}\1{u<\cal{L}^1}$ towards $0$. \\
Finally, Lebesgue's dominated convergence theorem yields
\begin{equation*}
\liminf_{n\to\infty}\E_n\big[\exp(-t\frac{L^1}{n}+t W_\ell^M)\big]≥\Eb\Big[ \exp(-t\sum_{|u|=k}\ee{-V(u)}+tW_\ell^M)\Big], 
\end{equation*} 
and $k$ being arbitrary, $\liminf_n \E_n[\exp(-t(\frac{L^1}{n}-W_\ell^M))]≥\lim_{k\to\infty}\Eb[ \exp(-t\sum_{|u|=k}\ee{-V(u)}+tW_\ell^M)]=\Eb[\exp(-t(W_\infty-W_\ell^M))]$ (the equality is obtained by Lebesgue's dominated convergence theorem). 

Let us now tackle the upper bound. For any $t>0$, 
\begin{align*}
\E_n\big[\exp(-t(\frac{L^1}{n}-W_\ell^M))\big]&=\E_n\Big[\Big(\exp\big(-\frac{t}{n}(\sum_{|u|<k}\1{u\in\L^1_n}+\sum_{|u|=k, u<\L^1}L^1(u))\big)\wedge 1\Big)\exp(t W_\ell^M)\Big]\\
&≤\E_n\Big[\Big(\prod_{|u|=k,u<\L^1}\E_n\big[\exp\big(-\frac{t}{n}L^1(u)\big)\mid\cal{F}_k\big]\wedge 1\Big)\exp(t W_\ell^M)\Big]. 
\end{align*}
Now as remarked in~\cite{maillard-zeitouni} between equations~(B.4) and~(B.5), the inequality $\ee{-x}≤1-x+x^{1+\alpha}≤\ee{-x+x^{1+\alpha}}$ yields
\begin{align*}
&≤\E_n\Big[\Big(\prod_{|u|=k,u<\L^1}\E_n[1-\frac{t}{n}L^1(u)+(\frac{t}{n}L^1(u))^{1+\alpha}]\wedge 1\Big)\exp(t W_\ell^M)\Big]\\
&≤\E_n\Big[\Big(\prod_{|u|=k,u<\L^1}\exp\big(-\frac{t}{n}\E_n[L^1(u)]+(\frac{t}{n})^{1+\alpha}\E_n[(L^1(u))^{1+\alpha}]\big)\wedge 1\Big)\exp(t W_\ell^M)\Big]\\
&≤\E_n\Big[\Big(\exp\big(-\sum_{|u|=k,u<\L^1} t\frac{\beta(u)}{n}+t^{1+\alpha}C_{\alpha}(\frac{\beta(u)}{n})^{1+\alpha}\big)\wedge 1\Big)\exp(t W_\ell^M)\Big], 
\end{align*}
where the last inequality results from~(\ref{eq:majLjalpha}). Now since $\frac{1}{n}\sum_{|u|=k}\1{u<\cal{L}^1}\beta(u)\sto{n\to\infty}\sum_{|u|=k}\ee{-V(u)}$ in probability and $\frac{1}{n^{1+\alpha}}\sum_{|u|=k}\1{u<\cal{L}^1}\beta(u)^{1+\alpha}\sto{n\to\infty}\sum_{|u|=k}\ee{-(1+\alpha)V(u)}$ (this can be proved using the same strategy as in~(\ref{eq:majEkp})), by Lebesgue's dominated convergence theorem we get 
\begin{equation*}
\limsup_{n\to\infty}\E_n[\exp(-(t\frac{L^1}{n}-W_\ell^M))]≤\Eb\Big[\big(\exp\big(-t\sum_{|u|=k} \ee{-V(u)}+C_{\alpha}t^{1+\alpha}\sum_{|u|=k}\ee{-(1+\alpha)V(u)}\big)\wedge 1\big)\exp(tW_\ell^M)\Big]
\end{equation*}
But according to our hypotheses, as $\Eb[\sum_{|u|=1}\ee{-(1+\alpha)V(u)}]<1$ since $\alpha<\kappa-1$, $\sum_{|u|=k}\ee{-(1+\alpha)V(u)}\sto{k\to\infty}0$ in probability, and therefore the function in the expectation being bounded, and $k$ arbitrary
\begin{align*}
\limsup_{n\to\infty}\E_n[\exp\big(-t(\frac{L^1}{n}-&W_\ell^M)\big)]\\
&≤\lim_{k\to\infty}\Eb\Big[(\exp(-t\sum_{|u|=k} \ee{-V(u)}+C_{\alpha}t^{1+\alpha}\sum_{|u|=k}\ee{-(1+\alpha)V(u)})\wedge 1)\exp(tW_\ell^M)\Big]\\
&=\Eb\Big[\exp(-t(W_\infty-W_\ell^M))\Big]. 
\end{align*}
Thus, the convergence in law of $(L^1/n-W_\ell^M)_n$ is established, as its Laplace transform converges towards that of $W_\infty-W_\ell^M$. 

Let us now consider $\alpha\in(0;\kappa-1)$, we have
\begin{equation*}
\E_n[|\frac{L^1}{n}-W_\ell^M|^{1+\alpha}]≤2\E_n[|\frac{L^1}{n}|^{1+\alpha}]+2\Eb[|W_\ell|^{1+\alpha}]≤C_\alpha+C', 
\end{equation*}
where we used equation~(\ref{eq:majLjalpha}) for the first inequality and Theorem~2.1 of~\cite{liu} for the second ($C'$ being then a constant independent of $\ell$ and $M$). This ensures that $(|L^1/n-W_\ell^M|^{1+\alpha'})_{n≥1}$ is uniformly integrable for any $\alpha'<\alpha$, and $\alpha$ being arbitrary $(|L^1/n-W_\ell^M|^{1+\alpha})_{n≥1}$ is uniformly integrable. This together with the convergence in law of $L^1/n-{W_\ell} \wedge M$ implies the convergence of $\E_n[|L^1/n-W_\ell^M|^{1+\alpha}]$ towards $\Eb[|W_\infty-W_\ell^M|^{1+\alpha}]$ (the convergence in law being enough for the convergence of the expectation). Therefore, as for any $k≥1$,
\begin{equation*}
\E_n\big[|\frac{L^1}{n}-W_\infty|^{1+\alpha}\big]≤2\E_n\big[|\frac{L^1}{n}-W_\ell^M|^{1+\alpha}\big]+2\Eb\big[|W_\ell^M-W_\infty|^{1+\alpha}\big], 
\end{equation*}
we get, as $k$ and $M$ are arbitrary, 
\begin{equation*}
\limsup_n\E_n\big[|\frac{L^1}{n}-W_ \infty|^{1+\alpha}\big]≤4\Eb\big[|W_\infty-W_\ell^M|^{1+\alpha}\big]\sto{M\to\infty} 4\Eb\big[|W_\infty-W_\ell|^{1+\alpha}\big]\sto{\ell\to\infty}0, 
\end{equation*}
thus ensuring the convergence in mean of order ${1+\alpha}$ of $L^1/n$ under $\P_n$ towards $W_\infty$. 
\end{proof}

\subsubsection{Tail of $L^1$ under $\Ph$}\label{subsec:cvqueue}
We now intend to show that~(\ref{eq:equivPhL1>n}) stands. In what follows, we will systematically consider $(\T,V,(w_k)_{k≥0})$ under the law $\Pbh$ and the walk under the law $\Ph_1$ (that we will denote by $\Ph$). Our proof will be strongly inspired from that of H.~Kesten, M.V.~Kozlov and F.~Spitzer in~\cite{kks}. In the latter the authors studied the behaviour of the random walks $(X^{1,w_i}_n)_{n≥0}$ restricted to the spine. Letting $\tauh^1_1:=\min\{k≥0\st \ty^1(w_k)=0\}$ and $\kappa':=\kappa-1$, they proved that under our conditions, 
\begin{equation*}
\Ph(\sum_{k=0}^{\tauh^1_1}\ty^1(w_k)>x)\ssim{x\to\infty} K x^{-\kappa'}, 
\end{equation*}
where $K$ is a positive constant. Our strategy will rely on the fact that under $\Ph$, $\cal{L}^1$ can be decomposed into smaller lines stemming from the spine (plus $w_{\tauh_1}$). In terms of counting, this yields
\begin{equation}\label{eq:decompL1}
L^1=\sum_{k=1}^{\tauh_1} \sum_{u\in\Omega(w_k)} L^1_u\quad+1
\end{equation}
where we recall that for $k≥1$, we denoted by $\Omega(w_k):=\{u\neq w_k\st \parent{u}=w_{k-1}\}$ the set of the siblings of $w_k$, and that for $u\in\T$, $L^1_u$ stands for the cardinal of the vertices of $\cal{L}^1$ descending from $u$ (that we choose here to be equal to $1$ if $u\in\cal{L}^1$). Now let for every $u\in\T$ 
\begin{equation}\label{eq:defWu}
W_\infty^u:=\lim_{n\to\infty} \sum_{u<v, |v|-|u|=n,}\ee{-(V(v)-V(u))}
\end{equation}
be the limit of the additive martingale stemming from $u$ (so $W_\infty=W_\infty^{\root}$). We also let for all $A≥1$
\begin{equation*}
\sigma_{A}:=\min\{k≥1\st \ty(w_k)>A \} 
\end{equation*}
be the first time that the edge local time of the spine is larger than $A$. The first thing that we will prove, in Lemma~\ref{lemma:tau1sigmaA}, is that actually if we want $L^1$ to be large, then $\sigma_A$ has to be smaller than $\tauh_1$, that is, there has to be a vertex on the spine whose edge local time is larger than $A$. Then, still in Lemma~\ref{lemma:tau1sigmaA} we will show that the number of vertices of $\cal{L}^1$ which separated from the spine below $w_{\sigma_A}$ is negligible when $L^1$ is large; that is we will show that the two quantities
\begin{itemize}
\item $L^1=\sum_{k=1}^{\tauh_1}\sum_{u\in\Omega(w_k)}L^1_u+1$ 
\item $\sum_{k=\sigma_A+1}^{\tauh_1}\sum_{u\in\Omega(w_k)}L^1_u$
\end{itemize}
are close. Now, as the heuristics of Proposition~\ref{prop:Winfini} say that for $u\in\T$, $L^1_u\approx \ty(u)W_\infty^u$ when $\ty(u)$ is large, we expect that when $L^1$ is large, 
\begin{equation*}
L^1\approx\sum_{k=\sigma_A+1}^{\tauh_1} \sum_{u\in\Omega(w_k)} \ty(u)W^u_\infty. 
\end{equation*}
This is what we will show at the end of this subsection, in Lemma~\ref{lemma:L1toW}. Before that, in Lemma~\ref{lemma:negspineabove}, we will see that the contribution to $L^1$ of the walks launched above $w_{\sigma_A}$ is negligible; that is if we denote for every $u\in\T$ by
\begin{equation}\label{eq:defbetaA}
\ty^A(u):=\sum_{k=0}^{\sigma_A-1} \ty^1_k(u)+\ty^2_k(u)
\end{equation}
the edge local time on $u$ of the walks launched below $w_{\sigma_A}$, then the quantities 
\begin{itemize}
\item $\sum_{k=\sigma_A+1}^{\tauh_1} \sum_{u\in\Omega(w_k)} \ty(u)W^{u}_\infty$
\item $\sum_{k=\sigma_A+1}^{\tauh_1} \sum_{u\in\Omega(w_k)} \ty^A(u)W^{u}_\infty$
\end{itemize}
are close. And we will see in Lemma~\ref{lemma:final} that for large $A$, the behaviour of this last quantity is dictated by $\ty^A(w_{\sigma_A})$ together with the environment above $w_{\sigma_A}$, namely that
\begin{itemize}
\item $\sum_{k=\sigma_A+1}^{\tauh_1} \sum_{u\in\Omega(w_k)} \ty^A(u)W^{u}_\infty$\quad and
\item $\ty^A(w_{\sigma_A})W^{w_{\sigma_A}}_\infty$
\end{itemize}
are close. 

To sum up, the heuristics indicate that when $L^1$ is large, then it is close to $\ty^A(w_{\sigma_A})W^{w_{\sigma_A}}_\infty$, and we have $\sigma_A<\tauh_1$ w.h.p. We will conclude in Proposition~\ref{prop:xkappa}: we will see that $\Eh[\ty^A(w_{\sigma_A}))^{\kappa'}]\times\P( \sigma_A<\tauh_1 )$ converges towards a non-trivial constant $K$, and since there exists a non-trivial constant $\widehat{C}_\infty$ such that $\Pbh( W^{w_{\sigma_A}}_\infty>x )\sim \widehat{C}_\infty x^{-\kappa'}$, we will get the same tail distribution for $L^1$: $\Ph(L^1>x)\sim K\widehat{C}_\infty x^{-\kappa'}$. 

We determine the tail distribution of $W_{\infty}^{w_{\sigma_A}}$ under $\Pbh$ in the following lemma (this result was first shown by Liu~\cite{liu} in order to get the tail distribution of $W_\infty$ under $\Pb$). 
\begin{lemma}\label{lemma:tailW}
{\bf \cite{liu}} Recall from~\eqref{eq:defWu} the definition of $W_\infty^{u}$ for $u\in\T$. When $x\to\infty$, 
\begin{equation*}
\Pbh\big( W^{w_{\sigma_A}}_\infty>x \big)=\Pbh\big( W_\infty>x \big)\sim \widehat{C}_\infty x^{-\kappa'}, 
\end{equation*}
where 
\begin{equation*}
\widehat{C}_\infty=\frac{1}{\kappa\Eb\Big[\sum_{|u|=1}(-V(u))\ee{-\kappa V(u)}\Big]}\Ebh\Big[ \Big( W_\infty \Big)^{\kappa'}-\Big( \ee{-V(w_1)}W^{w_1}_\infty\Big)^{\kappa'} \Big]. 
\end{equation*}
\end{lemma}

\begin{proof}
The environment above $w_{\sigma_A}$ being independent of $w_{\sigma_A}$, $W^{w_{\sigma_A}}_\infty$ has same law as $W_\infty$, hence the first equality. 

Now notice that under $\Pbh$, 
\begin{equation*}
W_{\infty}=\sum_{k=0}^{\infty}\ee{-V(w_k)}\sum_{u\in\Omega(w_{k+1})}\ee{-\Delta V(u)}W^{u}_{\infty}. 
\end{equation*}
The $(\ee{-\Delta V(w_k)})_{k≥1}$ are i.i.d., with $\Ebh[\ee{-\kappa' V(w_1)}]=1$ and $\Ebh[(-V(w_1))^+\ee{-\kappa' V(w_1)}]<\infty$ thanks to hypothesis $\bf (H_{\kappa})$. Moreover, 
\begin{equation*}
\Ebh\Big[\big(\sum_{u\in\Omega(w_{1})}\ee{-\Delta V(u)}W^{u}_{\infty}\big)^{\kappa'}\Big]≤\Ebh\Big[\big(\sum_{|u|=1}\ee{-\Delta V(u)}\Ebh[W^{u}_{\infty}\mid \mathcal{G}_1]\big)^{\kappa'}\Big]=\Eh\Big[\sum_{|u=1|}\big(\sum_{|u|=1}\ee{-\Delta V(u)}\times 1\big)^{\kappa'}\Big]=1<\infty, 
\end{equation*}
by Jensen's inequality and the many-to-one lemma. Hence, as $N$ is non-lattice, Theorem~B of~\cite{kesten} can be applied to $W_{\infty}$ under $\Pbh$:
\begin{equation*}
\Pbh(W_\infty^{w_{\sigma_A}}>x)=\Pbh(W_\infty>x)\sim \widehat{C}_\infty x^{-\kappa'}. 
\end{equation*}
The expression of $\widehat{C}_\infty$ in terms of fractional moments of $W_\infty$ is given in Theorem~4.1~of~\cite{goldie}. 
\end{proof}

Let us now prove Lemmas~\ref{lemma:tau1sigmaA},~\ref{lemma:final},~\ref{lemma:negspineabove} and~\ref{lemma:L1toW}. 

\begin{lemma}\label{lemma:tau1sigmaA}
For all $\eps>0$, $A>0$, 
\begin{equation*}
\Ph\big( L^1>\eps x,\, \tauh_1≤\sigma_A\big)=o(x^{-\kappa'}). 
\end{equation*}
Moreover, we have for all $\eps>0$, $A>0$, 
\begin{equation*}
\Ph\Big( \sum_{k=1}^{\sigma_A}\sum_{u\in\Omega(w_k)}L^1_u>\eps x,\, \sigma_A<\tauh_1\Big)=o(x^{-\kappa'}). 
\end{equation*}
\end{lemma}
\begin{proof}
Let $\eps>0$ and $A>0$. Let us apply Markov's inequality in both cases:
\begin{align*}
\Ph\big( L^1>\eps x,\, \tauh_1≤\sigma_A\big)&≤(\eps x)^{-\kappa'}\Eh\Big[ (L^1\1{\tauh_1≤\sigma_A})^{\kappa'}\1{L^1>\eps x} \Big]\\
&=(\eps x)^{-\kappa'}\Eh\Big[ \Big(1+\sum_{k=1}^{\tauh_1}\sum_{u\in\Omega(w_k)}L^1_u\Big)^{\kappa'}\1{\tauh_1≤\sigma_A}\1{L^1>\eps x} \Big]\\
&≤(\eps x)^{-\kappa'}\Eh\Big[ \Big(1+\sum_{k=1}^{\tauh_1}\1{\ty(w_{k-1})<A}\sum_{u\in\Omega(w_k)}L^1_u\Big)^{\kappa'}\1{L^1>\eps x} \Big], 
\end{align*}
where the equality is due to~(\ref{eq:decompL1}). For the same reason, 
\begin{align*}
\Ph\big( \sum_{k=1}^{\sigma_A}\sum_{u\in\Omega(w_k)}L^1_u>\eps x,\, \sigma_A<\tauh_1\big)&≤(\eps x)^{-\kappa'}\Eh\Big[ \Big(1+\sum_{k=1}^{\sigma_A}\sum_{u\in\Omega(w_k)}L^1_u\Big)^{\kappa'}\1{\sigma_A<\tauh_1}\1{L^1>\eps x} \Big]\\
&≤(\eps x)^{-\kappa'}\Eh\Big[ \Big(1+\sum_{k=1}^{\tauh_1}\1{\ty(w_{k-1})<A}\sum_{u\in\Omega(w_k)}L^1_u\Big)^{\kappa'}\1{L^1>\eps x} \Big]. 
\end{align*}
Now since $\1{L^1>\eps x}$ tends to $0$ when $x$ tends to infinity, it is enough to show the finiteness of~$\Eh\Big[ \Big(\sum_{k=1}^{\tauh_1}\1{\ty(w_{k-1})<A}\sum_{u\in\Omega(w_k)}L^1_u\Big)^{\kappa'} \Big]$ to conclude. We have, $\kappa'$ being smaller than $1$, 
\begin{align*}
\Eh\Big[ \Big(\sum_{k=1}^{\tauh_1}\1{\ty(w_{k-1})<A}\sum_{u\in\Omega(w_k)}L^1_u\Big)^{\kappa'} \Big]&≤\Eh\Big[ \sum_{k=1}^{\tauh_1}\1{\ty(w_{k-1})<A}\big(\sum_{u\in\Omega(w_k)}L^1_u\big)^{\kappa'} \Big]\\
&≤\Eh\Big[ \sum_{k=1}^{\tauh_1}\1{\ty(w_{k-1})<A}\big(\sum_{u\in\Omega(w_k)}\E_{\ty(u)}\big[L^1\big]\big)^{\kappa'} \Big]\\
&=\Eh\Big[ \sum_{k=1}^{\tauh_1}\1{\ty(w_{k-1})<A}\big(\sum_{u\in\Omega(w_k)}\ty(u)\big)^{\kappa'} \Big], 
\end{align*}
where the second inequality is obtained after applying Jensen's inequality (as $\kappa'≤1$) and the branching property. The last equality is due to~(\ref{eq:majLj}). Now conditioning on the edge local times of the spine and using the branching property, we get that this is
\begin{align*}
&≤\Eh\Big[ \sum_{k=1}^{\tauh_1}\1{\ty(w_{k-1})<A}\Eh_{\ty(w_{k-1})}\big[\big(\sum_{|u|=1}\ty(u)\big)^{\kappa'}\big] \Big]\\
&=\Eh\Big[ \sum_{k=1}^{\tauh_1}\frac{\1{\ty(w_{k-1})<A}}{\ty(w_{k-1})}\E_{\ty(w_{k-1})}\big[\big(\sum_{|u|=1}\ty(u)\big)\big(\sum_{|u|=1}\ty(u)\big)^{\kappa'}\big] \Big]\\
&=\Eh\Big[ \sum_{k=1}^{\tauh_1}\frac{\1{\ty(w_{k-1})<A}}{\ty(w_{k-1})}\E_{\ty(w_{k-1})}\big[\big(\sum_{|u|=1}\ty(u)\big)^\kappa\big] \Big]
\end{align*}
where the first equality comes from Proposition~\ref{prop:markovtype}. Finally, as noticed in Lemma~\ref{lemma:lawbetaconseq}, under $\P_{\ty(w_{k-1})}$, $\sum_{|u|=1}\ty(u)$ has the law of the sum of $(\ty(w_{k-1})$ geometric r.v. of parameter $1/\big(1+\sum_{|u|=1}\ee{-V(u)}\big)$, and Lemma~\ref{lemma:NegBin} yields that this is
\begin{align*}
&≤\Eh\Big[ \sum_{k=1}^{\tauh_1}\frac{\1{\ty(w_{k-1})<A}}{\ty(w_{k-1})}\Big(16\ty(w_{k-1})\big(\Eb\Big[\sum_{|u|=1}\ee{-V(u)}\Big]+\Eb\Big[\big(\sum_{|u|=1}\ee{- V(u)}\big)^\kappa\Big]\big)+2(\ty(w_{k-1}))^\kappa\Eb\Big[\big(\sum_{|u|=1}\ee{-V(u)}\big)^\kappa\Big]\Big) \Big]\\
&≤\Eh\Big[ \sum_{k=1}^{\tauh_1}\Big(16\big(\Eb\Big[\sum_{|u|=1}\ee{-V(u)}\Big]+\Eb\Big[\big(\sum_{|u|=1}\ee{- V(u)}\big)^\kappa\Big]\big)+2 A^{\kappa'}\Eb\Big[\big(\sum_{|u|=1}\ee{-V(u)}\big)^\kappa\Big]\Big) \Big]<\infty, 
\end{align*}
by $\bf (H_\kappa)$, and as $\Eh[\tauh_1]$ is finite. 
\end{proof}
In light of~(\ref{eq:decompL1}), we get from this lemma that for any $\eps>0$, for sufficiently large $A$ and $x$, 
\begin{equation}\label{eq:encadreL1}
\Ph\Big( \sum_{k=\sigma_A+1}^{\tauh_1}\sum_{u\in\Omega(w_k)}L^1_u> x,\, \sigma_A<\tauh_1\Big)≤\Ph\big(L^1>x\big)≤\Ph\Big( \sum_{k=\sigma_A+1}^{\tauh_1}\sum_{u\in\Omega(w_k)}L^1_u>(1-\eps) x,\, \sigma_A<\tauh_1\Big)+2\eps x^{-\kappa'}. 
\end{equation}
In the lemmas to come, we will consider a certain quantity, namely $\Eh\big[\big(\ty^A(w_{\sigma_A})\big)^{\kappa'}\1{\sigma_A<\tauh_1}\big]$. Let us prove that it is finite.
\begin{lemma}\label{lemma:KA}
We let
\begin{equation*}
K_A:=\Eh\big[\big(\ty^A(w_{\sigma_A})\big)^{\kappa'}\1{\sigma_A<\tauh_1}\big]=\Eh\big[\big(\ty^A(w_{\sigma_A})\big)^{\kappa'}\big]\Ph\Big(\sigma_A<\tauh_1\Big). 
\end{equation*}
The quantity $K_A$ is finite. 
\end{lemma}
\begin{proof}
The equality is due to the strong Markov property: if we let ${\sigma_A^+:=\min\{k>\tauh_1\st \ty(w_k)>A\}}$, then
\begin{equation*}
\Eh\Big[\big(\ty^A(w_{\sigma_A})\big)^{\kappa'}(1-\1{\sigma_A<\tauh_1})\Big]=\Eh\Big[\big(\ty^A(w_{\sigma^+_A})\big)^{\kappa'}(1-\1{\sigma_A<\tauh_1})\Big]=\Eh\Big[\big(\ty^A(w_{\sigma_A})\big)^{\kappa'}\Big]\Ph\Big(1-\1{\sigma_A<\tauh_1})\Big), 
\end{equation*}
by the strong Markov property in $\tauh_1$. 

To prove that this quantity is finite, let us just follow the lines of Lemma~4 of~\cite{kks}. For convenience, we will actually study $\ty(w_{\sigma_A})=\ty^A(w_{\sigma_A})+1$. Let $A>0$, we have on the event $\{\sigma_A<\tauh_1\}$, 
\begin{equation*}
\ty(w_{\sigma_A})=\big(\ty(w_{\sigma_A-1})+1\big)\frac{\ty(w_{\sigma_A})}{\ty(w_{\sigma_A-1})+1}≤(A+1)\frac{\ty(w_{\sigma_A})}{\ty(w_{\sigma_A-1})+1}≤(A+1)\sum_{k=1}^{\tauh_1-1}\frac{\ty(w_{k})}{\ty(w_{k-1})+1}.
\end{equation*}
Thus, as $\kappa'≤1$, and then by Jensen's inequality, we have
\begin{align*}
K_A=\Eh\big[\big(\ty(w_{\sigma_A})\big)^{\kappa'}\1{\sigma_A<\tauh_1}\big]&≤(A+1)^{\kappa'}\sum_{k≥1}\Eh\Big[\1{k<\tauh_1}\big(\frac{\ty(w_{k})}{\ty(w_{k-1})+1}\big)^{\kappa'}\Big]\\
&≤(A+1)^{\kappa'}\sum_{k≥1}\Eh\Big[\1{k<\tauh_1}\Big(\frac{\Eh\big[\ty(w_{k})\big|\big(\ty(w_{\ell})\big)_{\ell\in\brint{1;k-1}}\big]}{\ty(w_{k-1})+1}\Big)^{\kappa'}\Big]. 
\end{align*}
Now recall from Subsection~\ref{subsec:lawtb} that for any $k≥1$ $\ty(w_k)$ is a negative binomial random variable of parameters $\big(\ty(w_{k-1})+1,1/(1+\ee{V(w_k)-V(w_{k-1})})\big)$ plus one. Thus, 
\begin{align*}
K_A&≤(A+1)^{\kappa'}\sum_{k≥1}\Eh\Big[\1{k<\tauh_1}\Big(\frac{\big(\ty(w_{k-1})+1\big)\ee{-\Delta V(w_k)}+1}{\ty(w_{k-1})+1}\Big)^{\kappa'}\Big]\\
&≤(A+1)^{\kappa'}\sum_{k≥1}\Eh\Big[\1{k<\tauh_1}\big(\ee{-\kappa'\Delta V(w_k)}+1\big)\Big]
\end{align*}
The event $\{k<\tauh_1\}=\{k-1≤\tauh_1\}$ being independent of the environment above $w_{k-1}$ (so independent of $\ee{-\Delta V(w_k)}$), we get
\begin{align*}
K_A&≤(A+1)^{\kappa'}\sum_{k≥1}\Eh[\1{k<\tauh_1}]\big(\Ebh\big[\ee{-\kappa'\Delta V(w_k)}\big]+1\big)\\
&=(A+1)^{\kappa'}\Eh[\tauh_1-1]\big(\psi(\kappa)+1\big)<\infty, 
\end{align*}
where we used the branching property and the many-to-one lemma (Lemma~\ref{lemma:many-to-one-environnement}) in the last equality. The expectation $\Eh[\tauh_1-1]$ is finite thanks to~(\ref{eq:taulog}).
\end{proof}
\begin{lemma}\label{lemma:final}
For all $\eps>0$, there exists an $A_0$ such that for any $A>A_0$ we have for all $x≥1$,
\begin{equation*}
\Ph\Big(\Big| \big(\sum_{k=\sigma_A+1}^{\tauh_1}\sum_{u\in\Omega(w_{k})}\ty^A(u)W^{u}_\infty\big) - \ty^A(w_{\sigma_A})W^{w_{\sigma_A}}_\infty \Big|>\eps x,\ \sigma_A<\tauh_1\Big)≤\eps x^{-\kappa'} K_A. 
\end{equation*}
\end{lemma}
\begin{proof}
Actually, we intend to show a more general result, which will be helpful in the proof of Lemma~\ref{lemma:L1toW}: let us consider a family of \iid random variables indexed by the set of the siblings of the spine, that we denote by $(Z_u)_{u\in\Omega}$ (where we have set $\Omega :=\bigcup_{k≥1}\Omega(w_k)$). Suppose that this family is independent of the walk and the environment on the spine and the siblings of the spine $\big((\ty(w_k),V(w_k))_{k≥0},(\ty(u),V(u))_{u\in\Omega}\big)$, and suppose that these random variables admit a finite moment of order $\kappa'$. Let for every $k≥0$
\begin{equation}\label{eq:Zwk}
Z_{w_k}:=\sum_{\ell≥k+1}\ee{-(V(w_{\ell-1})-V(w_k))}\sum_{u\in\Omega(w_\ell)}\ee{-\Delta V(u)}Z_u. 
\end{equation}
Notice that when $(Z_u)_{u\in\Omega}=(W^u_\infty)_{u\in\Omega}$, then $Z_{w_k}=W_\infty^{w_k}$. As the $Z_u$ have a finite moment of order $\kappa'≤1$, so do the $\sum_{u\in\Omega(w_\ell)}\ee{-(V(u)-V(w_{\ell-1}))}Z_u$ for $\ell≥1$ (this can be seen using the many-to-one lemma, Hölder's inequality and $\bf (H_\kappa)$), and then Theorem~B of~\cite{kesten} ensures that there exists a constant $C_Z$ such that for any $k≥0$, 
\begin{equation}\label{eq:limZu}
\Ph\Big( Z_{w_k}>x \Big) \ssim{x\to\infty} C_Z x^{-\kappa'}. 
\end{equation}
Let us fix $\eps>0$. We want to prove that there exists an $A_0$ such that for any $A>A_0$ and any $x≥1$:
\begin{equation}\label{eq:lemmafinal}
\Ph\Big(\Big| \big(\sum_{k=\sigma_A+1}^{\tauh_1}\sum_{u\in\Omega(w_{k})}\ty^A(u)Z_u\big) - \ty^A(w_{\sigma_A})Z_{w_{\sigma_A}} \Big|>\eps x,\ \sigma_A<\tauh_1\Big)≤\eps x^{-\kappa'} K_A. 
\end{equation}
The $(W^u_\infty)_{u\in\Omega}$ being \iid and independent of the walk and the environment on the spine and the siblings of the spine, and admitting a finite moment of order $\kappa'$, proving~(\ref{eq:lemmafinal}) is enough to prove the lemma. \\

\noindent As $\ty^A(v)=0$ for vertices $v\in\T$ such that $w_{\tauh_1}≤v$, and by definition of $Z_{w_{\sigma_A}}$, we have
\begin{align*}
\sum_{k=\sigma_A+1}^{\tauh_1}\sum_{u\in\Omega(w_{k})}\ty^A(u)&Z_u - \ty^A(w_{\sigma_A})Z_{w_{\sigma_A}} =\sum_{k≥\sigma_A+1}\sum_{u\in\Omega(w_{k})}\big(\ty^A(u)- \ty^A(w_{k-1})\ee{-\Delta V(u)}\big)Z_u\nonumber\\
&+\sum_{k≥\sigma_A+1}\big(\ty^A(w_{k-1})-\ty^A(w_{\sigma_A})\ee{-(V(w_{k-1})-V(w_{\sigma_A}))}\big)\sum_{u\in\Omega(w_{k})}\ee{-\Delta V(u)}Z_u. 
\end{align*}
Hence, by the union bound, denoting $Z_{\Omega(w_k)}:=\sum_{u\in\Omega(w_{k})}\ee{-(V(u)-V(w_{k-1}))}Z_u$ for every $k≥1$, we just have to prove that the two following quantities
\begin{align*}
&f_A(x):=\Ph\Big(\Big|\sum_{k≥\sigma_A}\big(\ty^A(w_{k})-\ty^A(w_{\sigma_A})\ee{-(V(w_{k})-V(w_{\sigma_A}))}\big)Z_{\Omega(w_{k+1})}\Big|>\frac{\eps}{2} x,\ \sigma_A<\tauh_1\Big)\quad\text{and} \\ 
&g_A(x):=\Ph\Big(\Big|\sum_{k≥\sigma_A+1}\sum_{u\in\Omega(w_{k})}\big(\ty^A(u)- \ty^A(w_{k-1})\ee{-(V(u)-V(w_{\sigma_A}))}\big)Z_u\Big|>\frac{\eps}{2} x,\ \sigma_A<\tauh_1\Big)
\end{align*}
are smaller than $\frac{\eps}{2} x^{-\kappa'} \Eh\big[\big(\ty^A(w_{\sigma_A})\big)^{\kappa'}\1{\sigma_A<\tauh_1}\big]$ for $A$ large enough uniformly in $x≥1$ and $\eps>0$ to get~(\ref{eq:lemmafinal}). Let $\eps>0$, $x≥1$. 

Let us start with $f_A(x)$. Let us decompose $\ty^A(w_k)-\ty^A(w_{\sigma_A})\ee{-(V(w_k)-V(w_{\sigma_A}))}$ for $k≥\sigma_A$: 
\begin{align*}
\big(\ty^A(w_k) - \ty^A(w_{\sigma_A})\ee{-(V(w_k)-V(w_{\sigma_A}))}\big)&=\sum_{\ell=\sigma_A+1}^{k}\big(\ty^A(w_{\ell})\ee{-(V(w_k)-V(w_{\ell}))} - \ty^A(w_{\ell-1})\ee{-(V(w_k)-V(w_{\ell-1}))}\big)\\
&=\sum_{\ell=\sigma_A+1}^{k}\ee{-(V(w_k)-V(w_{\ell}))}\big(\ty^A(w_\ell) - \ty^A(w_{\ell-1})\ee{-(V(w_\ell)-V(w_{\ell-1}))}\big). 
\end{align*}
Hence, inverting the sums on $k$ and $\ell$ we get
\begin{align*}
 \sum_{k≥\sigma_A}\big|\ty^A&(w_{k})-\ty^A(w_{\sigma_A}|\ee{-(V(w_{k})-V(w_{\sigma_A}))}\big)Z_{\Omega(w_{k+1})}\\
&≤ \sum_{\ell≥\sigma_A+1}\big|\ty^A(w_\ell)-\ty^A(w_{\ell-1})\ee{-(V(w_\ell)-V(w_{\ell-1}))}\big|\sum_{k≥\ell}\ee{-(V(w_k)-V(w_{\ell}))}Z_{\Omega(w_{k+1})} 
\\
&=\sum_{\ell≥\sigma_A+1}\big(\ty^A(w_\ell) - \ty^A(w_{\ell-1})\ee{-\Delta V(w_\ell)}\big)Z_{w_\ell}.
\end{align*}
This yields (using also the fact that $\sum_{\ell≥1}\ell^{-2}=\pi^2/6<2$), 
\begin{align}\label{eq:longproof1ineg}
\begin{aligned}&f_A(x)=\Ph\Big(\big|\sum_{\ell≥\sigma_A+1}\big(\ty^A(w_\ell) - \ty^A(w_{\ell-1})\ee{-\Delta V(w_\ell)}\big)Z_{w_\ell}\big|> \sum_{\ell≥\sigma_A+1}(\ell-\sigma_A)^{-2}6\pi^{-2}\eps x, \sigma_A<\tauh_1\Big)\\
&≤\Eh\Big[\sum_{\ell≥\sigma_A+1} \Ph\Big(Z_{w_\ell}>\frac{(\ell-\sigma_A)^{-2}\eps x}{2|\ty^A(w_\ell)-\ty^A(w_{\ell-1})\ee{-\Delta V(w_\ell)}|}, \sigma_A<\tauh_1\,\Big|\, \sigma_A, (V(w_k),\ty^A(w_k))_{k≤\ell} \Big)\Big]
\end{aligned}\nonumber\\
≤\Eh\Big[\sum_{\ell≥\sigma_A+1} C'_Z(\eps x)^{-\kappa'}2^{\kappa'}(\ell-\sigma_A)^{2\kappa'}|\ty^A(w_\ell)-\ty^A(w_{\ell-1})\ee{-\Delta V(w_\ell)}|^{\kappa'}\1{\sigma_A<\tauh_1}\Big], 
\end{align}
where the last inequality comes from equation~(\ref{eq:limZu}), and where $C'_Z≥C_Z$ is a suitable constant. Now, as each $\ty^A(w_\ell)$ has the law of the sum of $\ty^A(w_{\ell-1})$ geometric r.v. of parameter $\ee{-\Delta V(w_\ell)}/(1+\ee{-\Delta V(w_\ell)})$, we have (using Jensen's inequality first)
\begin{align}\label{eq:longdiff}
\Eh\Big[ |\ty^A(w_\ell)-\ty^A&(w_{\ell-1})\ee{-\Delta V(w_\ell)}|^{\kappa'} \big| \ty^A(w_{\ell}),\Delta V(w_\ell) \Big]\nonumber\\
&≤\Eh\Big[ |\ty^A(w_\ell)-\ty^A(w_{\ell-1})\ee{-\Delta V(w_\ell)}|^{2} \big| \ty^A(w_{\ell}),\Delta V(w_\ell)\Big]^{\kappa'/2}\nonumber\\
&=\Eh\Big[\ty^A(w_{\ell-1})\big(\ee{-2\Delta V(w_\ell)}+\ee{-\Delta V(w_\ell)}\big)\big| \ty^A(w_{\ell}),\Delta V(w_\ell)\Big]^{\kappa'/2}\nonumber\\
&≤\big(\ty^A(w_{\ell-1})\big)^{\kappa'/2}\big(\ee{-\kappa'\Delta V(w_\ell)}+\ee{-\frac{\kappa'}{2}\Delta V(w_\ell)}\big), 
\end{align}
( we recall that $\kappa'≤1$). Plugging this into~(\ref{eq:longproof1ineg}) yields
\begin{align}\label{eq:longproof1last}
f_A(x)&≤C'_Z 2^{\kappa'}(\eps x)^{-\kappa'}\Eh\big[\sum_{\ell≥\sigma_A+1}(\ell-\sigma_A)^{2\kappa'}\big(\ty^A(w_{\ell-1})\big)^{\kappa'/2}\big(\ee{-\kappa'\Delta V(w_\ell)}+\ee{-\frac{\kappa'}{2}\Delta V(w_\ell)}\big)\1{\sigma_A<\tauh_1}\big]\nonumber\\
&≤C'_Z 2^{\kappa'}(\eps x)^{-\kappa'}\Eh\big[\sum_{\ell≥\sigma_A+1}(\ell-\sigma_A)^{2\kappa'}\big(\ty^A(w_{\ell-1})\big)^{\kappa'/2}\1{\sigma_A<\tauh_1}\big]\times\Ebh\big[\big(\ee{-\kappa'\Delta V(w_1)}+\ee{-\frac{\kappa'}{2}\Delta V(w_1)}\big)\big],
\end{align}
by the branching property of the environment. Thanks to \eqref{eq:majtaualpha} and the many-to-one lemma, we finally obtain
\begin{align}\label{eq:longdiff2}
f_A(x)&≤C'_ZC_\alpha 2^{\kappa'}(\eps x)^{-\kappa'}\Eh\big[\big(\ty^A(w_{\sigma_A})\big)^{\kappa'/2}\1{\sigma_A<\tauh_1}\big](\psi(\kappa)+\psi(1+\frac{\kappa}{2}))\nonumber\\
&≤CA^{-\kappa'/2} x^{-\kappa'} \Eh\big[\big(\ty^A(w_{\sigma_A})\big)^{\kappa'}\1{\sigma_A<\tauh_1}\big],
\end{align}
since $\ty^A(w_{\sigma_A})≥A$. Now for $A$ large enough, this is smaller than $\frac{\eps}{2} x^{-\kappa'}\Eh\big[\big(\ty^A(w_{\sigma_A})\big)^{\kappa'}\1{\sigma_A<\tauh_1}\big]$, which is what we wanted on $f_A(x)$. 

Let us now deal with $g_A(x)$. The equality $\sum_k 1/k^2=\pi^2/6$ and Markov's inequality yield
\begin{equation}\label{eq:majg}
g_A(x)≤(\frac{\eps}{2} x)^{-\kappa'}\Eh\Big[\sum_{k≥\sigma_A+1}(k-\sigma_A)^{2\kappa'}\big|\sum_{u\in\Omega(w_{k})}\big(\ty^A(u)- \ty^A(w_{w_{k-1}})\ee{-(V(u)-V(w_{\sigma_A}))}\big)Z_u\big|^{\kappa'}\1{\sigma_A<\tauh_1}\Big]. 
\end{equation}
As for any $k>\sigma_A$, conditionally on $\ty^A(w_{k-1})$ and $\V$, Lemma~\ref{lemma:lawbetaconseq} ensures that the $(\ty^A(u))_{u\in\Omega(w_k)}$ are uncorrelated and have variances $(\ty^A(w_{k-1})(\ee{-\Delta V(u)}+\ee{-2\Delta V(u)}))_{u\in\Omega(w_{k})}$, we have (after applying Jensen's inequality)
\begin{align}
\Eh^\V\Big[ \big|&\sum_{u\in\Omega(w_k)}\big(\ty^A(u)-\ty^A(w_{k-1})\ee{-\Delta V(u)}\big)Z_u\big|^{\kappa'}\Big]\nonumber\\
&≤\Eh^\V\Big[\Eh^\V\Big[\big|\sum_{u\in\Omega(w_k)}\big(\ty^A(u)-\ty^A(w_{k-1})\ee{-\Delta V(u)}\big)Z_u \big|^2\ \Big| \ty^A(w_{k-1}),(Z_u)_u \Big]^{\kappa'/2}\Big]\nonumber\\
&=\Big[\Eh^\V\Big[\ty^A(w_{k-1})\Big(\sum_{u\in\Omega(w_k)}(\ee{-\Delta V(u)}+\ee{-2\Delta V(u)})Z_u^2\Big)^{\kappa'/2}\Big], 
\end{align}
and we can conclude in a similar way as we did for $f_A(x)$. 
\end{proof}

\begin{lemma}\label{lemma:negspineabove}
For all $\eps>0$, there exists an $A_0$ such that for any $A>A_0$ we have for all $x≥1$
\begin{equation*}
\Ph\Big( \sum_{k=\sigma_A+1}^{\tauh_1}\sum_{u\in\Omega(w_k)}\big(\ty(u)-\ty^A(u)\big)W^u_\infty >\eps x,\ \sigma_A<\tauh_1\Big)≤\eps x^{-\kappa'}K_A. 
\end{equation*}
\end{lemma}

\begin{proof}

Just as in the proof of the previous lemma (Lemma~\ref{lemma:final}), we intend to prove a little more than required. Denoting by $(Z_u)_{u\in\Omega}$ (with $\Omega :=\bigcup_{k≥1}\Omega(w_k)$) a family of random variables independent of the walk and the environment on the spine and the siblings of the spine $\big((\ty(w_k),V(w_k))_{k≥0},(\ty(u),V(u))_{u\in\Omega}\big)$ with a finite moment of order $\kappa'$, we want to prove that for all $\eps>0$, there exists an $A_0$ such that for any $A>A_0$, for all $x≥1$,
\begin{equation}\label{eq:wanttoshowneg}
h_A(x):=\Ph\Big( \sum_{k=\sigma_A+1}^{\tauh_1}\sum_{u\in\Omega(w_k)}\big(\ty(u)-\ty^A(u)\big)Z_u >\eps x,\ \sigma_A<\tauh_1\Big)≤\eps x^{-\kappa'}K_A . 
\end{equation}
As explained in the proof of Lemma~\ref{lemma:final}, the properties of $(W^u_\infty)_{u\in\Omega}$ are such that proving~(\ref{eq:wanttoshowneg}) is enough to prove the lemma. 

So let $\eps>0$ be fixed. Recall the definition of $\ty_i^1$ and $\ty_i^2$ from Subsection~\ref{subsec:lawtb}. Notice that on $\{\sigma_A<\tauh_1\}$, if $k\in\brint{\sigma_A+1;\tauh_1}$, then for any $u\in\Omega(w_k)$
\begin{equation*}
\big(\ty(u)-\ty^A(u)\big)=\sum_{i=\sigma_A}^{\tauh_1-1}(\ty^1_i(u)+\ty^2_i(u)). 
\end{equation*}
This yields 
\begin{align*}
\sum_{k=\sigma_A+1}^{\tauh_1}\sum_{u\in\Omega(w_k)}\big(\ty(u)-\ty^A(u)\big)Z_u&=\sum_{i=\sigma_A}^{\tauh_1-1}\Big(\sum_{k=\sigma_A+1}^{\tauh_1}\sum_{u\in\Omega(w_k)}\big(\ty^1_i(u)+\ty^2_i(u)\big)Z_u\Big)\\
&=\sum_{i=\sigma_A}^{\tauh_1-1}\Big(\sum_{k≥i+1}\sum_{u\in\Omega(w_k)}\big(\ty^1_i(u)+\ty^2_i(u)\big)Z_u\Big), 
\end{align*}
where the last equality comes from the fact that $\ty^1_i(w_{k-1})=\ty^2_i(w_{k-1})=0$ for $k≤i$, as the walks launched on $w_i$ are killed when they reach $w_{i-1}$, and $\ty^1_i(w_{k-1})=\ty^2_i(w_{k-1})=0$ for $k≥\tauh_1$, as $\tauh_1$ is the first level such that none of the walk launched below reach it. Now as $\sum_{i=1}^\infty i^{-2}=\pi^2/6<2$, the union bound yields
\begin{align*}
h_A(x)&=\Ph\Big(\sum_{i=\sigma_A}^{\tauh_1-1}\Big(\sum_{k≥i+1}\sum_{u\in\Omega(w_k)}\big(\ty^1_i(u)+\ty^2_i(u)\big)Z_u\Big)>\eps x6\pi^{-2}\sum_{i≥\sigma_A}(i-\sigma_A+1)^{-2}\Big)\\
&≤\sum_{i≥1}\Ph\Big(\1{\sigma_A≤i<\tauh_1}\sum_{k≥i+1}\sum_{u\in\Omega(w_k)}\ty^1_i(u)Z_u>\frac{\eps}{4} x (i-\sigma_A+1)^{-2}\Big)\\
&\hspace{0.1\textwidth}+\Ph\Big(\1{\sigma_A≤i<\tauh_1}\sum_{k≥i+1}\sum_{u\in\Omega(w_k)}\ty^2_i(u)Z_u>\frac{\eps}{4} x (i-\sigma_A+1)^{-2}\Big)
\end{align*}
But $\ty^1_i$, $\ty^1_i$ and $(Z_u)_{u\in\Omega(w_k)}$ (for $k>\sigma_A$) only depend on the walks launched above $w_i$ and on the environment above $w_i$, when the event $\{\sigma_A≤i<\tauh_1\}$ only depends on the walks launched strictly below $w_i$ and on the environment of the spine below $w_i$. Thus $\ty^1_i$ and $\ty^2_i$ and the $(Z_u)_{u\in\Omega(w_k)}$ are independent of the event $\{\sigma_A≤i<\tauh_1\}$, and since all the $\sum_{k≥i+1}\sum_{u\in\Omega(w_k)}\ty^1_i(u)Z_u$ and $\sum_{k≥i+1}\sum_{u\in\Omega(w_k)}\ty^2_i(u)Z_u$ have the distribution of $\sum_{k≥1}\sum_{u\in\Omega(w_k)}\ty^1_0(u)Z_u$ we get (with the change of variable $j=i-\sigma_A$) that 
\begin{equation}\label{eq:decompYs}
h_A(x)≤\sum_{j≥1}2\Ph(1≤j<\tauh_1-\sigma_A)\Ph\Big(\sum_{k≥1}\sum_{u\in\Omega(w_k)}\ty^1_0(u)Z_u>\frac{\eps}{4} x j^{-2}\Big). 
\end{equation}
Notice that now, it would be enough to conclude that there exists a constant $C$ such that for $x≥1$, 
\begin{equation}\label{eq:majsumeW}
\Ph\Big(\sum_{k≥1}\sum_{u\in\Omega(w_k)}\ty^1_0(u)Z_u>x\Big)≤Cx^{-\kappa'}. 
\end{equation}
Indeed, if this inequality was satisfied, equation~(\ref{eq:decompYs}) would yield
\begin{align*}
h_A(x)&≤C(\frac{4}{\eps})^{\kappa'}x^{-\kappa'}\sum_{j=1}^{\infty}2\Ph(1≤j<\tauh_1-\sigma_A)j^{2\kappa'}\\
&=C(\frac{4}{\eps})^{\kappa'}x^{-\kappa'}\Eh[(\tauh_1-\sigma_A)^{2\kappa'+1}\1{\tauh_1≥\sigma_A}]\\
&= C(\frac{4}{\eps})^{\kappa'}x^{-\kappa'}\Eh\big[\Eh_{\ty(w_{\sigma_A})}[(\tauh_1)^{2\kappa'+1}]\1{\tauh_1≥\sigma_A}\big]\\
&≤C(\frac{4}{\eps})^{\kappa'}x^{-\kappa'}\Eh[C_1\ln^{1+2\kappa'}(1+\ty(w_{\sigma_A}))\1{\tauh_1≥\sigma_A}]\\
&≤CC_1 \frac{\ln^{1+2\kappa'}(1+A)}{A^{\kappa'}}(\frac{4}{\eps})^{\kappa'}x^{-\kappa'}\Eh[\big(\ty(w_{\sigma_A})\big)^{\kappa'}\1{\tauh_1≥\sigma_A}], 
\end{align*}
where we used the branching property in the second equality, and equation~(\ref{eq:taulog}) in the last but one inequality. Now using the fact that $\ty^A(w_{\sigma_A})=\ty(w_{\sigma_A})+1≥\frac{1}{2}\ty(w_{\sigma_A})$, this last quantity could be made smaller than $\eps x^{-\kappa'}\Eh[\big(\ty^A(w_{\sigma_A})\big)^{\kappa'}\1{\tauh_1≥\sigma_A}]$ for $A$ large enough, thus proving the lemma. So let us prove that~(\ref{eq:majsumeW}) stands. We have
\begin{align*}
\Ph\Big(\sum_{k≥1}\sum_{u\in\Omega(w_k)}\ty^1_0(u)Z_u>x\Big)≤&\underbrace{\Ph\Big(\sum_{k≥1}\ty^1_0(w_{k-1})\sum_{u\in\Omega(w_k)}\ee{-(V(u)-V(w_{k-1}))}Z_u>\frac{x}{2}\Big)}_{=:f(x)}\nonumber\\
&+\underbrace{\Ph\Big(\sum_{k≥1}\sum_{u\in\Omega(w_k)}|\ty^1_0(u)-\ty^1_0(w_{k-1})\ee{-(V(u)-V(w_{k-1}))}|Z_u>\frac{x}{2}\Big)}_{=:g(x)}. 
\end{align*}
We denoted by $f(x)$ and $g(x)$ these last two quantities. Let us show that they are both smaller than $Cx^{-\kappa'}$ for a certain $C$. 

Let us start with $f(x)$. Recall the definition of $(Z_{w_k})_{k≥0}$ from~(\ref{eq:Zwk}). For any $k≥0$, noticing that $Z_{w_k}=\big(\sum_{u\in\Omega(w_{k+1})}\ee{-(V(u)-V(w_k))}Z_u\big)+\ee{-(V(w_{k+1})-V(w_k))}Z_{w_{k+1}}$, we get
\begin{align*}
\sum_{k≥1}\ty^1_0(w_{k-1})\sum_{u\in\Omega(w_k)}\ee{-(V(u)-V(w_k))}Z_u&=\sum_{k≥1}\ty^1_0(w_k) \big(Z_{w_k}-\ee{-(V(w_{k+1})-V(w_k))}Z_{w_{k+1}} \big)\\
&=\sum_{k≥1}\big(\ty^1_0(w_{k})-\ty^1_0(w_{k-1})\ee{-(V(w_{k})-V(w_{k-1}))} \big)Z_{w_k}, 
\end{align*}
(we recall that $\ty^1_0(w_0)=0$). 
This yields by the union bound (and using the fact that $\sum_{k≥1}k^{-2}=\pi^2/6<2$)
\begin{align*}
f(x)&=\Ph\Big(\sum_{k≥1}\big(\ty^1_0(w_{k})-\ty^1_0(w_{k-1})\ee{-(V(w_{k})-V(w_{k-1}))} \big)Z_{w_k}>x6\pi^{-2}\sum_{k≥1}k^{-2}\Big)\nonumber\\
&≤\sum_{k≥1}\Ph\Big(\big(\ty^1_0(w_{k})-\ty^1_0(w_{k-1})\ee{-(V(w_{k})-V(w_{k-1}))} \big)Z_{w_k}>\frac{1}{2}xk^{-2}\Big)\nonumber\\
\end{align*}
and we can proceed as for $f_A(x)$ in~\eqref{eq:longproof1ineg},~\eqref{eq:longdiff} and~\eqref{eq:longdiff2}, with $A=1$ (so $\ty^A(w_{\sigma_A})=1$). Likewise, the domination of $g(x)$ can be made similarly to that of $g_A(x)$ in~\eqref{eq:majg} and what follows. 
\end{proof}

\begin{lemma}\label{lemma:L1toW}
For all $\eps>0$, there exists an $A_0$ such that for any $A>A_0$ we have for $x$ large enough
\begin{equation*}
\Ph\Big(\big|\sum_{k=\sigma_A+1}^{\tauh_1}\sum_{u\in\Omega(w_k)}\ty(u)W^{u}_\infty-\sum_{k=\sigma_A+1}^{\tauh_1}\sum_{u\in\Omega(w_{k})}L^1_u\big|>\eps x,\ \sigma_A<\tauh_1\Big)≤\eps x^{-\kappa'}K_A. 
\end{equation*}
\end{lemma}
\begin{proof}
Let $\eps>0$. Let $A≥1$ and $x≥1$. First, the union bound yields
\begin{align}\label{eq:firststatement}
\Ph\Big(\big|\sum_{k=\sigma_A+1}^{\tauh_1}\sum_{u\in\Omega(w_k)}\ty(u)W^{u}_\infty-&\sum_{k=\sigma_A+1}^{\tauh_1}\sum_{u\in\Omega(w_{k})}L^1_u\big|>\eps x,\ \sigma_A<\tauh_1\Big)\nonumber\\
≤\Ph\Big(&\sum_{k=\sigma_A+1}^{\tauh_1}\sum_{u\in\Omega(w_k)}\ty^A(u)\big|W^{u}_\infty-\frac{L^1_u}{\ty(u)}\big|>\eps x, \sigma_A<\tauh_1\Big)\nonumber\\
&+\Ph\Big(\sum_{k=\sigma_A+1}^{\tauh_1}\sum_{u\in\Omega(w_k)}\big(\ty(u)-\ty^A(u)\big)\big|W^{u}_\infty-\frac{L^1_u}{\ty(u)}\big|>\eps x, \sigma_A<\tauh_1\Big). 
\end{align}
Now according to Proposition~\ref{prop:Winfini}, we can apply Lemma~\ref{lemma:stochastic} (in the appendix) to the $|W^{u}_\infty-\frac{L^1_u}{\ty(u)}\big|$ for $u\in\Omega$: there exists a positive random variable $Y$ with a finite moment of order $1+\alpha$ (with $\alpha\in(0;\kappa')$) and a decreasing sequence $(a_n)_{n≥1}$ which tends to zero (with $a_0≤1$), such that for any $n$, $a_n Y$ is stochastically greater than $|W_\infty-\frac{L^1}{k}|$ under $\P_k$ for any $k≥n$. Let us set $(Y_u)_{u\in\Omega}$ a sequence of independent random variables of same law as $Y$, independent of the environment of the spine and the walk on the spine. This yields (using also the fact that the $(a_n)_{n≥1}$ are smaller than $1$) that the second quantity of~(\ref{eq:firststatement}) is smaller than
\begin{equation*}
\Ph\Big(\sum_{k=\sigma_A+1}^{\tauh_1}\sum_{u\in\Omega(w_k)}\big(\ty(u)-\ty^A(u)\big)Y_u>\eps x, \sigma_A<\tauh_1\Big), 
\end{equation*}
which according to~(\ref{eq:wanttoshowneg}) can be made smaller than $\eps x^{-\kappa'}\Eh\big[\big(\ty^A(w_{\sigma_A})\big)^{\kappa'}\1{\sigma_A<\tauh_1}\big]$ for $A$ large enough. Therefore, we just have to bound the first quantity of~(\ref{eq:firststatement}) to show the lemma. For the reasons explained above, we can write that
\begin{equation*}
\Ph\Big(\sum_{k=\sigma_A+1}^{\tauh_1}\sum_{u\in\Omega(w_k)}\ty^A(u)\big|W^{u}_\infty-\frac{L^1_u}{\ty(u)}\big|>\eps x, \sigma_A<\tauh_1\Big)≤\Ph\Big(\sum_{k=\sigma_A+1}^{\tauh_1}\sum_{u\in\Omega(w_k)}\ty^A(u)a_{\ty(u)}Y_u>\eps x, \sigma_A<\tauh_1\Big)
\end{equation*}
We will deal with this last quantity until the end of the proof, as it will be more convenient to bound as required. In order to deal with the case where $\#\{u\in\Omega(w_k)\}=\infty$, let us introduce for every $k≥1$, $p≥1$, 
\begin{equation*}
\widehat{\cal{E}}_{\Omega(w_k)}^p:=\{u\in\Omega(w_k), \#\{v\in\T\st \ee{-V(v)}≥\ee{-V(u)}\}<p \}, 
\end{equation*}
the set of the (at most) $p$ vertices of $\Omega(w_k)$ with lowest potential, which is finite (according to the second condition of $(\bf H_c)$).
Let $N≥1$, $p≥1$; discussing for each $u\in\Omega(w_k)$ ( with $k≥\sigma_A+1$) on whether $u\in\widehat{\cal{E}}_{\Omega(w_k)}^p$ or not, and then on whether $\ty(u)>N$ or not, we get by the union bound (and the fact that $(a_n)_{n≥1}≤1$, and $\ty^A≤\ty$)
\begin{align*}
&\Ph\Big(\sum_{k=\sigma_A+1}^{\tauh_1}\sum_{u\in\Omega(w_k)}\ty^A(u)a_{\ty(u)}Y_u>\eps x, \sigma_A<\tauh_1\Big)\\*
&≤\underbrace{\Ph\Big(\sum_{k=\sigma_A+1}^{\tauh_1}\sum_{u\in\widehat{\cal{E}}_{\Omega(w_k)}^p}\ty^A(u)a_N Y_u>\frac{\eps}{3} x, \sigma_A<\tauh_1\Big)}_{=:f_{N,p}(x)}+\underbrace{\Ph\Big(\sum_{k=\sigma_A+1}^{\tauh_1}\sum_{u\in\widehat{\cal{E}}_{\Omega(w_k)}^p}N Y_u>\frac{\eps}{3} x, \sigma_A<\tauh_1\Big)}_{=:g_{N,p}(x)}\\*
&+\underbrace{\Ph\Big(\sum_{k=\sigma_A+1}^{\tauh_1}\sum_{u\notin\widehat{\cal{E}}_{\Omega(w_k)}^p}\ty^A(u)Y_u>\frac{\eps}{3} x, \sigma_A<\tauh_1\Big)}_{=:h_p(x)}
\end{align*}
Therefore, to prove the lemma it suffices to show that there exist $p$ and $N$ such that $f_{N,p}(x)$, $g_{N,p}(x)$ and $h_p(x)$ can each be made smaller than $\eps x^{-\kappa'} \Eh\big[\big(\ty^A(w_{\sigma_A})\big)^{\kappa'}\1{\sigma_A<\tauh_1}\big]$ for $A$ and $x$ large enough. 

$\bullet$ Let us start with $h_p(x)$. We have that
\begin{align*}
h_p(x)≤&\Ph\Big(\big|\sum_{k=\sigma_A+1}^{\tauh_1}\sum_{u\notin\widehat{\cal{E}}_{\Omega(w_k)}^p}\big(\ty^A(u)-\ty^A(w_{\sigma_A})\ee{-(V(u)-V(w_{\sigma_A}))}\big)Y_u\big|>\frac{\eps}{6} x, \sigma_A<\tauh_1\Big)\nonumber\\
&+\Ph\Big(\ty^A(w_{\sigma_A})\sum_{k=\sigma_A+1}^{\tauh_1}\ee{-(V(w_{k-1})-V(w_{\sigma_A}))}\sum_{u\notin\widehat{\cal{E}}_{\Omega(w_k)}^p}\ee{-(V(u)-V(w_{k-1}))}Y_u>\frac{\eps}{6} x, \sigma_A<\tauh_1\Big)
\end{align*}
The first quantity of this last sum can be bounded by 
\begin{equation*}
\Ph\Big(\big|\sum_{k=\sigma_A+1}^{\tauh_1}\sum_{u\in\Omega(w_k)}\big(\ty^A(u)-\ty^A(w_{\sigma_A})\ee{-(V(u)-V(w_{\sigma_A}))}\big)Y_u\big|>\frac{\eps}{6} x, \sigma_A<\tauh_1\Big), 
\end{equation*}
which according to~(\ref{eq:lemmafinal}) can be bounded by $\frac{1}{2}\eps x^{-\kappa'} \Eh\big[\big(\ty^A(w_{\sigma_A})\big)^{\kappa'}\1{\sigma_A<\tauh_1}\big]$ for $A$ large enough (the $(Y_u)_{u\in\Omega}$ satisfying the same conditions as the $(Z_u)_{u\in\Omega}$). For the second quantity of this sum, notice that as for any $k≥1$
\begin{equation*}
\Ebh\Big[ \Big(\sum_{u\in\Omega(w_k)} \ee{-(V(u)-V(w_{k-1}))}Y_u\Big)^{\kappa'} \Big]≤\Ebh[Y_u]^{\kappa'}\Ebh\Big[ \Big(\sum_{\parent{u}=w_k} \ee{-(V(u)-V(w_{k-1}))}\Big)^{\kappa'} \Big]=\Ebh[Y_u]^{\kappa'}\Eb\Big[ \Big(\sum_{|u|=1} \ee{-V(u)}\Big)^{\kappa} \Big],
\end{equation*}
which is finite. We used Jensen's inequality first, and then the branching property and the many-to-one lemma (Lemma~\ref{lemma:many-to-one-environnement}). This allows us to apply Theorem~B of~\cite{kesten} to the random variable~$\sum_{k≥\sigma_A+1}\ee{-(V(w_{k-1})-V(w_{\sigma_A}))}\sum_{u\notin\widehat{\cal{E}}_{\Omega(w_k)}^p}\ee{-(V(u)-V(w_{k-1}))}Y_u$: there exists a constant ${C_p\in(0,\infty)}$ such that 
\begin{equation*}
\Ph\Big( \sum_{k≥\sigma_A+1}\ee{-(V(w_{k-1})-V(w_{\sigma_A}))}\sum_{u\notin\widehat{\cal{E}}_{\Omega(w_k)}^p}\ee{-(V(u)-V(w_{k-1}))}Y_u>x \Big)\ssim{x\to\infty} C_p x^{-\kappa'}, 
\end{equation*}
which implies that there exist positive constants $K_p$, $K_p'$ such that 
\begin{equation*}
\Ph\Big( \sum_{k≥\sigma_A+1}\ee{-(V(w_{k-1})-V(w_{\sigma_A}))}\sum_{u\notin\widehat{\cal{E}}_{\Omega(w_k)}^p}\ee{-(V(u)-V(w_{k-1}))}Y_u>x \Big)≤K_p\1{x≤K_p'}+C_p x^{-\kappa'}. 
\end{equation*}
Consequently, 
\begin{align*}
&\Ph\Big(\ty^A(w_{\sigma_A})\sum_{k=\sigma_A+1}^{\tauh_1}\ee{-(V(w_{k-1})-V(w_{\sigma_A}))}\sum_{u\notin\widehat{\cal{E}}_{\Omega(w_k)}^p}\ee{-(V(u)-V(w_{k-1}))}Y_u>\frac{\eps}{6} x, \sigma_A<\tauh_1\Big)\\
&≤\Eh\Big[ \Pbh\Big(\sum_{k≥\sigma_A+1}\ee{-(V(w_{k-1})-V(w_{\sigma_A}))}\sum_{u\notin\widehat{\cal{E}}_{\Omega(w_k)}^p}\ee{-(V(u)-V(w_{k-1}))}Y_u>\frac{\eps}{6\ty^A(w_{\sigma_A})} x\Big)\1{\sigma_A<\tauh_1} \Big]\\
&≤K_p\Ph\big(\ty^A(w_{\sigma_A})≥\frac{x}{6K_p'},\sigma_A<\tauh_1\big)+C_p \Eh\big[\big(\ty^A(w_{\sigma_A})\big)^{\kappa'}\1{\sigma_A<\tauh_1}\big] x^{-\kappa'}. 
\end{align*}
Now, as according to Corollary~4.3 of~\cite{goldie}, 
\begin{equation*}
C_p≤\frac{1}{\kappa\Eb[\sum_{|u|=1}-V(u)\ee{-\kappa V(u)}]}\Ebh\Big[\big(\sum_{u\notin\widehat{\cal{E}}_{\Omega(w_1)}^p}\ee{-(V(u)-V(w_{k-1}))}Y_u\big)^{\kappa'}\Big]
\end{equation*}
we have (by Lebesgue's dominated convergence theorem) that for $p$ large enough, for any $A>0$, this last quantity is smaller than
\begin{align*}
K_p\Pbh\big(\ty^A&(w_{\sigma_A})≥\frac{x}{6K_p'},\sigma_A<\tauh_1\big)+\frac{\eps}{2} x^{-\kappa'}\Eh\big[\big(\ty^A(w_{\sigma_A})\big)^{\kappa'}\1{\sigma_A<\tauh_1}\big] \\
&≤K_p(6Kp')^{\kappa'}x^{-\kappa'}\Eh\big[\big(\ty^A(w_{\sigma_A})\big)^{\kappa'}\1{\sigma_A<\tauh_1}\1{x≤Kp'\ty^A(w_{\sigma_A})}\big]+\frac{\eps}{2}x^{-\kappa'} \Eh\big[\big(\ty^A(w_{\sigma_A})\big)^{\kappa'}\1{\sigma_A<\tauh_1}\big], 
\end{align*}
by Markov's inequality. This yields that for $x$ large enough, we have
\begin{equation}\label{eq:fixp}
h_p(x)≤\eps x^{-\kappa'}\Eh\big[\big(\ty^A(w_{\sigma_A})\big)^{\kappa'}\1{\sigma_A<\tauh_1}\big]. 
\end{equation}

$\bullet$ Let us now deal with $f_{N,p}(x)$ for $p$ fixed and $x$ large enough such that~(\ref{eq:fixp}) holds. We have 
\begin{align*}
f_{N,p}(x)≤\Ph\Big(\big|\sum_{k=\sigma_A+1}^{\tauh_1}\sum_{u\in\Omega(w_k)}\ty^A(u)Y_u-\ty^A(w_{\sigma_A})Y_{w_{\sigma_A}}\big|>\frac{\eps}{6a_N} x, \sigma_A<\tauh_1\Big)&\\
+\Ph\Big(\ty^A(w_{\sigma_A})Y_{w_{\sigma_A}}>\frac{\eps}{6a_N} x, \sigma_A<\tauh_1\Big)&, 
\end{align*}
where we recall that following the notation introduced in the proofs of Lemmas~\ref{lemma:final} and~\ref{lemma:negspineabove}, we denoted by $Y_{w_{\sigma_A}}$ the random variable $\sum_{k≥\sigma_A+1}\ee{-(V(w_{k-1})-V(w_{\sigma_A}))}\sum_{u\in\Omega(w_{k})}\ee{-(V(u)-V(w_{k-1}))}Y_u$. Now equation~(\ref{eq:lemmafinal}) and Theorem~B of~\cite{kesten} applied to $Y_{w_{\sigma_A}}$ yield that for $A$ large enough
\begin{equation*}
f_{N,p}(x)≤\eps (6a_N)^{\kappa'}x^{-\kappa'}\Eh\big[\big(\ty^A(w_{\sigma_A})\big)^{\kappa'}\1{\sigma_A<\tauh_1}\big]+C_Y (6a_N)^{\kappa'}x^{-\kappa'}\Eh\big[\big(\ty^A(w_{\sigma_A})\big)^{\kappa'}\1{\sigma_A<\tauh_1}\big], 
\end{equation*}
where $C_Y$ is to $Y_{w_{\sigma_A}}$ what $C_Z$ is to $Z_{w_{\sigma_A}}$ in~(\ref{eq:limZu}). Then, as $a_n\to 0$, we have for $N$ large enough
\begin{equation}\label{eq:fixN}
f_{N,p}(x)≤\eps x^{-\kappa'}\Eh\big[\big(\ty^A(w_{\sigma_A})\big)^{\kappa'}\1{\sigma_A<\tauh_1}\big]. 
\end{equation}

$\bullet$ Finally, let us consider $g_{N,p}(x)$ for $p$ and $N$ fixed such that~(\ref{eq:fixp}) and~(\ref{eq:fixN}) hold. We have by Markov's inequality
\begin{equation}\label{eq:majgNp}
g_{N,p}(x)≤(\frac{3}{\eps})x^{-\kappa'}\Eh\Big[\big(\sum_{k=\sigma_A+1}^{\tauh_1}\sum_{u\in\widehat{\cal{E}}_{\Omega(w_k)}^p}N Y_u\big)^{\kappa'}\1{\sigma_A<\tauh_1}\1{\sum_{k=\sigma_A+1}^{\tauh_1}\sum_{u\in\widehat{\cal{E}}_{\Omega(w_k)}^p}N Y_u>\frac{\eps}{3} x}\Big].
\end{equation}
Let us study the expectation appearing in this expression; we have as $\kappa'≤1$ and as ${\#\widehat{\cal{E}}_{\Omega(w_k)}≤p}$, 
\begin{align*}
\Eh\Big[\big(\sum_{k=\sigma_A+1}^{\tauh_1}\sum_{u\in\widehat{\cal{E}}_{\Omega(w_k)}^p}N Y\big)^{\kappa'}\1{\sigma_A<\tauh_1}\Big]&≤\Eh\Big[\sum_{k=\sigma_A+1}^{\tauh_1}\sum_{u\in\widehat{\cal{E}}_{\Omega(w_k)}^p}N^{\kappa'} Y_u^{\kappa'}\1{\sigma_A<\tauh_1}\Big]\\
&≤\Eh\big[(\tauh_1-\sigma_A)\1{\sigma_A<\tauh_1}\big]pN^{\kappa'}\E[Y^{\kappa'}\big]\\
&=\Eh\big[\Eh_{\ty({w_{\sigma_A}})}[\tauh_1]\1{\sigma_A<\tauh_1}\big]pN^{\kappa'}\E[Y^{\kappa'}]\\
&≤\Eh\big[C_1\ln(1+\ty({w_{\sigma_A}}))\1{\sigma_A<\tauh_1}\big]pN^{\kappa'}\E[Y^{\kappa'}]\\
&≤C_1\Eh\big[\big(\ty({w_{\sigma_A}})\big)^{\kappa'}\1{\sigma_A<\tauh_1}\big]pN^{\kappa'}\E[Y^{\kappa'}], 
\end{align*}
where the equality is obtained by the strong Markov property and where the last but one inequality is obtained thanks to~(\ref{eq:taulog}). This together with~(\ref{eq:majgNp}) (and the fact that $\ty^A(w_{\sigma_A})=\ty(w_{\sigma_A})-1≥\frac{1}{2}\ty(w_{\sigma_A})$) yields that for any $A$, for $x$ large enough, 
\begin{equation*}
g_{N,p}(x)≤\eps x^{-\kappa'}\Eh\big[\big(\ty^A(w_{\sigma_A})\big)^{\kappa'}\1{\sigma_A<\tauh_1}\big]
\end{equation*}
as required, thus concluding the proof of the lemma. 

\end{proof}

We can now prove that~(\ref{eq:equivPhL1>n}) stands. 
\begin{proposition}\label{prop:xkappa}
There exists a constant $K\in(0;\infty)$ such that
\begin{equation*}
\Ph\big(L^1>x\big)\sim K x^{-\kappa'}. 
\end{equation*}
The forests $\Fb^X$ and $\Fb^R$ satisfy hypothesis ${\bf (H_{\ell})}(iii)$. 
\end{proposition}
\begin{proof}
Let $\eps>0$. Combining~(\ref{eq:encadreL1}) with Lemma~\ref{lemma:L1toW} allows us to write that for $A$ and $x$ large enough, 
\begin{equation*}
\Ph\big(L^1>x\big)≥\Ph\Big(\sum_{k=\sigma_A+1}^{\tauh_1}\sum_{u\in\Omega(w_k)}\ty(u)W^u_\infty> (1+\eps)x,\ \sigma_A<\tauh_1\Big)-\eps x^{-\kappa'}K_A
\end{equation*}
and 
\begin{equation*}
\Ph\big(L^1>x\big)≤\Ph\Big(\sum_{k=\sigma_A+1}^{\tauh_1}\sum_{u\in\Omega(w_k)}\ty(u)W^u_\infty> (1-2\eps)x,\ \sigma_A<\tauh_1\Big)+\eps x^{-\kappa'}\big(2+K_A\big). 
\end{equation*}
Then, Lemmas~\ref{lemma:negspineabove} and~\ref{lemma:final}, yield that for $A$ and $x$ large enough, 
\begin{equation}\label{eq:minfinal}
\Ph\big(L^1>x\big)≥\Ph\Big(\ty^A(w_{\sigma_A})W^{w_{\sigma_A}}_\infty> (1+3\eps)x,\ \sigma_A<\tauh_1\Big)-3\eps x^{-\kappa'}K_A
\end{equation}
and 
\begin{equation}\label{eq:majfinal}
\Ph\big(L^1>x\big)≤\Ph\Big(\ty^A(w_{\sigma_A})W^{w_{\sigma_A}}_\infty> (1-4\eps)x,\ \sigma_A<\tauh_1\Big)+\eps x^{-\kappa'}\big(2+3K_A\big). 
\end{equation}
Now, Lemma~\ref{lemma:tailW} and Lebesgue's dominated convergence theorem yield that for any $A>0$, we have
\begin{align*}
\lim_{x\to\infty} x^{\kappa'} \Ph\Big(\ty^A(w_{\sigma_A})W^{w_{\sigma_A}}_\infty>x,\sigma_A<\tauh_1\Big)&=\lim_{x\to\infty} x^{\kappa'} \Eh\Big[ \Pbh\Big( W^{w_{\sigma_A}}_\infty>\frac{x}{\ty^A(w_{\sigma_A})} \mid \scr{F}_{\sigma_A} \Big)\1{\sigma_A<\tauh_1} \Big]\\
&=\widehat{C}_\infty \Eh\big[\big(\ty^A(w_{\sigma_A})\big)^{\kappa'}\1{\sigma_A<\tauh_1}\big]\\
&=\widehat{C}_\infty K_A\in(0;\infty). 
\end{align*}
Therefore, equations~\eqref{eq:minfinal} and~\eqref{eq:majfinal} yield for $\eps>0$ and $A>0$ large enough,
\begin{align*}
0<((1+3\eps)^{-\kappa}\widehat{C}_\infty-3\eps)K_A&≤\liminf_{x\to\infty} x^{\kappa}\Ph\big(L^1>x\big)\\
&≤\limsup_{x\to\infty} x^{\kappa}\Ph\big(L^1>x\big)≤((1-4\eps)^{-\kappa}\widehat{C}_\infty+3\eps)K_A+2\eps<\infty
\end{align*}
This ensures the non-triviality of the potential limit in $x$ of $x^{\kappa}\Ph\big(L^1>x\big)$. Taking the $\limsup$ (resp.\ $\liminf$) in $A$ on the left member (resp.\ right member) of this inequality, and then letting $\eps\to 0$ ensures that this limit exists and is equal to $\widehat{C}_\infty\lim_{A\to\infty} K_A$.

As explained at the beginning of Subsection~\ref{subsec:cvqueue1}, $\Ph(L^1>x)\sim Kx^{-\kappa'}$ is equivalent to $\P(L^1>x)\sim\frac{\kappa - 1}{\kappa} Kx^{-\kappa}$, and therefore the reproduction laws of forests $\Fb^X$ and $\Fb^R$ satisfy hypothesis ${\bf (H_{\ell})}(iii)$. 
\end{proof}

\begin{remark}
The constant $K^\star$ for which ${\bf(H_\ell)} (iii)$ is satisfied is thus equal to $\frac{\kappa-1}{\kappa}\widehat{C}_\infty \lim_{A\to\infty} K_A$, but we are unable to compute this last limit. 
\end{remark}

\section{Proof of Theorem~\ref{th:main}}\label{sec:final}
Let us now return to the random walk on $\V$ that we initially considered, $(X^\V_n)_{n≥0}$. We intend to prove Theorem~\ref{th:main} by linking this random walk to another one taking place on a forest as introduced at the beginning of Section~\ref{sec:strategy}. The strategy of the proof is the same as that presented in Section~5 of~\cite{aidekon-raphelis}, but we write it thoroughly for the sake of completeness. 

For every $k≥1$, let 
\begin{equation*}
T_k:=\inf\{n≥1\st \sum_{1≤i≤k}\1{X_{i-1}=\parent{\root},X_i=\root}=k\}. 
\end{equation*}
be the time of the $k$\up{th} visit to $\root$ from its parent. In other words, $T_n$ is the time after which $n$ excursions have been made from $\parent{\root}$, and so the law of $\Tb^{T_n}$ the visited tree at time $T_n$ is the same as that of $\Tb$ under $\P_n$. We extend the notation $\beta$ introduced in Subsection~\ref{subsec:reduction} by setting for every $u\in\T$ and every $k≥1$ 
\begin{equation*}
\beta^{(k)}(u):=\sum_{i=1}^{T_k}\1{X_{i-1}=\parent{u},X_i=u}. 
\end{equation*} 
We have the following lemma, which is the analogue of Lemma~7.1 of~\cite{aidekon-raphelis}: 
\begin{lemma}\label{lemma:Rnlog}
 $\P^*$-almost surely, $T_{n^{1-1/\kappa}\fl{\ln^{10}(n)}}≥n$ for $n≥1$ large enough. 
\end{lemma}
\begin{proof}
For every $k≥1$, let us set $R^{(k)}:=\#\Tb^{T_k}$ as the number of distinct vertices visited by $(X_n)_{n≥0}$ after $k$ excursions from $\parent{\root}$. Observe that $\{T_{n^{1-1/\kappa}\fl{\ln^{10}(n)}}≥n\}\supset\{R^{(n^{1-1/\kappa}\fl{\ln^{10}(n)})}≥n\}$. Therefore, it is enough to show that $\P^*$-almost surely,
\begin{equation*}
R^{(k\fl{\ln^{10}(k)})}≥k^{\frac{\kappa}{\kappa-1}}\quad\text{for $k$ large enough}. 
\end{equation*}
Now notice that the probability $\P^\V\big(R^{(k\fl{\ln^{10}(k)})}<k^{\frac{\kappa}{\kappa-1}}\big)$ is smaller than $\P^\V\big(R^{(1)}<k^{\frac{\kappa}{\kappa-1}}\big)^{k\fl{\ln^{10}(k)}}$. On the event $\big\{ \P^{\V}\big( R^{(1)}≥k^{\frac{\kappa}{\kappa-1}} \big)≥\frac{1}{k\fl{\ln^5(k)}} \big\}$, we get that $\P^\V\big(R^{(k\fl{\ln^{10}(k)})}<k^{\frac{\kappa}{\kappa-1}}\big)≤\ee{-\fl{\ln^5(k)}}$ which is summable in $k$. Therefore, if we are able to show that $\P^*$-almost surely, for $k$ large enough, $\P^{\V}\big( R^{(1)}≥k^{\frac{\kappa}{\kappa-1}}\big)≥\frac{1}{k\fl{\ln^5(n)}}$, then the lemma will be proved. \\

Let us set $G_k:=\{u\in\T\st \ee{V(u)}≥k\ln^2(k)\quad\text{and}\quad \ee{V(v)}<k\ln^2(k)\quad \forall v<u\}$ as the set of vertices the exponential potential of which is the first of their ancestry line to overshoot $k\ln^2(k)$. According to the proof of Lemma~7.1 of~\cite{aidekon-raphelis}, there exist constants $C,c>0$ such that if we set $G'_k:=\{u\in G_k\st \ee{V(u)}<\frac{2}{C}k\ln^2(k)\}$, then $\P^*$-almost surely, there exists $\eps>0$ such that, 
\begin{itemize}
\item for $k$ large enough, $\#G'_k≥c\eps k \ln^2(k)$, and
\item for $k$ large enough, $|u|<\fl{\ln^2(k)}$ for any $u\in G'_k$. 
\end{itemize} 
For all $\eps>0$, $k_0>0$, let $E_{\eps,k_0}:=\bigcup_{k≥k_0}\big\{ \#G'_k≥c\eps k \ln^2(k) \quad\text{and}\quad |u|<\fl{\ln^2(k)} \quad \forall u\in G'_k \big\}$. Let us now consider for any $u\in\T$, $(X_n^{(u)})_{n≥0}$ as the first excursion from $u$ of the random walk $(X_n)_{n≥0}$. Let $\cal{E}_{u,k}$ be the event that the walk $(X_n^{(u)})_{n≥0}$ hits more than $k^{\frac{\kappa}{1-\kappa}}$ distinct vertices before hitting $\parent{u}$. Notice that the quenched probability $\P^\V\big(R^{(1)}≥k^{\frac{\kappa}{\kappa-1}}\big)$ is larger than the probability that there exists a $u\in G'_k$ such that $u$ gets visited before $T_1$ and $\cal{E}_{u,k}$ happens. Moreover, as remarked by A.O.~Golosov~\cite{golosov}, the probability that the walks visits $u$ before visiting $\parent{\root}$ is equal to $(1+\ee{V(u_1)}+\dots+\ee{V(u)})^{-1}$ and when $u\in G'_k$, $(1+\ee{V(u_1)}+\dots+\ee{V(u)})^{-1}≥\big(2/C\times k\fl{\ln^2(k)}\times(|u|+1)\big)^{-1}≥2\big(\fl{\ln(k)}\times k\fl{\ln^2(k)}\times\fl{\ln^2(k)}\big)^{-1}$ for $k$ large enough. This yields that $\P$-almost surely, for $k$ large enough, 
\begin{equation*}
\P^\V\big(R^{(1)}≥k^{\frac{\kappa}{\kappa-1}}\big)≥\frac{2}{k\fl{\ln^5(k)}}\P^\V\big(\bigcup_{u\in G'_k} \cal{E}_{u,k}\big). 
\end{equation*}
and therefore, for $k$ large enough, 
\begin{equation}\label{eq:events}
\Big\{\P^{\V}\big( R^{(1)}≥k^{\frac{\kappa}{\kappa-1}}\big)<\frac{1}{k\fl{\ln^5(n)}} \Big\}\supset \Big\{ \P^\V\big(\bigcup_{u\in G'_k} \cal{E}_{u,k}\big)<\frac{1}{2} \Big\}=\Big\{ \P^\V\big(\bigcap_{u\in G'_k} \cal{E}^c_{u,k}\big)≥\frac{1}{2} \Big\}.
\end{equation} 
Now for any $\eps>0$, $k_0≤k$, we have by independence of the events $\cal{E}_{u,k}$ for $k\in G'_k$,
\begin{equation*}
\P\big( \bigcap_{u\in G'_k} \cal{E}_{u,k}^{c} \big| E_{\eps,k_0} \big)≤\P\big(R^{(1)}<k^{\frac{\kappa}{\kappa-1}}\big)^{c\eps k\ln^2(k)}. 
\end{equation*}
The quantity $R^{(1)}$ is smaller than $\#\{u\in\T^{(1)}\st \beta^1(u)=1\}$, and since the tree made up of the vertices $u$ of type $\beta^{(1)}(u)=1$ is a critical Galton--Watson tree with power-law tail of index $\kappa≤2$, Theorem~1 of~\cite{doney} ensures that there exists a $c'$ such that $\P\big(R^{(1)}<k^{\frac{\kappa}{\kappa-1}}\big)≤1-\frac{c'}{k}$ if $1<\kappa<2$ and $\P\big(R^{(1)}<k^{\frac{\kappa}{\kappa-1}}\big)≤1-\frac{c'}{k}\ln^{\frac{1}{2}}(k)$ if $\kappa=2$. This yields in any case
\begin{equation*}
\P\big( \bigcap_{u\in G'_k} \cal{E}_{u,k}^{c} \big| E_{\eps,k_0} \big)≤\ee{-c'c\eps \ln^{\frac{3}{2}}(k)}, 
\end{equation*}
and then Markov's inequality gives
\begin{equation*}
\P\Big( \P^\V\big(\bigcap_{u\in G'_k} \cal{E}^c_{u,k}\big)≥\frac{1}{2} \big|E_{\eps,k_0}\Big)≤2\ee{-c'c\eps \ln^{3}{2}(k)}, 
\end{equation*}
which is summable in $k$. This yields that for any $\eps>0$, $k_0≥1$, on the event $E_{\eps,k_0}$, $\P^*$-almost surely for $k≥k_0$ large enough $\P^\V\Big(\bigcap_{u\in G'_k} \cal{E}^c_{u,k}\Big)≥\frac{1}{2}$. In view of~(\ref{eq:events}) and by the fact that $\P^*$-almost surely there exist $\eps>0$, $k_0≥1$ such that $E_{\eps,k_0}$ holds, this ensures that $\P^*$-almost surely for $k$ large enough $\P^{\V}\big( R^{(1)}≥k^{\frac{\kappa}{\kappa-1}}\big)≥\frac{1}{k\fl{\ln^5(n)}}$, which is enough to conclude the proof as explained at the beginning. 
\end{proof}

Let us now prove the analogue of Lemma~5.2 of~\cite{aidekon-raphelis}. Recall that for every $u\in\T$, $\ebrint{\root,u}$ denotes the set of strict ancestors of $u$, excluding $\root$. 
\begin{lemma}\label{lemma:maj2exp}
There exists a constant $c>0$ such that for any $\ell,k≥1$, 
\begin{equation*}
\P\big(\exists u\in\T\st |u|≥\ell\quad\text{and}\quad\beta^{(k)}(v)≥2\quad \forall v\in\ebrint{\root;u}\big)≤2k^2\ee{-c\ell}
\end{equation*}
\end{lemma}
\begin{proof}
We just have to follow the lines of Lemma~5.2, let apart that at the end where we write instead of equation~(5.6):
\begin{align*}
\E\Big[ \sum_{|v|=\ell/2}\1{\beta^{(1)}(v)≥1,\beta^{(2)}(v)-\beta^{(1)}(v)≥1} \Big]&=\E\Big[ \sum_{|v|=\ell/2}\P^\V\big( \beta^{(1)}(v)≥1 \big)^2 \Big]\\
&≤\E\Big[ \sum_{|v|=\ell/2}\P^\V\big( \beta^{(1)}(v)≥1 \big)^{\frac{1+\kappa}{2}} \Big], 
\end{align*}
as $\kappa≤2$. Then as explained in the proof of Lemma~\ref{lemma:Rnlog}, $\P^\V\big( \beta^{(1)}(v)≥1 \big)=(1+\ee{V(v_1)}+\dots+\ee{V(v)})^{-1}≤\ee{-V(v)}$, and therefore the last equation is less than $\psi(\frac{1+\kappa}{2})^{\ell/2}$ (with $\psi(\frac{1+\kappa}{2})<1$), which allows us to conclude as in the proof of Lemma~5.2 of~\cite{aidekon-raphelis}. 
\end{proof}

\noindent We are now ready to give the proof of Theorem~\ref{th:main}. \\
\begin{proof}[Proof of Theorem~\ref{th:main}]
Let us set for every $n≥1$, $c_n:=n^{1-1/\kappa}$ if $\kappa\in (1;2)$ and $c_n:=\big(n\ln^{-1}(n)\big)^{1/2}$ if $\kappa=2$. We need to prove that for any $a>0$, 
\begin{equation*}
c_n^{-1}\Big((|X_{\fl{nt}}|)_{0≤t≤a},\mathcal{R}_n\Big)\sto[\Longrightarrow]{n\to\infty}\Big((H_t)_{0≤t≤a},\Tr_{(H_t)_{0≤t≤1}}\Big).
\end{equation*}
It is actually enough to prove this for $a≥1$; so let us now fix $a≥1$. Let for all $n≥1$, $j_n:=(an)^{1-1/\kappa}\fl{\ln^{10}(an)}$. According to Lemma~\ref{lemma:Rnlog}, $\P^*$-almost surely, $T_{j_n}≥an$ for $n$ large enough. Let
\begin{equation*}
\cal{L}_n:=\big\{u\in\T\st \beta^{(j_n)}(u)=1\quad\text{and}\quad \beta^{(j_n)}(v)≥2\quad\forall v<u\ \text{s.t.}\ |v|≥\frac{1}{2}\ln^2(n) \big\}
\end{equation*}
be the set of vertices which at time $T_{j_n}$ are the first of their ancestry line since generation $\frac{1}{2}\ln^2(n)$ to be of edge-local time $1$. Let us also set
\begin{equation*}
\cal{Z}_n:=\big\{u\in\T\st \beta^{(j_n)}(u)\neq 1\quad\text{and}\quad \beta^{(j_n)}(v)≥2\quad\forall v>u, \big\}
\end{equation*}
as the set of the vertices "below" the line $\cal{L}_n$ (see Figure~\ref{f:Ti} for a representation of $\cal{Z}_n$). According to Lemma~\ref{lemma:maj2exp}, $\P$-almost surely, for $n$ large enough, $\max_{u\in\cal{L}_n\cup\cal{Z}_n}|u|<\ln^2(n)$. Moreover, as for any $u\in\T$, $k≥1$, $\E^\V\big[ \beta^{(k)}(u) \big]=k\ee{-V(u)}$, Markov's inequality yields that
\begin{equation*}
\P^\V\Big( \sum_{|u|≤\ln^2(n)} \beta^{(j_n)}(u)≥j_n\ln^{3}(n) \Big)≤\frac{1}{\ln^3(n)}\sum_{k≤\ln^2(n)}\sum_{|u|=k}\ee{-V(u)}
\end{equation*}
which goes to zero $\P$-almost surely by the Kesten-Stigum theorem~\cite{kesten-stigum}. To sum up, for both the convergence under $\P^*$ and under $\P^\V$ for $\P^*$-almost every $\V$, we just have to show that for any $n_0≥1$, the theorem holds on the event
\begin{equation*}
\cal{E}_{n_0}:=\Big\{\forall n≥n_0,\quad T_{j_n}≥an,\quad \max_{u\in\cal{L}_n\cup\cal{Z}_n} |u|<\ln^{2}(n), \quad \text{and}\quad \sum_{|u|≤\ln^2(n)} \beta^{(j_n)}(u)<j_n\ln^{3}(n) \Big\}. 
\end{equation*}

Let us fix an $n_0$ and restrict to the event $\mathcal{E}_{n_0}$. Let $n≥n_0$. 
Let us denote by $\widetilde\root_1,\widetilde\root_2,\dots$ the vertices of $\cal{L}_n$ ordered according to their first time of visit by the walk $(X_k)_{k≥0}$. For every $1≤i≤\#\cal{L}_n$ we denote by $\widetilde{\T}_i$ the sub-tree of $\T$ rooted in $\widetilde\root_i$. See Figure~\ref{f:Ti} for a picture of the situation. We also let for every $i≥\#\cal{L}_n$, $\widetilde{\V}_i=(\widetilde{\T}_i,V)$ be \iid branching random walks with same law as that of $\V$, and we denote by $\widetilde\W=(\widetilde{\F},V)$ the whole forest made up of the trees $(\widetilde\T_i)_{i≥1}$. Let for all $i≥1$, 
\begin{equation*}
\eta^{\text{in}}_i:=\min\{k≥1\st X_k=\widetilde\root_i\}\qquad\text{and}\qquad\eta^{\text{out}}_i:=\min\{k≥\eta^{\text{in}}_i\st X_k=\parent{\widetilde\root_i}\}
\end{equation*}
be the entrance and exit times of the trees $\widetilde{\T}_i$ by the walk. The walk $(X_k)_{k≥0}$ considered until time $an$ will therefore behave as follows: 
\begin{itemize}
\item It remains in the set $\cal{Z}_n$ from time $0$ to $\eta^{\text{in}}_1-1$, 
\item for any $i≥1$, between times $\eta^{\text{in}}_i$ and $\eta^{\text{out}}_i-1$ it makes an excursion in the tree $\widetilde\T_i$, 
\item for any $i≥1$, between times $\eta^{\text{out}}_i$ and $\eta^{\text{in}}_i-1$ it remains in the set $\cal{Z}_n$. 
\end{itemize}

\begin{figure}[H]
\definecolor{zzttqq}{rgb}{0.6,0.2,0}
\definecolor{qqqqcc}{rgb}{0,0,0.8}
\begin{tikzpicture}[line cap=round,line join=round,>=triangle 45,x=0.25cm,y=1.0cm]
\clip(-19.5,-0.75) rectangle (54,11.25);
\fill[line width=0.4pt,dotted,color=zzttqq,fill=zzttqq,fill opacity=0.1] (-5,5.5) -- (-2,7) -- (-3,7.5) -- (-5.5,7.5) -- (-4.5,9) -- (-4.5,10.5) -- (-13,10.5) -- (-13.5,9) -- (-12.5,8.5) -- (-8.5,6.5) -- (-6,6) -- cycle;
\fill[dotted,color=zzttqq,fill=zzttqq,fill opacity=0.1] (-9,4.5) -- (-8.5,5) -- (-8.5,6) -- (-9.5,6.5) -- (-9.5,6) -- (-11,6) -- (-11.5,7) -- (-13,8.5) -- (-13.5,8.5) -- (-16,10.5) -- (-18,10.5) -- (-18,8) -- cycle;
\fill[dotted,color=zzttqq,fill=zzttqq,fill opacity=0.1] (13,5.5) -- (19,7.5) -- (21,8.5) -- (21.5,9) -- (21.5,10.5) -- (10.5,10.5) -- (10.5,8.5) -- (9.5,8) -- (9.5,7) -- (9.5,6.5) -- cycle;
\fill[dotted,color=zzttqq,fill=zzttqq,fill opacity=0.1] (25,5.5) -- (24.5,6) -- (24.5,7) -- (23.5,8) -- (24.5,9) -- (24.5,10.5) -- (30.5,10.5) -- (30.5,8.5) -- (28.5,7.5) -- (28.5,6.5) -- cycle;
\fill[dotted,color=qqqqcc,fill=qqqqcc,fill opacity=0.1] (-17.5,6) -- (-16.5,7) -- (-16.53,7.43) -- (-13.5,6) -- (-13.5,5) -- (-11.5,4) -- (-8,3.5) -- (-4,4.5) -- (-3,4) -- (-7,3) -- (-1.5,2.5) -- (1.5,3) -- (-3.5,5) -- (-3.5,5.5) -- (-2.5,5.5) -- (0.5,4.5) -- (5.5,3.5) -- (10,2.5) -- (11.5,3.5) -- (9.5,5) -- (9.5,5.5) -- (10.5,5.5) -- (12.5,4.5) -- (12.5,2.5) -- (16.5,4) -- (16.5,4.5) -- (21,5.5) -- (21.5,5.5) -- (21.5,4.5) -- (22.5,4.5) -- (22.5,5.5) -- (24,5.5) -- (26.5,4.5) -- (26,3) -- (21.5,2.5) -- (21.5,1.5) -- (10,-0.5) -- (-8,2.5) -- (-12,3.5) -- (-14,4.5) -- cycle;

\draw (10,0)-- (4,1);
\draw (10,0)-- (16,1);
\draw (4,1)-- (10,2);
\draw (4,1)-- (-1,2);
\draw (16,1)-- (12,2);
\draw (16,1)-- (21,2);
\draw (12,2)-- (17,3);
\draw (10,2)-- (5,3);
\draw [dash pattern=on 3pt off 3pt] (10,2)-- (10,3);
\draw (10,2)-- (12,3);
\draw (21,2)-- (21,3);
\draw [dash pattern=on 3pt off 3pt] (-1,2)-- (-5,3);
\draw (-1,2)-- (3,3);
\draw [dash pattern=on 3pt off 3pt] (-1,2)-- (-1,3);
\draw (-1,2)-- (-8,3);
\draw (5,3)-- (0,4);
\draw [dash pattern=on 3pt off 3pt] (5,3)-- (9,4);
\draw [dash pattern=on 3pt off 3pt] (12,3)-- (16,4);
\draw (12,3)-- (12,4);
\draw (17,3)-- (20,4);
\draw (17,3)-- (17,4);
\draw (-8,3)-- (-12,4);
\draw (-8,3)-- (-4,4);
\draw (0,4)-- (-3,5);
\draw [dash pattern=on 3pt off 3pt] (0,4)-- (3,5);
\draw [dash pattern=on 3pt off 3pt] (9,4)-- (6,5);
\draw (12,4)-- (10,5);
\draw [dash pattern=on 3pt off 3pt] (12,4)-- (15,5);
\draw (17,4)-- (21,5);
\draw (21,3)-- (26,4);
\draw (26,4)-- (23,5);
\draw [dash pattern=on 3pt off 3pt] (26,4)-- (26,5);
\draw (-12,4)-- (-9,5);
\draw (-12,4)-- (-14,5);
\draw (-3,5)-- (-5,6);
\draw [dash pattern=on 3pt off 3pt] (-3,5)-- (-1,6);
\draw [dash pattern=on 3pt off 3pt] (3,5)-- (5,6);
\draw [dash pattern=on 3pt off 3pt] (3,5)-- (1,6);
\draw (10,5)-- (13,6);
\draw [dash pattern=on 3pt off 3pt] (10,5)-- (7,6);
\draw [dash pattern=on 3pt off 3pt] (15,5)-- (17,6);
\draw [dash pattern=on 3pt off 3pt] (21,5)-- (19,6);
\draw (21,5)-- (25,6);
\draw [dash pattern=on 3pt off 3pt] (26,5)-- (29,6);
\draw (-14,5)-- (-17,6);
\draw (-14,5)-- (-14,6);
\draw (-9,5)-- (-12,6);
\draw (-9,5)-- (-9,6);
\draw [dash pattern=on 3pt off 3pt] (1,6)-- (-1,7);
\draw [dash pattern=on 3pt off 3pt] (5,6)-- (7,7);
\draw [dash pattern=on 3pt off 3pt] (5,6)-- (3,7);
\draw (13,6)-- (10,7);
\draw (13,6)-- (16,7);
\draw (-5,6)-- (-9,7);
\draw (-5,6)-- (-3,7);
\draw [dash pattern=on 3pt off 3pt] (19,6)-- (23,7);
\draw [dash pattern=on 3pt off 3pt] (25,6)-- (28,7);
\draw (25,6)-- (25,7);
\draw [dash pattern=on 3pt off 3pt] (29,6)-- (32,7);
\draw [dash pattern=on 3pt off 3pt] (29,6)-- (29,7);
\draw (-12,6)-- (-12,7);
\draw (-12,6)-- (-14,7);
\draw (-14,6)-- (-16,7);
\draw [dash pattern=on 3pt off 3pt] (-17,6)-- (-17,7);
\draw [dash pattern=on 3pt off 3pt] (-1,7)-- (-4,8);
\draw [dash pattern=on 3pt off 3pt] (-1,7)-- (3,8);
\draw [dash pattern=on 3pt off 3pt] (7,7)-- (4,8);
\draw [dash pattern=on 3pt off 3pt] (7,7)-- (8,8);
\draw [dash pattern=on 3pt off 3pt] (10,7)-- (10,8);
\draw (10,7)-- (13,8);
\draw (10,7)-- (11,8);
\draw (16,7)-- (14,8);
\draw [dash pattern=on 3pt off 3pt] (16,7)-- (16,8);
\draw (16,7)-- (19,8);
\draw [dash pattern=on 3pt off 3pt] (23,7)-- (21,8);
\draw [dash pattern=on 3pt off 3pt] (23,7)-- (23,8);
\draw (25,7)-- (24,8);
\draw (25,7)-- (28,8);
\draw [dash pattern=on 3pt off 3pt] (29,7)-- (31,8);
\draw [dash pattern=on 3pt off 3pt] (32,7)-- (34,8);
\draw [dash pattern=on 3pt off 3pt] (32,7)-- (32,8);
\draw [dash pattern=on 3pt off 3pt] (-9,7)-- (-11,8);
\draw (-9,7)-- (-6,8);
\draw [dash pattern=on 3pt off 3pt] (-12,7)-- (-13,8);
\draw (-14,7)-- (-17,8);
\draw [dash pattern=on 3pt off 3pt] (-14,7)-- (-14,8);
\draw [dash pattern=on 3pt off 3pt] (-17,7)-- (-19,8);
\draw [dash pattern=on 3pt off 3pt] (-19,8)-- (-19,9);
\draw (-17,8)-- (-15,9);
\draw [dash pattern=on 3pt off 3pt] (-11,8)-- (-13,9);
\draw [dash pattern=on 3pt off 3pt] (-11,8)-- (-9,9);
\draw [dash pattern=on 3pt off 3pt] (-4,8)-- (-4,9);
\draw (-6,8)-- (-8,9);
\draw [dash pattern=on 3pt off 3pt] (-6,8)-- (-5,9);
\draw [dash pattern=on 3pt off 3pt] (3,8)-- (1,9);
\draw [dash pattern=on 3pt off 3pt] (3,8)-- (3,9);
\draw [dash pattern=on 3pt off 3pt] (4,8)-- (7,9);
\draw [dash pattern=on 3pt off 3pt] (8,8)-- (10,9);
\draw [dash pattern=on 3pt off 3pt] (8,8)-- (8,9);
\draw (11,8)-- (11,9);
\draw (11,8)-- (13,9);
\draw [dash pattern=on 3pt off 3pt] (16,8)-- (15,9);
\draw [dash pattern=on 3pt off 3pt] (16,8)-- (17,9);
\draw (19,8)-- (18,9);
\draw (19,8)-- (21,9);
\draw [dash pattern=on 3pt off 3pt] (23,8)-- (22,9);
\draw [dash pattern=on 3pt off 3pt] (24,8)-- (25,9);
\draw [dash pattern=on 3pt off 3pt] (23,8)-- (22,9);
\draw [dash pattern=on 3pt off 3pt] (23,8)-- (23,9);
\draw [dash pattern=on 3pt off 3pt] (23,8)-- (23,9);
\draw [dash pattern=on 3pt off 3pt] (23,8)-- (24,9);
\draw (28,8)-- (26,9);
\draw [dash pattern=on 3pt off 3pt] (28,8)-- (30,9);
\draw [dash pattern=on 3pt off 3pt] (31,8)-- (31,9);
\draw [dash pattern=on 3pt off 3pt] (32,8)-- (34,9);
\draw [dash pattern=on 3pt off 3pt] (32,8)-- (32,9);
\draw [dash pattern=on 3pt off 3pt] (34,8)-- (36,9);
\draw [line width=0.4pt,dotted,color=zzttqq] (-5,5.5)-- (-2,7);
\draw [line width=0.4pt,dotted,color=zzttqq] (-2,7)-- (-3,7.5);
\draw [line width=0.4pt,dotted,color=zzttqq] (-3,7.5)-- (-5.5,7.5);
\draw [line width=0.4pt,dotted,color=zzttqq] (-5.5,7.5)-- (-4.5,9);
\draw [line width=0.4pt,dotted,color=zzttqq] (-4.5,9)-- (-4.5,10.5);
\draw [line width=0.4pt,dotted,color=zzttqq] (-4.5,10.5)-- (-13,10.5);
\draw [line width=0.4pt,dotted,color=zzttqq] (-13,10.5)-- (-13.5,9);
\draw [line width=0.4pt,dotted,color=zzttqq] (-13.5,9)-- (-12.5,8.5);
\draw [line width=0.4pt,dotted,color=zzttqq] (-12.5,8.5)-- (-8.5,6.5);
\draw [line width=0.4pt,dotted,color=zzttqq] (-8.5,6.5)-- (-6,6);
\draw [line width=0.4pt,dotted,color=zzttqq] (-6,6)-- (-5,5.5);
\draw [dotted,color=zzttqq] (-9,4.5)-- (-8.5,5);
\draw [dotted,color=zzttqq] (-8.5,5)-- (-8.5,6);
\draw [dotted,color=zzttqq] (-8.5,6)-- (-9.5,6.5);
\draw [dotted,color=zzttqq] (-9.5,6.5)-- (-9.5,6);
\draw [dotted,color=zzttqq] (-9.5,6)-- (-11,6);
\draw [dotted,color=zzttqq] (-11,6)-- (-11.5,7);
\draw [dotted,color=zzttqq] (-11.5,7)-- (-13,8.5);
\draw [dotted,color=zzttqq] (-13,8.5)-- (-13.5,8.5);
\draw [dotted,color=zzttqq] (-13.5,8.5)-- (-16,10.5);
\draw [dotted,color=zzttqq] (-16,10.5)-- (-18,10.5);
\draw [dotted,color=zzttqq] (-18,10.5)-- (-18,8);
\draw [dotted,color=zzttqq] (-18,8)-- (-9,4.5);
\draw [dotted,color=zzttqq] (13,5.5)-- (19,7.5);
\draw [dotted,color=zzttqq] (19,7.5)-- (21,8.5);
\draw [dotted,color=zzttqq] (21,8.5)-- (21.5,9);
\draw [dotted,color=zzttqq] (21.5,9)-- (21.5,10.5);
\draw [dotted,color=zzttqq] (21.5,10.5)-- (10.5,10.5);
\draw [dotted,color=zzttqq] (10.5,10.5)-- (10.5,8.5);
\draw [dotted,color=zzttqq] (10.5,8.5)-- (9.5,8);
\draw [dotted,color=zzttqq] (9.5,8)-- (9.5,7);
\draw [dotted,color=zzttqq] (9.5,7)-- (9.5,6.5);
\draw [dotted,color=zzttqq] (9.5,6.5)-- (13,5.5);
\draw [dotted,color=zzttqq] (25,5.5)-- (24.5,6);
\draw [dotted,color=zzttqq] (24.5,6)-- (24.5,7);
\draw [dotted,color=zzttqq] (24.5,7)-- (23.5,8);
\draw [dotted,color=zzttqq] (23.5,8)-- (24.5,9);
\draw [dotted,color=zzttqq] (24.5,9)-- (24.5,10.5);
\draw [dotted,color=zzttqq] (24.5,10.5)-- (30.5,10.5);
\draw [dotted,color=zzttqq] (30.5,10.5)-- (30.5,8.5);
\draw [dotted,color=zzttqq] (30.5,8.5)-- (28.5,7.5);
\draw [dotted,color=zzttqq] (28.5,7.5)-- (28.5,6.5);
\draw [dotted,color=zzttqq] (28.5,6.5)-- (25,5.5);

\draw (30,-0.25) node[anchor=east] {$\mathbb{T} \begin{cases} \\ \end{cases}$};

\draw (28,0)-- (29,0);
\draw (29.5,0) node[anchor=west] {Visited edges};

\draw [dash pattern=on 3pt off 3pt] (28,-0.5)-- (29,-0.5);
\draw (29.5,-0.5) node[anchor=west] {Unvisited edges};

\draw [color=red] (28.5,0.5) ++(-2.5pt,0 pt) -- ++(2.5pt,2.5pt)--++(2.5pt,-2.5pt)--++(-2.5pt,-2.5pt)--++(-2.5pt,2.5pt);
\draw (29,0.5) node[anchor=west] {Vertices of local time $1$};

\draw [dotted,color=zzttqq] (27,1.75)-- (29,1.75);
\draw [dotted,color=zzttqq] (29,1.75)-- (29,2.25);
\draw [dotted,color=zzttqq] (29,2.25)-- (27,2.25);
\draw [dotted,color=zzttqq] (27,2.25)-- (27,1.75);
\fill[dotted,color=zzttqq,fill=zzttqq,fill opacity=0.1] (27,1.75) -- (29,1.75) -- (29,2.25) -- (27,2.25) -- cycle;
\draw (29,2) node[anchor=west] {Trees $\widetilde{\mathbb{T}}_i$};

\fill[dotted,color=blue,fill=blue,fill opacity=0.1] (27,1) -- (27,1.5) -- (29,1.5) -- (29,1) -- cycle;
\draw (29,1.25) node[anchor=west] {The set $\mathcal{Z}_n$};

\draw [dash pattern=on 6pt off 6pt,color=green] (35.5,3.5)-- (-16,3.5);
\draw [dash pattern=on 6pt off 6pt,color=green] (35.5,7)-- (-16,7);
\draw [<->] (35,3.5)-- (35,7);
\draw (35.5,3.5) node[right] {$\frac{1}{2}\ln^2(n)$};
\draw (35.5,7) node[right] {$\ln^2(n)$};

\draw [color=zzttqq](27,11) node {$\widetilde{\mathbb{T}}_1$};
\draw [color=zzttqq](15,11) node {$\widetilde{\mathbb{T}}_3$};
\draw [color=zzttqq](-9,11) node {$\widetilde{\mathbb{T}}_2$};
\draw [color=zzttqq](-17,11) node {$\widetilde{\mathbb{T}}_4$};
\draw [fill=black] (10,0) circle (0.5pt);
\draw [fill=black] (4,1) circle (0.5pt);
\draw [fill=black] (16,1) circle (0.5pt);
\draw [color=red] (10,2) ++(-2.5pt,0 pt) -- ++(2.5pt,2.5pt)--++(2.5pt,-2.5pt)--++(-2.5pt,-2.5pt)--++(-2.5pt,2.5pt);
\draw [color=red] (-1,2) ++(-2.5pt,0 pt) -- ++(2.5pt,2.5pt)--++(2.5pt,-2.5pt)--++(-2.5pt,-2.5pt)--++(-2.5pt,2.5pt);
\draw [fill=black] (12,2) circle (0.5pt);
\draw [fill=black] (21,2) circle (0.5pt);
\draw [fill=black] (17,3) circle (0.5pt);
\draw [fill=black] (5,3) circle (0.5pt);
\draw [fill=black] (10,3) circle (0.5pt);
\draw [fill=black] (12,3) circle (0.5pt);
\draw [color=red] (21,3) ++(-2.5pt,0 pt) -- ++(2.5pt,2.5pt)--++(2.5pt,-2.5pt)--++(-2.5pt,-2.5pt)--++(-2.5pt,2.5pt);
\draw [fill=black] (-5,3) circle (0.5pt);
\draw [fill=black] (3,3) circle (0.5pt);
\draw [fill=black] (-1,3) circle (0.5pt);
\draw [fill=black] (-8,3) circle (0.5pt);
\draw [fill=black] (0,4) circle (0.5pt);
\draw [fill=black] (9,4) circle (0.5pt);
\draw [fill=black] (16,4) circle (0.5pt);
\draw [fill=black] (12,4) circle (0.5pt);
\draw [fill=black] (20,4) circle (0.5pt);
\draw [fill=black] (17,4) circle (0.5pt);
\draw [fill=black] (-12,4) circle (0.5pt);
\draw [fill=black] (-4,4) circle (0.5pt);
\draw [fill=black] (-3,5) circle (0.5pt);
\draw [fill=black] (3,5) circle (0.5pt);
\draw [fill=black] (6,5) circle (0.5pt);
\draw [fill=black] (10,5) circle (0.5pt);
\draw [fill=black] (15,5) circle (0.5pt);
\draw [fill=black] (21,5) circle (0.5pt);
\draw [fill=black] (26,4) circle (0.5pt);
\draw [fill=black] (23,5) circle (0.5pt);
\draw [fill=black] (26,5) circle (0.5pt);
\draw [color=red] (-9,5) ++(-2.5pt,0 pt) -- ++(2.5pt,2.5pt)--++(2.5pt,-2.5pt)--++(-2.5pt,-2.5pt)--++(-2.5pt,2.5pt);
\draw [color=red] (-9,5) node[right]{$\widetilde{\root}_4$};
\draw [fill=black] (-14,5) circle (0.5pt);
\draw [color=red] (-5,6) ++(-2.5pt,0 pt) -- ++(2.5pt,2.5pt)--++(2.5pt,-2.5pt)--++(-2.5pt,-2.5pt)--++(-2.5pt,2.5pt);
\draw [color=red] (-5,6) node[right]{$\widetilde{\root}_2$};
\draw [fill=black] (-1,6) circle (0.5pt);
\draw [fill=black] (5,6) circle (0.5pt);
\draw [fill=black] (1,6) circle (0.5pt);
\draw [color=red] (13,6) ++(-2.5pt,0 pt) -- ++(2.5pt,2.5pt)--++(2.5pt,-2.5pt)--++(-2.5pt,-2.5pt)--++(-2.5pt,2.5pt);
\draw [color=red] (13,6) node[right]{$\widetilde{\root}_3$};
\draw [fill=black] (7,6) circle (0.5pt);
\draw [fill=black] (17,6) circle (0.5pt);
\draw [fill=black] (19,6) circle (0.5pt);
\draw [color=red] (25,6) ++(-2.5pt,0 pt) -- ++(2.5pt,2.5pt)--++(2.5pt,-2.5pt)--++(-2.5pt,-2.5pt)--++(-2.5pt,2.5pt);
\draw [color=red] (25,6) node[right]{$\widetilde{\root}_1$};
\draw [fill=black] (29,6) circle (0.5pt);
\draw [fill=black] (-17,6) circle (0.5pt);
\draw [fill=black] (-14,6) circle (0.5pt);
\draw [fill=black] (-12,6) circle (0.5pt);
\draw [fill=black] (-9,6) circle (0.5pt);
\draw [fill=black] (-1,7) circle (0.5pt);
\draw [fill=black] (7,7) circle (0.5pt);
\draw [fill=black] (3,7) circle (0.5pt);
\draw [fill=black] (10,7) circle (0.5pt);
\draw [fill=black] (16,7) circle (0.5pt);
\draw [fill=black] (-9,7) circle (0.5pt);
\draw [fill=black] (-3,7) circle (0.5pt);
\draw [fill=black] (23,7) circle (0.5pt);
\draw [fill=black] (28,7) circle (0.5pt);
\draw [fill=black] (25,7) circle (0.5pt);
\draw [fill=black] (32,7) circle (0.5pt);
\draw [fill=black] (29,7) circle (0.5pt);
\draw [fill=black] (-12,7) circle (0.5pt);
\draw [fill=black] (-14,7) circle (0.5pt);
\draw [fill=black] (-16,7) circle (0.5pt);
\draw [fill=black] (-17,7) circle (0.5pt);
\draw [fill=black] (-4,8) circle (0.5pt);
\draw [fill=black] (3,8) circle (0.5pt);
\draw [fill=black] (4,8) circle (0.5pt);
\draw [fill=black] (8,8) circle (0.5pt);
\draw [fill=black] (10,8) circle (0.5pt);
\draw [fill=black] (13,8) circle (0.5pt);
\draw [fill=black] (11,8) circle (0.5pt);
\draw [fill=black] (14,8) circle (0.5pt);
\draw [fill=black] (16,8) circle (0.5pt);
\draw [fill=black] (19,8) circle (0.5pt);
\draw [fill=black] (21,8) circle (0.5pt);
\draw [fill=black] (23,8) circle (0.5pt);
\draw [fill=black] (24,8) circle (0.5pt);
\draw [fill=black] (28,8) circle (0.5pt);
\draw [fill=black] (31,8) circle (0.5pt);
\draw [fill=black] (34,8) circle (0.5pt);
\draw [fill=black] (32,8) circle (0.5pt);
\draw [fill=black] (-11,8) circle (0.5pt);
\draw [fill=black] (-6,8) circle (0.5pt);
\draw [fill=black] (-13,8) circle (0.5pt);
\draw [fill=black] (-17,8) circle (0.5pt);
\draw [fill=black] (-14,8) circle (0.5pt);
\draw [fill=black] (-19,8) circle (0.5pt);
\draw [fill=black] (-19,9) circle (0.5pt);
\draw [fill=black] (-15,9) circle (0.5pt);
\draw [fill=black] (-13,9) circle (0.5pt);
\draw [fill=black] (-9,9) circle (0.5pt);
\draw [fill=black] (-4,9) circle (0.5pt);
\draw [fill=black] (-8,9) circle (0.5pt);
\draw [fill=black] (-5,9) circle (0.5pt);
\draw [fill=black] (1,9) circle (0.5pt);
\draw [fill=black] (3,9) circle (0.5pt);
\draw [fill=black] (7,9) circle (0.5pt);
\draw [fill=black] (10,9) circle (0.5pt);
\draw [fill=black] (8,9) circle (0.5pt);
\draw [fill=black] (11,9) circle (0.5pt);
\draw [fill=black] (13,9) circle (0.5pt);
\draw [fill=black] (15,9) circle (0.5pt);
\draw [color=red] (15,9) ++(-2.5pt,0 pt) -- ++(2.5pt,2.5pt)--++(2.5pt,-2.5pt)--++(-2.5pt,-2.5pt)--++(-2.5pt,2.5pt);
\draw [color=red] (28,8) ++(-2.5pt,0 pt) -- ++(2.5pt,2.5pt)--++(2.5pt,-2.5pt)--++(-2.5pt,-2.5pt)--++(-2.5pt,2.5pt);
\draw [color=red] (-6,8) ++(-2.5pt,0 pt) -- ++(2.5pt,2.5pt)--++(2.5pt,-2.5pt)--++(-2.5pt,-2.5pt)--++(-2.5pt,2.5pt);
\draw [fill=black] (17,9) circle (0.5pt);
\draw [fill=black] (18,9) circle (0.5pt);
\draw [fill=black] (21,9) circle (0.5pt);
\draw [fill=black] (22,9) circle (0.5pt);
\draw [fill=black] (25,9) circle (0.5pt);
\draw [fill=black] (23,9) circle (0.5pt);
\draw [fill=black] (24,9) circle (0.5pt);
\draw [fill=black] (26,9) circle (0.5pt);
\draw [fill=black] (30,9) circle (0.5pt);
\draw [fill=black] (31,9) circle (0.5pt);
\draw [fill=black] (34,9) circle (0.5pt);
\draw [fill=black] (32,9) circle (0.5pt);
\draw [fill=black] (36,9) circle (0.5pt);
\end{tikzpicture} 
\caption{ The tree $\T$ (represented up to generation 9), the set $\mathcal{Z}_n$ and the trees $\widetilde{\T}_i$. }
\label{f:Ti}
\end{figure}
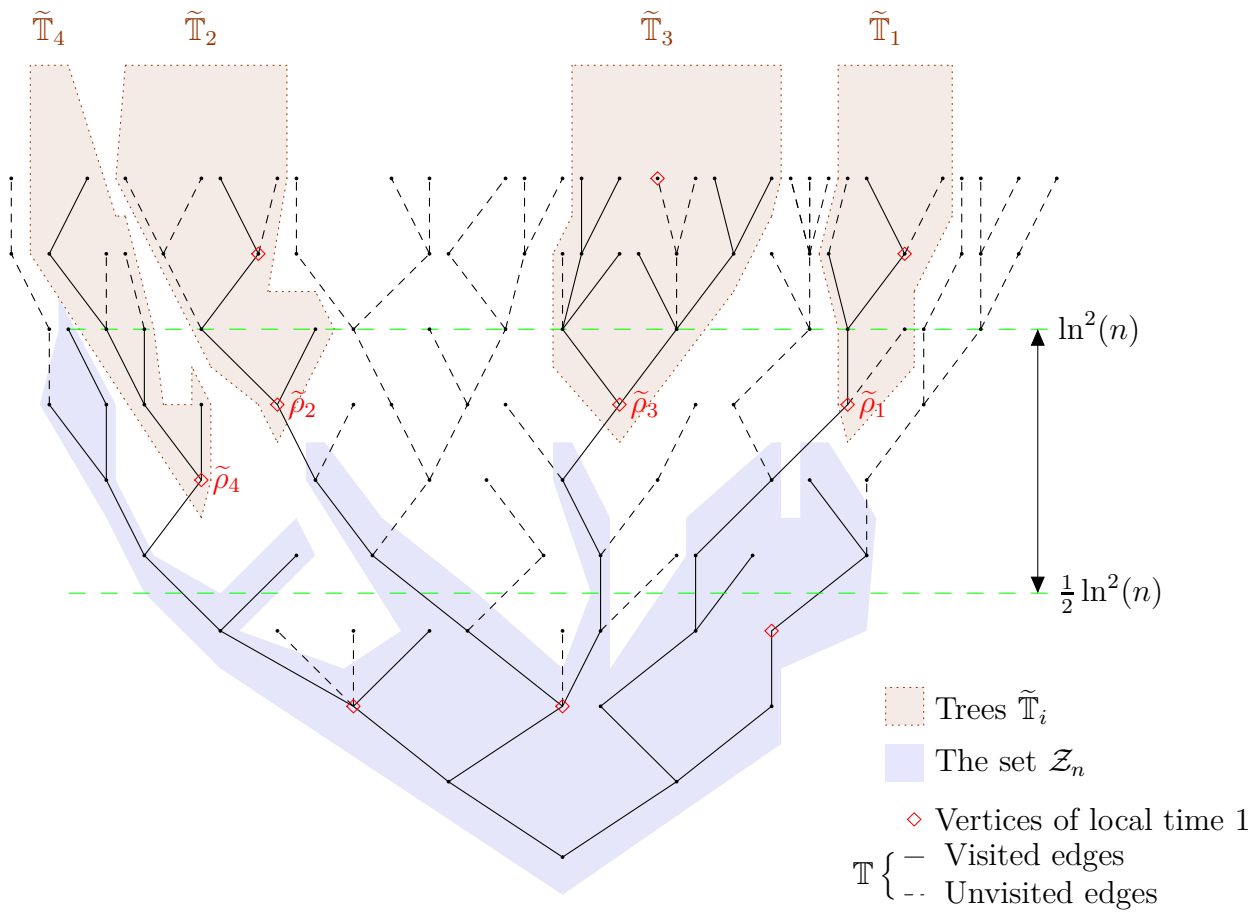

We denote by $(\widetilde X_k)_{0≤k≤\sum_{i=1}^{\#\cal{L}_n}(\eta^{\text{out}}_i-\eta^{\text{in}}_i)}$ the random walk made up of the excursions of the walk on the trees $(\widetilde{\T}_i)_{1≤i≤\#\cal{L}_n}$ (the roots of which are considered to be at height $0$), that we extend for convenience as $(\widetilde X_k)_{k≥0}$ by adding excursions on the trees $(\widetilde\T_i)_{i≥\#\cal{L}_n+1}$. The random walk $(\widetilde X_k)_{k≥0}$ is therefore a random walk on a Galton--Watson forest in random environment, as introduced at the beginning of Section~\ref{sec:strategy}. \\
For all $k≤an$, we set $f_n(k):=\sum_{1≤i≤k}\1{X_i\notin \cal{Z}_n}$ as the time spent by the walk in the forest $\widetilde{\F}$ before time $k$. On the event $\cal{E}_{n_0}$, uniformly in $k≤an$, 
\begin{equation}\label{eq:nntilde}
0≤k-f_n(k)≤n-f_n(n)≤2j_n\ln^{3}(n)=o(n),  
\end{equation} 
Let $\widetilde{\mathcal{R}}_{f_n(n)}$ be the trace of $(\widetilde X_k)_{k≥0}$ up to time $f_n(n)$. Proposition~\ref{prop:forest} together with~(\ref{eq:nntilde}) (which also yields $c_{f_n(n)}\sim c_n$) says
\begin{equation}\label{eq:finaltoprove}
c_{n}^{-1}\Big((|\widetilde{X}_{{\fl{f_n(n t)}}|})_{0≤t≤a},\widetilde{\mathcal{R}}_{f_n(n)}\Big)\sto[\Longrightarrow]{n\to\infty}\Big((H_t)_{0≤t≤a},\Tr_{(H_t)_{0≤t≤1}}\Big)
\end{equation}
under the annealed law $\P^*$. 

We can use the same arguments as those used in Section~5 of~\cite{aidekon-raphelis} to extend~\eqref{eq:finaltoprove} result to the quenched law. The only difference is in equation~(5.9),where we can bound the probability to touch a vertex at generation $\frac{1}{2}\ln^{2}(n)$ by $\max_{|u|=\frac{1}{2}\ln^2(n)}\ee{-V(u)}$. Then, the upper bound of~(5.9) of~\cite{aidekon-raphelis} can be replaced by $nj_n\Eb\big[\max_{|u|=\frac{1}{2}\ln^2(n)}\ee{-V(u)}\big]$. Now we write
\begin{align*}
nj_n\Eb\big[\max_{|u|=\frac{1}{2}\ln^2(n)}\ee{-V(u)}\big]&≤nj_n\Eb\Big[\Big(\sum_{|u|=\frac{1}{2}\ln^2(n)}\ee{-(\frac{1+\kappa}{2})V(u)}\Big)^{\frac{2}{1+\kappa}}\Big]\\
&≤nj_n\Eb\Big[\sum_{|u|=\frac{1}{2}\ln^2(n)}\ee{-(\frac{1+\kappa}{2})V(u)}\Big]^{\frac{2}{1+\kappa}}=nj_n\psi(\frac{1+\kappa}{2})^{\frac{1}{1+\kappa}\ln^2(n)}, 
\end{align*}
which is summable as $\psi(\frac{1+\kappa}{2})\in(0;1)$. \\
Now we just notice that by construction, for any $k≤an$, on the event $\mathcal{E}_{n_0}$,
\begin{equation}\label{eq:finalth1}
\Big||X_k|-|\widetilde X_{f_n(k)}|\Big|≤\max_{u\in\cal{Z}_n\cup\cal{L}_n}|u|≤\ln^2(n). 
\end{equation}
Moreover, Lemma~2.3 of~\cite{evans-pitman-winter} says that if $(R,d)$ and $(R',d')$ are two real trees, and if $\phi:R\mapsto R'$ is a surjective map which sends the root of $R$ to that of $R'$ (we mention that the roots of a forest are considered to be at a null distance from each other for the graph distance), then $d_{GH}(R,R')$ the Gromov-Hausdorff distance between $R$ and $R'$ is smaller than $\frac{1}{2}\sup \{|d(x,y)-d'(\phi(x),\phi(y))|,(x,y)\in R\}$. Using this with the map $\phi_1:\mathcal{R}_n \mapsto \widetilde{\mathcal{R}}_{f_n(n)}$ which is the identity on the trees $\widetilde\T_i$ and which sends all the vertices of $\mathcal{Z}_n$ to $\widetilde{\root}_1$, we get that 
\begin{equation*}
d_{GH}(\mathcal{R}_n,\widetilde{\mathcal{R}}_{f_n(n)})≤\ln^2(n)=o(c_n).
\end{equation*}
This and~(\ref{eq:finalth1}) used together with~(\ref{eq:finaltoprove}) concludes the proof. 
\end{proof}

\newpage
\appendix
\setcounter{equation}{0}
\renewcommand{\theequation}{A.\arabic{equation}}
\section*{Appendix}

\subsection*{Proof of Theorem~\ref{th:bitype}}
We recall that we are in the setting of~\cite{raphelis}: let $(F,\tyl,\ell)$ be a leafed Galton--Watson forest of reproduction law $\zeta$, recall that for every vertex $u\in F$ we let $\nu(u)$ (resp.\ $\nu^1(u)$) be the total number of children (resp.\ number of children of type $1$) of $u$ in $F$, and recall that we denote by $F^1$ the forest $F$ limited to its vertices of type $1$. 
We recall the hypotheses we make on $\zeta$: 
\begin{equation*}
{\bf (H_{\ell})}\begin{cases}\parbox{\textwidth}{
\begin{itemize}
\item[$(i)$] $\E\big[ \sum_{|u|=1,\tyl(u)=1} 1 \big]=\E\big[ \nu^1 \big]=1$,
\item[$(ii)$] There exists $\eps>0$ s.t.\ $\E\big[ \big(\sum_{|u|=1} 1\big)^{1+\eps} \big]=\E\big[ (\nu)^{1+\eps} \big]<\infty$,
\item[$(iii)$] There exists $C_0>0$ s.t.\ $\P\big( \sum_{|u|=1,\tyl(u)=1}1>x \big)=\P(\nu^1>x)\ssim{x\to\infty} C_0 x^{-\kappa}$, 
\item[$(iv)$] There exists $r>1$ s.t.\ $\E\Big[ \sum_{|u|=1} r^{\ell(u)} \Big]<\infty$. 
\end{itemize}}\end{cases}
\end{equation*}
Recall also $m:=\E\big[\sum_{|u|=1}1\big]=\E[\nu]$, $\mu:=\E\big[\sum_{|u|=1,\tyl(u)=1} \ell(u)\big]$, and for each vertex $u\in F$, $h(u):=\sum_{k=1}^{|u|} \ell(u_k)$. We define $H^1$ the height process of $F^1$ and $H^{\ell}$ the weighted depth-first exploration process of $F$ as follows:
\begin{equation*}
\forall n\in\N,\;H^1(n):=|u^1(n)|\;\text{ and }\;H^{\ell}(n):=h(u(n)), 
\end{equation*}
where $u(n)$ (resp.\ $u^1(n)$) stands for the $n$\up{th} vertex of $F$ (resp.\ $F^1$) taken in the lexicographical order. Recall Theorem~\ref{th:bitype}. 
\addtocounter{theorem}{-1}
\begin{theorem}
Let $(F,\tyl,\ell)$ be a leafed Galton--Watson forest with edge lengths, with offspring distribution $\zeta$ satisfying hypothesis $\bf (H_{\ell})$. The following convergence in law holds for the Skorokhod topology on the space $\mathbb{D}(\R_+,\R)$ of càdlàg functions: 
\begin{align*}
&\frac{1}{n^{1-1/\kappa}}\left(\Big(H^{\ell}(\fl{ns})\Big)_{s≥0},(H^{1}(\fl{ns}))_{s≥0}\right) \substack{ \Longrightarrow \\ n\to\infty} \frac{1}{(C_0|\Gamma(1-\kappa)|)^{\frac{1}{\kappa}}}\left(\Big(\mu H_{m^{-1}s}\Big)_{s≥0},(H_s)_{s≥0}\right)&&\text{if $1<\kappa<2$}, \\
&or\\
&\frac{1}{(n\ln^{-1}(n))^{\frac{1}{2}}}\left(\Big(H^{\ell}(\fl{ns})\Big)_{s≥0},(H^{1}(\fl{ns}))_{s≥0}\right) \substack{ \Longrightarrow \\ n\to\infty} \frac{1}{(C_0)^{\frac{1}{2}}}\left(\Big(\mu |B_{m^{-1}s}|\Big)_{s≥0},(|B_s|)_{s≥0}\right)&&\text{if $\kappa=2$}, 
\end{align*}
where $H$ is the continuous-time height process of a spectrally positive Lévy process $Y$ of Laplace transform $\E[\exp(-\lambda Y_t)]=\exp(t\lambda^\kappa)$ for $\lambda,t≥0$, $(B_t)_{t≥0}$ is a standard Brownian motion, and where $C_0$ is the constant introduced in ${\bf (H_{\ell})}(iii)$. 
\end{theorem}

As we said in Subsection~\ref{subsec:cvleafed}, the proof of this theorem will follow that of Theorem~1 in~\cite{raphelis}. First, Subsection~2.1 of~\cite{raphelis} remains valid under our hypotheses: we can consider a measure $\widehat{\P}$ on $(F,\tyl,\ell,(w_k)_{k≥0})$ where $(w_k)_{k≥0}$ is a marked ray of $\T$ (which is not to be considered under $\P$ then) such that the following \textit{many-to-one lemma} stands. 
\begin{lemma}\label{lemma:many-to-one}
The random variables $(\ell(w_k))_{k≥0}$ are \iid under $\Ph$ and such that for all $n\in\N$, if $g:\R^{n+1}\to\R_+$ is a measurable function, \\
\begin{equation*}
\E\Big[\sum_{|u|=n,u\in T^1}g\big(\ell(u_0),\dots,\ell(u_{n-1}),\ell(u)\big)\Big]=\Eh\Big[g\big(\ell(w_0),\dots,\ell(w_{n-1}),\ell(w_n)\big)\Big]. 
\end{equation*}
\end{lemma}
\noindent The next lemma is an adaptation of Lemma~2 of~\cite{raphelis}. 
\begin{lemma}\label{lemma:estimatesGW}
Let $\Gamma^1_n:=u^1(n)_0$ be the index of the tree in $F^1$ to which the n\up{th} vertex of $F^1$ belongs. Then under $\bf (H_{\ell})$, for all $\eps>0$, there exist $M,M'>0$ such that for all sufficiently large $n\in\N$,
\begin{equation*}
\P\Big(\Gamma^1_n>Mn^{1/\kappa}\text{ or }\max_{0≤i≤n} |u(i)|>M'n^{1-1/\kappa} \Big)<\eps. 
\end{equation*}
\end{lemma} 
The proof of this lemma is similar to that of Lemma~2 in~\cite{raphelis}. We give it for the sake of completeness. 
\begin{proof}
According to Corollary~2.5.1 of~\cite{duquesne-le-gall}, 
\begin{equation*}
\P\Big( \frac{\Gamma^1_n}{n^{1/\kappa}}>M \Big)\sto{n\to\infty} \P\Big(L^0_1>M \Big)<\frac{\eps}{3}
\end{equation*}
for $M$ large enough, where $L_1^0$ is the local time at level $0$ at time 1 of a strictly stable spectrally positive Lévy process $Y$. Moreover, 
\begin{equation*}
\P\Big( \frac{\max_{0≤i≤n}{|u(i)|}}{n^{1-1/\kappa}}>M' \Big)\sto{n\to\infty} \P\Big(\max_{0≤s≤1}|H_s|>M' \Big)<\frac{\eps}{3} 
\end{equation*}
for $M'$ large enough, where $H$ is the continuous-time height process of $Y$. The union bound concludes the proof. 
\end{proof}

Let for $i\in\N$, $\varphi(i)$ be the index of $u(i)$ in $\F^1$ if $\tyl(u(i))=1$, or of its parent in $\F^1$ if $\tyl(u(i))=0$; that is
\begin{equation*}
\varphi(i):=\begin{cases}
k\text{, where } u^1(k)=u(i) & \text{ if } \tyl(u(i))=1 \\
k\text{, where } u^1(k)=\parent{u(i)} & \text{ if } \tyl(u(i))=0.
\end{cases}
\end{equation*}
Let us now explain how to extend Proposition~2 of~\cite{raphelis} to the following proposition:
\begin{proposition}\label{prop:vertical}
Let $(F,\tyl,\ell)$ be a leafed Galton--Watson forest with edge lengths satisfying hypothesis $\bf (H_{\ell})$. Then, for all $\eps>0$, 
\begin{equation*}
\P\Big(\max_{1≤ i ≤ n}\Big|H^{\ell}(i)-\mu H^1(\varphi(i))\Big|>\eps n^{1-1/\kappa}\Big)\sto{n\to\infty} 0. 
\end{equation*} 
\end{proposition}
\begin{proof}
Let us follow the lines of the proof of Proposition~2 in~\cite{raphelis}. As explained at the beginning of that proof, it suffices to show that
\begin{equation}\label{eq:vertical}
\P\Big(\max_{1≤ i ≤ n}\Big|h(u^1(i))-\mu H^1(i)\Big|>\eps n^{1-1/\kappa}\Big)\sto{n\to\infty} 0. 
\end{equation}
Now ${\bf(H_{\ell})} {\it(iv)}$ ensures that there exists $c>0$ such that $\Ph(\ell(w_1)>c\ln(n))\sto[=]{n\to\infty}o(\frac{1}{n})$. Similarly to how~(2.4) is proved in~\cite{raphelis}, we get that
\begin{equation*}
\P\Big( \exists i≤ n \st \ell(u^1(i))>c\ln(n) \Big)\;\sto{n\to\infty}\;0. 
\end{equation*}
Letting for every $n≥1$ and every $u\in F^1$, $\ell^{(n)}(u):=\ell(u)\1{\ell(u)<c\ln(n)}$, showing~(\ref{eq:vertical}) boils down to showing that 
\begin{equation*}
\P\Big(\exists\, i≤ n,\; \Big|\sum_{u\vdash u^1(i)} \ell^{(n)}(u) -\mu H^1(i) \Big|>\eps n^{1-1/\kappa}\Big)\sto{n\to\infty} 0.
\end{equation*}
Now Lemma~\ref{lemma:estimatesGW} allows us to get an equivalent of~(2.8) of~\cite{raphelis}, and using the same arguments that follow it, we get that it suffices to show that
\begin{equation}\label{eq:sumiid}
\Ph\Big( \Big|\sum_{i=1}^{k}\Big( \ell^{(n)}(w_i)- \Eh\Big[ \ell^{(n)}(w_1) \Big]\Big)\Big|>\frac{\eps}{2} n^{1-1/\kappa} \Big)=o(\frac{1}{n}) 
\end{equation}
uniformly in $k≤M'n^{1-1/\kappa}$. The many-to-one lemma (Lemma~\ref{lemma:many-to-one}) allows us to re-write ${\bf(H_{\ell})} (iv)$ as $\Eh[r^{\ell(w_1)}]<\infty$. The $\ell^{(n)}(w_i)- \Eh\Big[ \ell^{(n)}(w_1) \Big]$ being \iid centred random variables with some finite exponential moments, the Cramer-Chernoff theorem yields~(\ref{eq:sumiid}), which concludes the proof. 
\end{proof}
Let us now prove the equivalent of Proposition~3 of~\cite{raphelis} in our case. 

\begin{proposition}\label{prop:horizontal}
Recall that $m=\E[\nu]$. Under $\bf (H_{\ell})$, the function $(\varphi(\fl{ns})/n)_{s>0}$ converges in probability to $(m^{-1}s)_{s>0}$ as $n$ tends to infinity, for the topology of uniform convergence over compact sets. 
\end{proposition}
\begin{proof}
Let us follow the lines of the proof of Proposition~3 of ~\cite{raphelis}. As explained, we just have to prove that $R_n:=\sum_{k=0}^{n-1}\#\{ u\in\F \st \parent{u}=u^1(k), u^1(n)\prec u \}$ is such that $\P(R(n)>\eps n)\sto{n\to\infty}0$ for any $\eps>0$. Let $c_n:=n^{1/\kappa}\ln(n)$. Hypothesis ${\bf(H_{\ell})} (iii)$ implies that $P(\nu^1>c_n)=o(\frac{1}{n})$. Combining this with Lemma~\ref{lemma:estimatesGW} and Markov's inequality yields 
\begin{align*}
\P\Big( R(n)>\eps n \Big)≤\eps' + \frac{1}{\eps n}\underbrace{\E\Big[ \Big(\sum_{u\vdash u^1(n)}\nu(u)\Big)\1{\max_{i<n}\nu^1(u^1(i))<c_n,\:|u^1(n)|≤ M\fl{n^{1-1/\kappa}}} \Big]}_{=:A_n}, 
\end{align*}
and it is sufficient to show that $A_n$ is $o(n)$ to conclude the proof. Let us introduce the \textit{Lukasiewicz path} of $F^1$, defined by $S_0:=0$ and for all $k≥1$, $S_k:=\sum_{i=0}^{k-1} (\nu^1(u^1(i))-1)$. Let us also set 
\begin{equation*}
\tau_1:=\inf \{k≥1 \st S_k>0 \}\text{ and }\forall i\in\N,\: \tau_{i+1}=\inf \{k>\tau_i \st S_k>\max_{\ell<k} S_\ell \} 
\end{equation*}
as the stopping times at which record high are achieved. Following the lines of the proof of Proposition~3 of ~\cite{raphelis} we get that
\begin{align*}
A_n=\E\Big[ \Big(\sum_{k≥1} \nu(u^1(\tau_k))\1{\tau_k≤n}\Big)\1{\max_{i≤n}\nu^1(u^1(i))<c_n,\tau_{\cl{Mn^{1-1/\kappa}}}≥n} \Big]. 
\end{align*}
Applying the strong Markov property to stopping times $\tau_1,\dots,\tau_{\fl{Mn^{1-1\kappa}}}$, we obtain
\begin{equation*}
A_n ≤Mn^{1-1/\kappa}\E\Big[ \nu(u^1(\tau_1))\1{\nu^1(u^1(\tau_1))<c_n}\Big]. 
\end{equation*}
Using Jensen's inequality and then the same lines as in~\cite{raphelis} for $\E\Big[ \nu(u^1(\tau_1))\1{\nu^1(u^1(\tau_1))<c_n}\Big]$, we get for $\eps>0$ small enough
\begin{align*}
\E\Big[ \nu(u^1(\tau_1))\1{\nu^1(u^1(\tau_1))<c_n}\Big]&≤(\E\Big[ (\nu(u^1(\tau_1)))^{1+\eps}\1{\nu^1(u^1(\tau_1))<c_n}\Big])^{\frac{1}{1+\eps}} \\
&≤ m_\eps^{\frac{1}{1+\eps}}(\E\Big[ \sum_{k=0}^{\tau_1-1} \1{S_k>-c_n+1} \Big])^{\frac{1}{1+\eps}}, 
\end{align*}
where $m_\eps:=\E[(\nu)^{1+\eps}]<\infty$ according to ${\bf(H_{\ell})} (iii)$. Now the renewal theorem (p.360 of~\cite{feller}) ensures that the expectation in the last line is smaller than $c'c_n$ (where $c'$ is an adequate constant) and this yields
\begin{equation*}
A_n≤(Mn^{1-1/\kappa})m_\eps^{\frac{1}{1+\eps}}(c'c_n)^{\frac{1}{1+\eps}}\sto[=]{n\to\infty} o(n), 
\end{equation*}
as required, thus concluding the proof. 
\end{proof}

\begin{proof}[Proof of Theorem~\ref{th:bitype}.] The proof of the convergence of $n^{-(1-1/\kappa)}(H^\ell(\fl{ns}))$ is now similar to that in Section~2.4 of~\cite{raphelis}, let apart that we use Theorems 2.3.1 and 2.3.2 of~\cite{duquesne-le-gall} and Theorem~3 of Section XVII.5 of~\cite{feller}  in the case of the infinite variance: 
\begin{align*}
\Big(\frac{1}{n^{1-\frac{1}{\kappa}}}H^1(\fl{ns}) \Big)_{s≥0}\sto[\Longrightarrow]{n\to\infty}\left(\frac{1}{(C_0|\Gamma(1-\kappa)|)^{\frac{1}{\kappa}}} H_s\right)_{s≥0}\qquad&\text{if $1<\kappa<2$} \\
\Big(\frac{1}{(n\ln^{-1}(n))^{\frac{1}{2}}}H^1(\fl{ns}) \Big)_{s≥0}\sto[\Longrightarrow]{n\to\infty}\left(\frac{1}{(2C_0)^{\frac{1}{2}}} H_s\right)_{s≥0}\qquad&\text{if $\kappa=2$} \\
\end{align*}
where $H$ is the continuous-time height process of a spectrally positive Lévy process of Laplace transform $\exp(t\lambda^\kappa)$, and where $C_0$ is the constant introduced in ${\bf (H_{\ell})}(iii)$. If $\kappa=2$, then $H=\sqrt{2}|B|$ where $B$ is a standard Brownian motion. We refer to~\cite{janson} for an explicit computation of the constants. 
\end{proof}

\subsection*{Technical lemmas}

This subsection contains some technical lemmas which are useful in several proofs. 
\begin{lemma}\label{lemma:applyapunov}
Recall from~\eqref{eq:pij} that $\ph_{i,j}=\binom{i+j-1}{i}\Eb\Big[ \sum_{|u|=1} \frac{\ee{-jV(u)}}{(1+\ee{-V(u)})^{i+j}} \Big]$ are the transition probabilities of the Markov chain $(\beta(w_k))_{k≥0}$. Let $\alpha\in(0;\kappa-1)$, and let for all $i≥1$, $F(i):=\frac{\Gamma(i+1+\alpha)}{\Gamma(i+1)}$. There exists $d\in(0;1)$ such that for any $i$ large enough, 
\begin{equation*}
\sum_{j≥1}\widehat{p}_{i,j}F(j)≤dF(i), 
\end{equation*}
i.e. $F$ is a \textit{Lyapounov function} for the Markov chain $(\beta(w_k))_{k≥0}$. 
\end{lemma}
\begin{proof}
We have, for any $i≥1$, 
\begin{align}\label{eq:sumpijvj}
\sum_{j≥1} \ph_{i,j} F(j)&=\sum_{j≥1} \binom{i+j-1}{i}\E\Big[\sum_{|u|=1} \frac{\ee{-jV(u)}}{(1+\ee{-V(u)})^{i+j}} \Big]F(j) \nonumber\\
&=\E\Big[\sum_{|u|=1} \frac{1}{(1+\ee{-V(u)})^{i+1}}\times\ee{-V(u)}\times \frac{1}{i!}\sum_{j≥0}\frac{\Gamma(j+1+\alpha)}{\Gamma(j+1)}\frac{(j+i)!}{j!} \Big(\frac{\ee{-V(u)}}{1+\ee{-V(u)}}\Big)^{j} \Big]. 
\end{align}
Now for all $x\in(0;1)$, setting $y=1-x$, we have
\begin{align}\label{eq:sumGamma}
\sum_{j≥0} \frac{\Gamma(j+1+\alpha)}{\Gamma(j+1)}&\frac{(j+i)!}{j!} x^j= \frac{\d^i}{\d x^i}\Big(x^i\frac{\Gamma(1+\alpha)}{(1-x)^{1+\alpha}}\Big)\nonumber\\
&=\Gamma(1+\alpha) (-1)^i \frac{\d^i}{\d y^i}\Big( \frac{(1-y)^i}{y^{1+\alpha}} \Big)\nonumber\\
&=\Gamma(1+\alpha) (-1)^i \frac{\d^i}{\d y^i}\Big( \sum_{k=0}^{i} \binom{i}{k}(-1)^k y^{k-1-\alpha} \Big)\nonumber\\
&=\Gamma(1+\alpha) \sum_{k=0}^{i} \binom{i}{k}(-1)^i (k-1-\alpha)\times\dots\times (k-i-\alpha)(-1)^k y^{k-i-1-\alpha} \nonumber\\
&=\frac{\Gamma(1+\alpha)}{y^{i+1+\alpha}} \sum_{k=0}^{i}\binom{i}{k} (-1)^i(k-1-\alpha)\times\dots\times (k-i-\alpha)(-1)^k y^k \nonumber\\
&=\frac{i!}{y^{i+1+\alpha}} \sum_{k=0}^{i} \Gamma(1+\alpha)\frac{(i-k+\alpha)\times\dots\times(\alpha+1)}{(i-k)!}\times\frac{\alpha\times\dots\times (\alpha-(k-1))}{k!}(-y)^k. 
\end{align}
Denoting the last sum by $S$, notice that all its terms are negative for $k>1$, and thus we can write
\begin{equation*}
S≤\sum_{k=0}^{\fl{\ln(i)}} \frac{\Gamma(i-k+1+\alpha)}{(i-k)!}\times\frac{\alpha\times\dots\times (\alpha-(k-1))}{k!}(-y)^k. 
\end{equation*}
For $k≤\fl{\ln(i)}$, noticing that $\frac{\Gamma(i-k+1+\alpha)}{(i-k)!}≥\frac{\Gamma(i-\fl{\ln(i)}+1+\alpha)}{(i-\fl{\ln(i)})!}\sto[=]{i\to\infty}\frac{\Gamma(i+1+\alpha)}{i!}+o\left(\frac{\Gamma(i+1+\alpha)}{i!}\right)$, we get uniformly in $y\in(0;1)$,
\begin{equation*}
S\sto[≤]{i\to\infty}\Big(\frac{\Gamma(i+1+\alpha)}{i!}+o(\frac{\Gamma(i+1+\alpha)}{i!})\Big)\sum_{k=\fl{\alpha}+2}^{\fl{\ln(i)}}\frac{\alpha\times\dots\times (\alpha-(k-1))}{k!}(-y)^k. 
\end{equation*}
Plugging this inequality into~(\ref{eq:sumGamma}) yields
\begin{equation*}
\sum_{j≥0} \frac{\Gamma(j+1+\alpha)}{\Gamma(j+1)}\frac{(j+i)!}{j!} x^j\sto[≤]{i\to\infty}\frac{i!}{y^{i+1+\alpha}}\Big(\frac{\Gamma(i+1+\alpha)}{i!}+o(\frac{\Gamma(i+1+\alpha)}{i!})\Big)\Big((1-y)^\alpha-\eps^i(1-y)\Big),
\end{equation*}
where $\eps^i(y)$ is a sequence of positive functions of $y$ simply decreasing to $0$ and smaller than $(1-y)^\alpha$. Plugging this latter inequality into~(\ref{eq:sumpijvj}) yields
\begin{align*}
\sum_{j≥1} \ph_{i,j} F(j)&≤\E\Big[\sum_{|u|=1} \frac{1}{(1+\ee{-V(u)})^{i+1}}\times\ee{-V(u)}\times \Big(\frac{\Gamma(i+1+\alpha)}{i!}+o(\frac{\Gamma(i+1+\alpha)}{i!})\Big)\times(1+\ee{-V(u)})^{i+1+\alpha}\\
&\hspace{0.5\textwidth}\times\Big((\frac{\ee{-V(u)}}{1+\ee{-V(u)}})^\alpha-{\eps^i}(\frac{\ee{-V(u)}}{1+\ee{-V(u)}})\Big) \Big]\\
&\sto[=]{i\to\infty}\E\Big[ \sum_{|u|=1} \ee{-(1+\alpha)V(u)} \Big]\Big(\frac{\Gamma(i+1+\alpha)}{\Gamma(i+1)}+o(\frac{\Gamma(i+1+\alpha)}{\Gamma(i+1)})\Big),
\end{align*}
where between these two lines we used the monotone convergence theorem on $\Big((\frac{\ee{-V(u)}}{1+\ee{-V(u)}})^\alpha-{\eps^i}(\frac{\ee{-V(u)}}{1+\ee{-V(u)}})\Big)$.
Now since $1<1+\alpha<\kappa$, $\E\Big[ \sum_{|u|=1} \ee{-(1+\alpha)V(u)} \Big]<1$ so there exists $d<1$ such that for all $i>i_0$ large enough, 
\begin{equation*}
\sum_{j≥1} \ph_{i,j} F(j)< d F(i), 
\end{equation*}
which is what we wanted. 
\end{proof}

\begin{lemma}\label{lemma:NegBin}
Let $n\in\N$ and $p\in(0;1)$. Let $X^n\sim NB(n,p)$ be a negative binomial random variable of parameter $(n,p)$; we have for any $0<\alpha<1$, 
\begin{equation*}
 \E[(X^n)^{1+\alpha}]≤16n\big(\frac{p}{1-p}+(\frac{p}{1-p})^{1+\alpha}\big)+2n^{1+\alpha} (\frac{p}{1-p})^{1+\alpha}. 
\end{equation*}
\end{lemma}
\begin{proof}
Recall that $\E[X^n]=n\frac{p}{1-p}$ and ${\bf Var}(X^1)=\frac{p}{(1-p)^2}$. \\
\begin{itemize}
\item If $p≤1/2$, then ${\bf Var}(X^1)=\frac{p}{(1-p)^2}≤2\frac{p}{1-p}$. Therefore, 
\begin{equation*}
\E[|X^1-\frac{p}{1-p}|^{1+\alpha}]≤\E[|X^1-\frac{p}{1-p}|]+\E[(X^1-\frac{p}{1-p})^2]≤4\frac{p}{1-p}. 
\end{equation*}
\item On the other hand, if $p>1/2$, then ${\bf Var}(X^1)=\frac{p}{(1-p)^2}≤2\frac{p^2}{(1-p)^2}$, and then by Jensen's inequality, 
\begin{equation*}
\E[|X^1-\frac{p}{1-p}|^{1+\alpha}]≤\E[(X^1-\frac{p}{1-p})^2]^{\frac{2}{1+\alpha}}≤2(\frac{p}{1-p})^{1+\alpha}. 
\end{equation*}
\end{itemize}
So in general, $\E[|X^1-\frac{p}{1-p}|^{1+\alpha}]≤4(\frac{p}{1-p}+(\frac{p}{1-p})^{1+\alpha})$. Finally, according to B.~von~Bahr and C.-G.~Esseen~\cite{von-bahr-esseen}
\begin{equation*}
 \E[|X^n-\frac{np}{1-p}|^{1+\alpha}]≤2n\E[|X^1-\frac{p}{1-p}|^{1+\alpha}]≤8n(\frac{p}{1-p}+(\frac{p}{1-p})^{1+\alpha}). 
\end{equation*}
The inequality $(x+y)^{1+\alpha}≤2(x^{1+\alpha}+y^{1+\alpha})$ for $x,y≥0$ concludes the proof. 
\end{proof}

\begin{lemma}\label{lemma:stochastic}
Let $p>0$, and let $(X_n)_{n≥1}$ be a sequence of non-negative random variables which converges in mean of order $p$. Then for any $r\in(0;p)$, there exists a random variable $Y$ with a finite moment of order $r$ and a decreasing sequence $(a_n)_{n≥1}$ (with $a_1=1$) such that for any $n≥1$, $a_n Y$ is stochastically greater than $X_k$ for any $k≥n$. 
\end{lemma}
\begin{proof}
Let $r\in(0;p)$, and $s\in(r;p)$; let us consider $Y$ a random variable with distribution function defined for all $x>0$ by
\begin{equation*}
\P\big(Y≤x\big)=\Big(1-\frac{\max_{i≥1}\E\big[(X_i)^{s}\big]}{x^s}\Big)^+, 
\end{equation*} 
which is continuously increasing from $0$ to $1$ indeed, and defines a random variable with a finite moment of order $r$ as $s>r$. Let for all $n≥1$,
\begin{equation*}
a_n:=\Big(\frac{\max_{i≥n}\E\big[ (X_i)^s \big]}{\max_{i≥1}\E\big[ (X_i)^s \big]}\Big)^{\frac{1}{s}}
\end{equation*} 
(we do not consider the trivial case $X_n=0$ \as for all $n≥1$). As $(X_n)_{n≥1}$ converges in mean of order $p>s$, this sequence tends to zero as required. Now Markov's inequality together with the fact that distribution functions are positive yield that for any $k≥n$, 
\begin{equation*}
\P\big(X_k>x\big) ≤\Big(\frac{\max_{i≥n}(\E\big[(X_i)^{s}\big])}{x^s}\Big)\wedge 1=\Big(\frac{\max_{i≥1}(\E\big[(X_i)^{s}\big])}{(a_n)^{-s} x^s}\Big)\wedge 1=\P\big( a_nY>x \big), 
\end{equation*}
which proves that $a_n Y$ is stochastically greater than $X_k$ for any $k≥n$ indeed. 
\end{proof}

\subsection*{Notation}

\noindent\textit{The environment (Section~\ref{sec:intro})}
\renewcommand{\labelitemi}{} 
\begin{itemize}
\item $N$ : the point process
\item $\T$ : the genealogical tree
\item $(V(u))_{u\in\T}$ : the branching random walk generated by $N$
\item $\Delta V(u)$ : the increment $V(u)-V(\parent{u})$
\item $\root$ : the root of $\T$
\item $|u|$ : the generation of the vertex $u$
\item $\parent{u}$ : the parent of the vertex $u$
\item $c(u)$ : the set of children of the vertex $u$
\item $\T_u$ : the subtree of $\T$ rooted in $u$
\item $\Omega(u)$ : the set of strict siblings of $u$ (vertices having the same parent, excluding $u$)
\item $\psi(t)$ : Laplace transform of the point process $N$
\item $\kappa$ : characteristic quantity of $\psi$ defined as $\inf\{t>1\st \psi(t)>1\}$
\item $\W$ : Collection of \iid branching random walks $(\T_i,(V(u))_{u\in\T_i})_{i≥1}$
\item $\F$ : The genealogical forest of $\W$
\item $\root_i$ : the root of $\T_i$ the $i$\up{th} tree composing $\F$
\end{itemize}

\noindent\textit{The random walk (Section~\ref{sec:intro})}
\begin{itemize}
\item $(X_n)_{n≥0}$ (also noted $(X_n^\V)_{n≥0}$ or $(X_n^\W)_{n≥0}$) : The random walk on $\V$ or $\W$
\item $p_{u,v}$ : probability transitions of $(X_n)_{n≥0}$, introduced in~\eqref{eq:probatrans}
\item $\P^\V$ : the law of $(X_n)_{n≥0}$ conditionally on $\V$ (\textit{quenched} law)
\item $\P$ : law $\P^\V$ averaged on $\Pb$ (\textit{annealed} law)
\item $\P^*$ : law $\P^\V$ averaged on $\Pb^*$
\item $C^\star$ : the renormalising constant of the random walk $(X_n)_{n≥0}$
\item $\mathcal{T}_g$ : the real tree coded by the càdlàg function $g$
\end{itemize}

\noindent\textit{The trace and the process of local times}
\begin{itemize}
\item $\Tb^n$ (resp.\ $\Fb^n$) : the set of vertices of $\T$ (resp.\ $\F$) visited by $(X_n)_{n≥0}$ before time $n$
\item $\mathcal{R}_n$ : the trace, that is the graph $\Tb^n$ (or $\Fb^n$) seen as a metric space when equipped by the graph distance
\item $R_n$ : $\#\Fb^n$ the range of the walk 
\item $\Fb$ : the set of vertices of $\F$ visited by $(X_n^\W)_{n≥0}$
\item $\Tb$ : generic random variable with same law as the trees composing $\Fb$
\item $\beta(u)$ : the edge local time of $(X_n^\W)_{n≥0}$ in $u$ 
\item $\mathcal{F}_n$ : the sigma-algebra generated by $(u,\beta(u))_{u\in\Tb,|u|≤n}$
\item $\mathcal{B}_u^1$ (resp.\ $\mathcal{B}^1$) : the set of vertices descending from $u\in\Fb$ (resp.\ $\root$) having no ancestor of type $1$ since $u$ (resp.\ $\root$). See Definition~\ref{def:optline}
\item $\mathcal{L}_u^1$ (resp.\ $\mathcal{L}^1$) : the set of vertices descending from $u\in\Fb$ (resp.\ $\root$) being the first of type $1$ in their ancestry line since $u$ (resp.\ $\root$). See Definition~\ref{def:optline}
\item $B^1_u$ and $L^1_u$ (resp.\ $B^1$ and $L^1$) : $\#\mathcal{B}^1_u$ and $\#\mathcal{L}^1_u$ (resp.\ $\#\mathcal{B}^1$ and $\#\mathcal{L}^1$) 
\item $u<\mathcal{L}^1$ : $u\in\mathcal{B}^1$ and $u\notin\mathcal{L}^1$
\item $\bzeta$ : reproduction law of $\Tb$ when seen as a multitype Galton--Watson tree
\item $\P_i$ : law of $\Tb$ as a multitype-Galton--Watson tree with initial type $i$
\item $m_{i,j}$ : coefficient of the mean matrix of $\Tb$ (defined in~\eqref{eq:defmij})
\item $(a_i)_{i≥1}$ (resp.\ $(b_i)_{i≥1}$) : defined in~\eqref{eq:defaibi}, left (resp.\ right) eigenvector of the matrix $(m_{i,j})_{i,j≥1}$
\item $Z_n$ : multitype additive martingale of $(\Tb,\beta)$, defined in~\eqref{eq:defZn}
\item $\Ph_i$ : biased law on $(\Tb,\beta,(w_k)_{k≥0})$ marked tree with spine
\item $\widehat{\bzeta}$ : biased offspring distribution of $(\Tb,\beta,(w_k)_{k≥0})$
\item $(\pi_i)_{i≥1}$ : invariant measure of the Markov chain $(\beta(w_k))_{k≥0}$
\item $\hat{p}_{i,j}$ : transition probabilities of the Markov chain $(\beta(w_k))_{k≥0}$
\item $\tauh_1$ : hitting time of the state $1$ by the Markov chain $(\beta(w_k))_{k≥0}$
\item $K^\star$ : defined in~\eqref{eq:queueL1}
\item $\sigma_A$ : hitting time of the set $\{A,A+1,A+2,\dots\}$ by the Markov chain $(\beta(w_k))_{k≥0}$
\item $\beta^{A}(u)$ : defined in~\eqref{eq:defbetaA}, sum of the local times of the walks launched below $w_{\sigma_A}$
\item $K_A$ : $\Eh\big[\big(\ty^A(w_{\sigma_A})\big)^{\kappa'}\1{\sigma_A<\tauh_1}\big]$, studied in Lemma~\ref{lemma:KA}
\end{itemize}

\noindent\textit{Leafed Galton--Watson forests with edge lengths (Section~\ref{sec:strategy})}
\begin{itemize}
\item $T$ (resp.\ $F$) : the genealogical tree (resp.\ forest)
\item $\zeta$ : offpring ditribution of $(T,e,\ell)$ as a leafed GW tree with edge lengths
\item $e(u)$ : type of the vertex $u$ (can be $0$ or $1$)
\item $\ell(u)$ : length of the edge linking $\parent{u}$ to $u$
\item $\nu$ : law of the count of the total progeny of $\zeta$ (counts both vertices of type $0$ and $1$)
\item $\nu^1$ : law of the count of vertices of type $1$ given by the law $\zeta$
\item $F^1$ : subforest of $F$ only made up of the vertices of type $1$
\item $m$ : $\E[\nu]$
\item $\mu$ : $\E[\sum_{|u|=1,e(u)=1}\ell(u)]$
\item $u(n)$ (resp.\ $u^1(n)$) : $n$\up{th} vertex of $F$ (resp.\ of $F^1$) for the lexicographical order
\item $H^\ell$ : \textit{weighted} height process of $F$
\item $H^1$ : (\textit{non-weighted}) height process of $F^1$
\item $\Fb^R$ : defined in Subsection~\ref{subsec:FR}
\item $\Fb^X$ : defined in Subsection~\ref{subsec:FX}
\item $\Fb^{R^1}$ (resp.\ $\Fb^{X^1}$) : subforest of $\Fb^R$ (resp.\ $\Fb^X$) made up of the vertices of type $1$ (both are equal up to re-ordering, by construction)
\item $\nu^R$ (resp.\ $\nu^X$) : law of the total count of the offspring of a vertex of type $1$ of $\Fb^R$ (resp.\ of $\Fb^X$)
\item $H_X^\ell$ (resp.\ $H_R^\ell$) : weighted height process of $\Fb^R$ (resp. $\Fb^X$)
\item $\mu_R,m_R$ (resp.\ $\mu_X,m_X$) : constants $\mu$ and $m$ as defined above for $F$ associated with the reproduction law of $\Fb^R$ (resp. $\Fb^X$)
\item $C_0$ : introduced in $\bf (H_\ell)$
\end{itemize}

\noindent\textit{Change of measure on the environment}
\begin{itemize}
\item $\Pb$ : The measure on the environment $(V(u),u\in\T)$
\item $\Pb^*$ : The measure $\Pb(\cdot|\#\T=\infty)$
\item $\Pbh$ : The biased measure on the environment with spine $(V(u),u\in\T,(\tilde{w}_k)_{k≥0})$, defined in Subsection~\ref{subsec:lawtb}
\item $\mathcal{G}_k$ : the sigma-algebra generated by the environment up to generation $k$, defined in~\eqref{eq:defGk}
\item $W_k=\sum_{|u|=k}\ee{-V(u)}$ : the additive martingale for the environment
\item $W_\infty$ : the a.s. limit of the positive martingale $(W_k)_{k≥0}$
\item $W_\infty^{u}$ : the a.s. limit of the martingale $(\sum_{v≥u,|v|=n})_{n≥|u|}\ee{-(V(v)-V(u))}$
\item $\widehat{N}$ : the biased point process
\item $(\tilde{w}_k)_{k≥0}$ : the spine of $\T$ built under $\Pbh$ 
\item $\widehat{S}_k$ : introduced in Subsection~~\ref{subsec:multitype}. Has same law as $(V(\tilde{w}_k))_{k≥0}$ under $\Pbh$
\item $\widehat{C}_\infty$ : coefficient of regular variation of the tail of $W_\infty$, defined in Lemma~\ref{lemma:tailW}
\end{itemize}

\noindent\textit{The launched random walks (Subsection~\ref{subsec:lawtb})}
\begin{itemize}
\item $(X_n^{1,\tilde{w}_i})_{n≥0}$ and $(X_n^{2,\tilde{w}_i})_{n≥0}$ : random walks on the marked tree with spine $(\T, (\tilde{w}_k)_{k≥0})$, launched on $w_i$ and die when hitting $w_{i-1}$
\item $\widetilde{\beta}_i^1(u)$ (resp.\ $\widetilde{\beta}_i^2(u)$) : local time of $(X_n^{1,\tilde{w}_i})_{n≥0}$ (resp. $(X_n^{1,\tilde{w}_i})_{n≥0}$) in $u$
\item $\widetilde{\beta}^1(u)$ (resp.\ $\widetilde{\beta}^2(u)$) : sum of the local times of all the walks $(X_n^{1,\tilde{w}_i})_{n≥0}$ (resp.\ $(X_n^{2,\tilde{w}_i})_{n≥0}$) for $i≥0$
\item $\widetilde{\beta}(u)$ : $\widetilde{\beta}^1(u)+\widetilde{\beta}^2(u)-\1{u\in(w_k)_{k≥0}}$. Has actually same law as $\beta$ (see Proposition~\ref{prop:lawbetatilde})
\item $\widetilde{T}$ : set of vertices visited by the collection of walks $(X_n^{1,\tilde{w}_i})_{n≥0}$ and $(X_n^{2,\tilde{w}_i})_{n≥0}$ for $i≥0$. Has actually same law as $\Tb$ under $\Ph$
\end{itemize}

\bigskip

\bigskip

\newpage
\noindent {\bf Acknowledgements:} I thank my advisor Elie Aïdékon for guiding me throughout the development of this article. I thank Xinxin Chen for her advice on the organization of the proofs and for spotting mistakes in an earlier version. I also thank an anonymous referee for his/her numerous precise comments which greatly improved the quality of this paper.

\end{document}